\documentclass[reqno, 11pt, a4paper]{amsart}

\usepackage{amsfonts, amsthm, amsmath, amssymb, multicol}
\usepackage{hyperref}
\hypersetup{colorlinks=false}
\usepackage{booktabs}
\usepackage[font=small, margin=40pt,figureposition=below]{caption}
\usepackage{graphicx}
\usepackage{tikz-cd}

\usepackage[colorinlistoftodos]{todonotes}

\usepackage[margin=3.5cm]{geometry}

\numberwithin{equation}{section}

\renewcommand{\d}{\mathrm{d}}
\renewcommand{\phi}{\varphi}
\renewcommand{\rho}{\varrho}

\newcommand{\0}{\mathbf{0}}
\newcommand{\PP}{\mathbb{P}}
\newcommand{\AAA}{\mathbb{A}}
\newcommand{\A}{\mathbf{A}}
\newcommand{\FF}{\mathbb{F}}
\newcommand{\ZZ}{\mathbb{Z}}

\newcommand{\NN}{\mathbb{N}}
\newcommand{\QQ}{\mathbb{Q}}
\newcommand{\QQbar}{{\overline{\QQ}}}
\newcommand{\RR}{\mathbb{R}}
\newcommand{\CC}{\mathbb{C}}

\newcommand{\cA}{\mathcal{A}}
\newcommand{\cB}{\mathcal{B}}

\newcommand{\cO}{\mathcal{O}}

\newcommand{\cL}{\mathcal{L}}

\newcommand{\cK}{\mathcal{K}}

\newcommand{\Gal}{{\rm Gal}}
\newcommand{\cir}{\mathrm{circle}}
\newcommand{\BU}{\mathrm{BU}}
\newcommand{\expm}{\mathrm{exp}}
\newcommand{\reg}{\mathrm{reg}}
\newcommand{\sing}{\mathrm{sing}}

\renewcommand{\leq}{\leqslant}
\renewcommand{\le}{\leqslant}
\renewcommand{\geq}{\geqslant}
\renewcommand{\ge}{\geqslant}
\renewcommand{\bar}{\overline}

\newcommand{\x}{\mathbf{x}}

\renewcommand{\c}{\mathbf{c}}

\renewcommand{\b}{\mathbf{b}}

\newcommand{\singA}{\mathbf{A}}

\newcommand{\fo}{\mathfrak{o}}

\newcommand{\ve}{\varepsilon}

\DeclareMathOperator{\disc}{disc}

\DeclareMathOperator{\Eff}{Eff}

\DeclareMathOperator{\im}{im}

\DeclareMathOperator{\Pic}{Pic}

\DeclareMathOperator{\vol}{vol}
\DeclareMathOperator{\supp}{supp}

\DeclareMathOperator{\Fr}{Fr}
\DeclareMathOperator{\Res}{Res}
\DeclareMathOperator{\CH}{CH}
\DeclareMathOperator{\Br}{Br}

\DeclareMathOperator{\rk}{rk}
\DeclareMathOperator{\Spec}{Spec}

\newcommand{\1}{\mathbf{1}}

\newtheorem{theorem}{Theorem}[section]

\newtheorem{lemma}[theorem]{Lemma}
\newtheorem{conjecture}[theorem]{Conjecture}
\newtheorem{proposition}[theorem]{Proposition}
\newtheorem{heuristic}[theorem]{Heuristic}

\theoremstyle{definition}

\newtheorem{remark}[theorem]{Remark}
\newtheorem{example}[theorem]{Example}
\newtheorem{hyp}[theorem]{Hypothesis}

\newcommand{\tX}{\tilde{X}}
\newcommand{\tU}{\tilde{U}}
\newcommand{\tD}{\tilde{D}}
\newcommand{\norm}[1]{\left\lVert#1\right\rVert}
\newcommand{\abs}[1]{\left\lvert#1\right\rvert}
\newcommand{\cOPt}{\mathcal{O}_{\mathbb{P}^3}}
\newcommand*\diff{\mathop{}\!\mathrm{d}} 
\newcommand{\eps}{\varepsilon}

\newcommand{\bA}{\mathbb{A}}
\newcommand{\mX}{\mathfrak{X}}
\newcommand{\mU}{\mathfrak{U}}

\newcommand{\mtD}{\tilde{\mathfrak{D}}}
\newcommand{\mtX}{\tilde{\mathfrak{X}}}
\newcommand{\mtU}{\tilde{\mathfrak{U}}}

\begin{document}

\begin{abstract}
We develop a  heuristic for the density of integer points on affine cubic surfaces. Our heuristic applies to smooth  surfaces defined by  cubic polynomials that are
log K3, but it can also be adjusted to handle singular cubic surfaces. We compare our heuristic to 
Heath-Brown's prediction for sums of three cubes, as well as to   asymptotic formulae in the literature around Zagier's work on the Markoff cubic surface, and work of 
Baragar and Umeda on further surfaces of Markoff-type. 
We also test our heuristic against  numerical data for several families of cubic surfaces.
\end{abstract}


\date{\today}

\title[Integral points on cubic surfaces]{Integral points on cubic surfaces:\\  heuristics and numerics}

\author{Tim Browning}
\address{IST Austria\\
Am Campus 1\\
3400 Klosterneuburg\\
Austria}
\email{tdb@ist.ac.at}

\author{Florian Wilsch}
\address{G\"ottingen University\\
Bunsenstraße 3--5\\
37073 G\"ottingen\\
Germany}
\email{florian.wilsch@mathematik.uni-goettingen.de}

\maketitle

\thispagestyle{empty}
\setcounter{tocdepth}{1}
{\small
\begin{multicols}{2}
\tableofcontents
\end{multicols}
}

\section{Introduction}

Let $U\subset \AAA^3$ be a cubic surface defined by an irreducible polynomial $f\in \ZZ[x,y,z]$ of degree $3$, such that the surface is smooth over $\QQ$.
This paper develops  heuristics for the expected asymptotic behaviour of the counting function
\[
N_U(B)=\#\left\{ (x,y,z)\in \ZZ^3: \max\{|x|,|y|,|z|\}\leq  B, ~f(x,y,z)=0\right\},
\]
as $B\to \infty$. 
A well-studied  example is the cubic polynomial
\begin{equation}\label{3cubes}
  f(x,y,z)=x^3+y^3+z^3-k,
\end{equation}
for  non-zero $k\in \ZZ$. When $k$ is cube-free, it has been conjectured by
Heath-Brown~\cite{33}  that $N_U(B)\sim c_k\log B$ for an appropriate constant $c_k$, which is positive if and only if $k\not\equiv \pm 4 \bmod{9}$. 
Thus, 
in this example, 
 $U(\ZZ)$ is expected to be infinite if and only if $k\not\equiv \pm 4 \bmod{9}$. 
For some values of $k$, this follows from the presence of parametric solutions. 
When $k=1$, for example, the parameterisation
\begin{equation}\label{eq:k=1}
(9t^4)^3+(3t-9t^4)^3+(1-9t^3)^3=1
\end{equation}
was discovered by Mahler \cite{mahler} in 1936. In 1956, Lehmer \cite{lehmer}  discovered an infinite family of parameterisations for the case $k=1$.
In general, however, even showing that 
$U(\ZZ)$ is non-empty has proved very challenging. Indeed,  for the cases $k=33$ and $k=42$, this has only recently been established by  Booker~\cite{booker} and Booker--Sutherland~\cite{booker'}, respectively. In fact, Booker and Sutherland \cite[Sec.~2A]{booker'} also  provide experimental evidence for Heath-Brown's conjecture by comparing $\sum_k N_{U}(B)$ with $\sum_k c_k \log B$, where the sum  runs over  cube-free integers $k\in [ 3, 1000]$ and $B$ runs over the interval $ [10^{7.5},10^{15}]$.

In some cases, the cubic surface admits a group action that renders an analysis of $N_U(B)$ more tractable.
When $\mathrm{N}_{K/\QQ}(x,y,z)$ is the norm form associated to a cubic extension $K/\QQ$, the proof of Dirichlet's unit theorem allows one to study the counting function for the 
polynomial $
f(x,y,z)=\mathrm{N}_{K/\QQ}(x,y,z)-k$,
for any non-zero $k\in \ZZ$.  Assuming that $U(\ZZ)\neq \emptyset$, it follows 
from \cite[Sec.~5]{wei}
that 
\begin{equation}\label{eq:lgw}
N_U(B)\sim c_k (\log B)^{r-1},
\end{equation}
for a suitable constant $c_k>0$, where $r$ is the number of infinite places in $K$. 

The Markoff surface $U\subset \AAA^3$ is defined by  
the polynomial
\begin{equation}\label{eq:markoff}
f(x,y,z)=x^2+y^2+z^2-3xyz.
\end{equation}
It follows from work of Zagier~\cite{zagier} that there exists a constant $c>0$ such that 
$N_U(B)\sim c (\log B)^2$, as $B\to \infty$. 
For given $k\in \ZZ$, the arithmetic of the surfaces 
\begin{equation}\label{eq:GS}
x^2+y^2+z^2-xyz=k
\end{equation}
 has been   investigated deeply by 
Ghosh and Sarnak~\cite{sarnak}, who raise interesting questions about failures of the integral Hasse principle. In particular, it follows from~\cite[Thm.~1.2]{sarnak} that the integral Hasse principle fails for at least $\sqrt{K}(\log K)^{-\frac{1}{4}}$ integer coefficients $|k|\leq K$.
These observations have been refined and  put into the context of the Brauer--Manin obstruction by Loughran--Mitankin~\cite{LM} and 
Colliot-Th\'el\`ene--Wei--Xu~\cite{colliot}.  In particular, the numerical evidence presented in~\cite[Conj.~10.2]{sarnak} for the density of Hasse failures is not wholly  accounted for by the Brauer group. 
For the surfaces  \eqref{eq:GS}, 
an asymptotic formula of the shape  $N_U(B)\sim c_{\text{GS}}(\log B)^2$ can  be deduced by 
taking $n=4$ and $a=1$ in recent work by Gamburd, Magee and Ronan \cite[Thm.~3]{gamburd}. 
These surfaces are smooth when $k\not \in \{0,4\}$, providing many  examples to compare with our heuristic. 

The Markoff surface defined by~\eqref{eq:markoff} is singular and  only fits into the scope of our heuristic after passing to a minimal desingularisation (as  explained  
in Section~\ref{s:markoff}). However, Baragar and Umeda~\cite{baragar} have shown how to adapt Zagier's argument to study $N_U(B)$ for surfaces $U\subset \AAA^3$ defined by the polynomial
\begin{equation}\label{eq:bu-equation}
f(x,y,z)=ax^2+by^2+cz^2-dxyz-1,
\end{equation}
for  $a,b,c,d\in \NN$ such that 
$4abc-d^2\neq 0$ and
such that $d$ is divisible by $a$, $b$ and $c$. 
This surface is smooth over $\QQ$. Moreover, $U$ admits three non-commuting involutions defined over $\ZZ$, which are the so-called {\em Vieta involutions}. The induced action by
the free product
 $(\ZZ/2\ZZ)*(\ZZ/2\ZZ)*(\ZZ/2\ZZ)$ has finitely many orbits and so, as for \eqref{eq:markoff},  this  can  be used to study the set $U(\ZZ)$ of integral points.
 The surfaces \eqref{eq:bu-equation} generalise cubic  surfaces considered by Mordell~\cite{mordell}, and  were first studied by 
Jin and Schmidt~\cite{jin}, who show $U(\ZZ)\neq \emptyset$ 
if and only if $f$ is one of seven possibilities (up to permutation of the coefficients), with one of them being given by 
\begin{equation}\label{eq:many}
f(x,y,z)=x^2+by^2+bz^2-2bxyz-1,
\end{equation}
for any  $b\in \NN$. This case is ignored, however, since the surface contains the line $x-1=y-z=0$, 
which contributes at least 
 $2B$ points  to $N_U(B)$.
Baragar and Umeda~\cite[Thm.~5.1]{baragar} have shown that in each of the six remaining cases, there is a constant $c_{\text{BU}}>0$ such that 
 \begin{equation}\label{eq:BU-asy}
N_U(B)\sim c_{\text{BU}} (\log B)^2,
\end{equation}
as $B\to \infty$. The coefficient vectors for the six  surfaces, together 
with a numerical value for $c_{\text{BU}}$, are presented in Table~\ref{tb:surface-params}. 
 In fact, the  article~\cite[Sec.~4]{baragar} contains a small oversight that affects the leading constant. The authors multiply their constant by $3$ to account for negative coordinates, 
 whereas it should be multiplied by $4$: for each solution $(x,y,z)\in \NN^3$, there are the three additional solutions $(x,-y,-z)$, $(-x,y,-z)$, and $(-x,-y,z)$. The same oversight applies to \cite[Sec.~5]{baragar} and the constants in our Table~\ref{tb:surface-params} are thus corrected by a factor of $4/3$.
It is worth highlighting that  while Baragar and Umeda use the height $|x| + |y| + |z|$, rather than the sup-norm, 
this makes no difference to the leading term, since these norms are equivalent and the counting function  grows logarithmically in $B$.

\begin{table}[t]
  \begin{tabular}{c@{\hspace{4\tabcolsep}}rrrr@{\hspace{4\tabcolsep}}l}\toprule
   & $a$ & $b$ & $c$ & $d$ & $c_{\text{BU}}$  \\\midrule
  (i) & $1$ & $5$ & $5$ & $5$ & $5.22750241554\dots$ \\
  (ii) &$1$ & $3$ & $6$ & $6$ & $2.96508393913\dots$\\
  (iii) &$2$ & $7$ & $14$ & $14$ &$2.46790596426\dots$\\
  (iv) &$2$ & $2$&$3$ & $6$ & $4.05640933744\dots$\\
  (v) &$6$ & $10$ & $15$ & $30$ & $2.49318310680\dots$\\
  (vi) &$1$ & $2$ & $2$ & $2$ & $4.92081804684\dots$\\\bottomrule
  \end{tabular}
  \medskip
  \caption{Surfaces studied by Baragar and Umeda~\cite{baragar}.}\label{tb:surface-params}
\end{table}

\medskip

We are now ready to discuss our main heuristic, which comes from
the circle method. Such heuristics are typically obtained 
by examining the major arc contribution, for a suitable set of major arcs, and ignoring the contribution from the minor arcs. This approach would suffice for surfaces with trivial Picard group, since then the associated singular series  converges.  
However, for surfaces with non-trivial Picard group, such as the cubic surface
$x^3+ky^3+kz^3=1$ considered in Section~\ref{s:oblong}, 
the singular series 
diverges and the precise choice of major arcs would have a strong effect on the purported value of the leading constant. We shall avoid this difficulty by 
adopting a variant of 
the smooth $\delta$-function version of the circle method originating in
work of Duke, Friedlander and Iwaniec~\cite{DFI}, and  later  developed by
Heath-Brown~\cite[Thm.~1]{HB'}.  Once coupled with Poisson summation, the main idea is to ignore the contribution from the non-trivial characters, in order to 
obtain a heuristic for $N_U(B)$ for any cubic surface $U\subset \AAA^3$ that is smooth and log K3 over $\QQ$.
Here, we say that a smooth cubic surface $U_\QQ\subset \AAA_\QQ^3$ is {\em log K3} 
 if the minimal desingularisation $\tX_\QQ$ of the compactification $X_\QQ$ of $U_\QQ$ in $\PP_\QQ^3$ satisfies the property that the \emph{boundary} $\tD=\tX_\QQ \setminus U_\QQ$ is a divisor with strict normal crossings
whose class in $\Pic \tX$ is $\omega_{\tX}^\vee$. In particular, 
it follows from the adjunction formula that 
$U$ is log K3 if $X$ itself is smooth over $\QQ$ and
$X\setminus U$ has strict normal crossings.

In general, it may happen that $U$ contains $\AAA^1$-curves that are defined over $\ZZ$; for instance, this happens if any of the lines on
$X_{\bar{\QQ}}=X\otimes_\ZZ \QQbar$
 are defined over $\ZZ$, as in the example~\eqref{eq:many}.
 It is therefore natural to try and classify those log K3 surfaces which admit infinitely many $\AAA^1$-curves, a programme that is already under way over $\bar\QQ$, thanks to Chen and Zhu~\cite{A1}. 
In the presence of $\AAA^1$-curves it is natural to study the subset
$U(\ZZ)^\circ$ obtained by removing those points in $U(\ZZ)$ that are contained in any $\AAA^1$-curves defined over $\ZZ$, since we expect the
 contribution from integer points on these curves to dominate the counting function. 
As $B\to \infty$, this  leads us to analyse the modified counting function
\begin{equation}\label{eq:restrict-count}
N_U^\circ(B)=
\#\left\{ (x,y,z)\in U(\ZZ)^\circ: \max\{|x|,|y|,|z|\}\leq  B\right\}.
\end{equation}
We are now ready to reveal the main conjecture issuing from our investigation. 

\begin{conjecture}\label{con1}
  Let $U\subset \AAA^3$ be a cubic surface that is smooth and log K3 over $\QQ$ and that is
     defined by a cubic polynomial $f\in \ZZ[x,y,z]$. 
     Denote by $\rho_U$ the Picard number of $U$ over $\QQ$ and by $b$ the maximal number of components of $\tD(\RR)$ that share a real point. Then 
  \[
    N_{U}^\circ(B)=O_{U} \left((\log B)^{\rho_U+b}\right).
  \]
  \end{conjecture}

This resonates with  a conjecture of Harpaz~\cite[Conj.~1.2]{harpaz},
where an unspecified  logarithmic growth is predicted for a certain class of log K3 surfaces. 
Conjecture~\ref{con1} is the crudest conclusion that one can draw from our heuristic, which actually predicts an asymptotic formula for 
$N_{U}^\circ(B)$. 

\begin{heuristic}\label{main-heur}
    Let $U\subset\bA^3$ be a cubic surface satisfying the hypotheses of Conjecture~\ref{con1}, with 
    Picard number $\rho_U$, such that  $U(\ZZ)$ is not thin. 
    Denote by $b$ the maximal number of components of $\tilde D(\RR)$ that share a real point, and by $\cA$ all such sets of cardinality $b$; for $A = \{\tilde D_1(\RR),\dots, \tilde D_b(\RR)\} \in \cA$, denote by $Z_A=\tilde D_1(\RR)\cap \cdots \cap \tilde D_b(\RR)$ the corresponding intersection. Then
    \[ 
    N^\circ_U(B) \sim c_{\mathrm{h}} (\log B)^{\rho_U + b},
    \]
    as $B\to \infty$,
    with
    \begin{equation}\label{eq:leading-const-with-gamma}
    c_{\mathrm{h}} = \gamma_U \cdot \tau_{U,H}(V),
    \end{equation}
    where
 $\gamma_U\in \QQ_{>0}$,  
$\tau_{U,H}$ is the Tamagawa measure  
induced by the standard height $H$, 
and 
$V$ is the set  of limit points of  $U(\ZZ)$ in
   $\bigcup_{A\in \cA} Z_A \times U(\A^{\mathrm{fin}}_\ZZ)$. 
  \end{heuristic}

It is natural to impose the assumption that $U(\ZZ)$ is Zariski dense in this heuristic. We have been led to further require that the set of integral points is not thin, since we do not expect the leading constant in 
 \eqref{eq:lgw}  to agree with this heuristic and 
 $U(\ZZ)$ is thin in this case.

The Tamagawa measure $\tau_{U,H}$ is defined using residue measures at the archimedean place, as described in work  of
Chambert-Loir and Tschinkel \cite[Secs.~2.1.9 and 2.1.12]{ACL}.
The set $V$ of limit points can be studied via the Brauer--Manin obstruction, whose use to study integral points goes back to Colliot-Th\'el\`ene and Xu \cite{CTX}; Santens~\cite{tim} has developed a variant that can also explain failures of accumulation phenomena at the infinite place. A further obstruction to approximation over $\RR$ is analytic in nature, as expounded in  work of Wilsch~\cite{wilsch} and Santens~\cite{tim}. 
For log Fano varieties whose Brauer group modulo constants is finite, it would follow from a conjecture of Santens~\cite[Conj.~6.6 and Thm.~6.11]{tim} and an equidistribution theorem of Chambert-Loir and Tschinkel \cite[Prop.~2.10]{ACL} that the algebraic Brauer--Manin obstruction is the only one.
However,  as observed in  \cite{colliot,LM} for the Markoff-type surfaces~\eqref{eq:GS}, the
Brauer--Manin obstruction is  not always sufficient for log K3 surfaces.

In order to illustrate our work, we state here a concrete conjecture for  the
polynomial 
$f(x,y,z)=x^3+ky^3+kz^3-1$, where $k>1$ is  square-free.
We shall see in Section \ref{s:oblong} that $b=1$ and $\rho_U=2$ for the surface 
$U\subset \AAA^3$
defined by $f$.

\begin{conjecture}\label{k>1}
  Let $U\subset \AAA^3$ be the cubic surface defined by 
$x^3+ky^3+kz^3=1$, for a square-free integer $k>1$. Then Heuristic \ref{main-heur} holds with 
$\gamma_U=\frac{3}{8}$ and $V=D(\RR)\times 
U(\A^{\mathrm{fin}}_\ZZ)$.
\end{conjecture}

By adapting the  parameterisation of 
Lehmer \cite{lehmer} to the setting of Conjecture~\ref{k>1}, we are led to 
infinitely many $\mathbb{A}^1$-curves of increasing degree contained in $U$. 
We expect that a modification to work of 
Coccia \cite{coccia} would yield analogues of his results for $k=1$: the set of integer points on the
Lehmer curves is thin (cf.~\cite[p.~371]{coccia}), while its complement is not (cf.~\cite[Thm.~8]{coccia}).

\subsection*{Summary of the article}
We conclude our introduction by summarising the contents of the article.

 \subsubsection*{Section \ref{ssec:arch-measures-general}}

Formally speaking, 
our heuristic will involve the quantities
$$
\int_{-\infty}^\infty \int_{[-B,B]^3} e(t f(x,y,z)) \d x\d y\d z\d t
$$
and 
$$
\mu_\infty(B)=\lim_{\ve\to 0} \frac{1}{2\ve} 
\vol\{\x\in [-B,B]^3:
| f(\x)|<\ve\},
$$
both of which capture the real density of points on $U$.  We shall introduce some hypotheses 
concerning the convergence properties of the oscillatory integral. Moreover, 
in Proposition~\ref{prop:arch-volume-expansion} we shall apply
work of Chambert-Loir and Tschinkel \cite{ACL} to deduce that 
$\mu_{\infty}(B)$
grows like a power of $\log B$, as $B\to \infty$.

 \subsubsection*{Section \ref{s:heuristic}}

This  
is the heart of our paper and concerns a circle method heuristic  
applied  to $N_U(B)$. We shall derive an  asymptotic expansion of the  contribution from the trivial character, as $B\to \infty$, 
for a smoothly weighted variant of the counting function $N_U(B)$. 
This is achieved in Theorem \ref{thm:trivial-character}, which will align  with Conjecture~\ref{con1}. When $\rho_U=0$, we will arrive at a  precise asymptotic prediction for $N_U(B)$ in 
 Heuristic~\ref{heur:circle-method}. Furthermore, we shall place Heuristic \ref{main-heur} in the context of the Manin conjecture for rational points on Fano varieties.

 \subsubsection*{Section \ref{s:norm}} 
We show that the exponent of $\log B$ in Heuristic~\ref{heur:circle-method}  matches 
the asymptotic formula in \eqref{eq:lgw}.

 \subsubsection*{Section \ref{s:cubes}} 
We demonstrate that Heuristic~\ref{heur:circle-method}  matches the heuristic developed by Heath-Brown 
 \cite{33} for the sums of cubes example in \eqref{3cubes}, when $k$ is cube-free. 

 \subsubsection*{Section \ref{s:gen_cubes}}
We shall adapt our work in  Section \ref{s:cubes} to develop a heuristic for the surface 
  \eqref{3cubes} when $k$ is a cube, together with  the cubic surface 
 $
 x^3+ky^3+kz^3=1,
 $ 
when $k>1$
is a square-free integer. 
 These surfaces will be seen to have Picard rank $3$ and $2$, respectively. In the former case 
we compare with Heuristic \ref{main-heur} when $k=1$, using numerical data provided for us by Andrew Sutherland. 
For the second family of surfaces, we will gather numerical data for all square-free integer values of $2\leq k\leq 1000$ and discuss Conjecture \ref{k>1}.

 \subsubsection*{Section \ref{s:eg2}}  
  We test Heuristic~\ref{heur:circle-method} against the asymptotic formula \eqref{eq:BU-asy} of Baragar and Umeda. We will find that it correctly predicts the exponent of $\log B$, but that it fails to explain the leading constant. (Although we omit the details, similar arguments should go through for the 
  surfaces  \eqref{eq:GS}.)
In line with  Heuristic 
\ref{main-heur}, 
 we shall modify the heuristic leading constant to take into account failures of strong approximation.  
All of the surfaces in Table~\ref{tb:surface-params} are equipped with a group action that makes it very efficient to test numerically for failures of strong approximation. In addition to uncovering failures coming  from  the Brauer--Manin obstruction, 
we will find failures of strong approximation that occur at infinitely many primes.
In particular, we observe a failure of the  \emph{relative Hardy--Littlewood property}, as  introduced by Borovoi and Rudnick \cite[Def.~2.3]{borovoi-rudnick}. 
On the other hand, we conduct a numerical investigation of equidistribution in Section \ref{s:equi}, finding that the observed frequencies of reductions modulo $m$ occur with the expected frequency,
for various  $m\in \NN$. Nonetheless, it  seems unlikely that 
Heuristic~\ref{main-heur} is compatible with the 
 numerical values  occurring in \eqref{eq:BU-asy}. 

 \subsubsection*{Section \ref{s:markoff}}
We extend our heuristic to the singular Markoff surface, as defined by the polynomial \eqref{eq:markoff}. We will find that the situation is similar to the examples of 
Baragar and Umeda in Section \ref{s:eg2}. However, while failures of strong approximation don't explain the leading constant, as in Section \ref{s:eg2}, such a modification does help to explain the power of $\log B$.

 \subsubsection*{Section  \ref{s:higher-picard}}

 We gather
 numerical evidence for 
two further cubic surfaces. The first 
is
the cubic surface $U\subset \AAA^3$ defined by  
$$
(x^2-ky^2)z=y-1,
$$ 
for a square-free integer $k>1$. Under suitable assumptions, it has been shown by  Harpaz \cite{harpaz} that 
$U(\ZZ)$ 
is  Zariski dense, prompting him to ask in  \cite[Qn.~4.4]{harpaz} about 
the exponent of $\log B$ in the associated counting function $N_U^\circ(B)$, after removing the $\AAA^1$-curve $z=y-1=0$. We apply 
Heuristics~\ref{main-heur} and \ref{heur:circle-method} to deduce that the expected exponent is $2$ and we gather numerical evidence for all square-free integers $2\leq k\leq 1000$, which strongly supports this.

The second surface,  which takes
the shape
$$
  (ax+1)(bx+1) + (cy+1)(dy+1) =xyz,
$$
for  suitable  $a,b,c,d\in \ZZ$, 
 was also suggested to us by  Harpaz (in private communication). This example will be seen to have non-trivial Picard group and contain several 
$\AAA^1$-curves that are defined over $\ZZ$. 
In this case, moreover,  there are no involutions defined over $\QQ$ and so we do not have a way to approach an asymptotic formula for the counting function, nor do we have an efficient way of enumerating integer points. 
By searching for points of height $\leq 10^{10}$, and identifying problematic $\AAA^1$-curves of degrees $1,\dots,4$,  we find that the numerical data bears little resemblance to  Heuristic~\ref{main-heur}.


 \subsubsection*{Section \ref{s:conclusion}}
 We offer some concluding remarks.

\subsection*{Data availability}

The data used in Sections 
\ref{s:gen_cubes} and \ref{s:higher-picard}
is hosted on the 
{\em G\"ottingen Research Online Data} repository \cite{data}. 
The code used to  determine the data in
Sections~\ref{s:oblong}, \ref{s:final} and \ref{s:outlier}
is found on the second author's 
\href{https://github.com/fwilsch/cubicpts}{github}
 page.

\subsection*{Acknowledgments}
The authors owe a debt of thanks to Yonatan Harpaz for asking 
about circle method heuristics for log K3 surfaces.
His contribution to the resulting discussion is gratefully acknowledged. 
Thanks are also due to Andrew Sutherland for help with numerical data for 
the equation $x^3+y^3+z^3=1$, together with 
Alex Gamburd, Amit Ghosh, Peter Sarnak and Matteo Verzobio  for their interest in this paper.
Special thanks are due to Victor Wang for numerous 
 helpful conversations about the circle method heuristics. 
While working on this paper, the authors were
supported by a 
FWF grant 
(DOI 10.55776/P32428), and the first author 
was supported  by a further 
FWF grant 
(DOI 10.55776/P36278)
and 
a grant from the 
School of Mathematics at the 
{\em Institute
for Advanced Study} in Princeton.

\section{Archimedean  densities} 
\label{ssec:arch-measures-general}

Let $f\in \ZZ[x,y,z]$ be a cubic polynomial and put 
\begin{equation}\label{eq:def-g}
g(\x)=B^{-3}f(B\x),
\end{equation}
where $\x=(x,y,z).$
Note that $g(\x)\ll 1$, where the implied constant is only allowed to depend on the coefficients of $f$.
Given a compactly supported bounded function 
$w:\RR^3\to \RR_{\geq 0}$, our work will feature the oscillatory integral
\begin{equation}\label{eq:def-I}
I(t)=\int_{\RR^3}
w(\x) e\left(tg(\x)\right)\d \x,
\end{equation}
for $t\in \RR$. 
Note that $I(t)$ also depends on $B$, in view of the definition of $g$. 
In traditional applications of the circle method the real density often arises formally via the oscillatory integral 
\begin{equation}\label{eq:infinity'}
\sigma_\infty(B)=
\int_{-\infty}^\infty
I(t) \d t.
\end{equation}
An alternative formulation is via the limit
\begin{equation}\label{eq:infinity}
\mu_{\infty}(B)=\lim_{\ve\to 0} \frac{1}{2\ve} 
\vol\{\x\in [-B,B]^3:
| f(\x)|<\ve\},
\end{equation}
and it usually  possible to prove 
that $\sigma_\infty(B)$ and $\mu_{\infty}(B)$ both converge to the same quantity, as $B\to \infty$. 
However,  there are subtleties in the present setting and it will be convenient to build this into our  assumptions. 

\begin{hyp}\label{hyp:garlic}
Assume that 
$\sigma_\infty(B)$ and 
$ \mu_\infty(B)$ both converge, 
in the notation of \eqref{eq:infinity'} and  \eqref{eq:infinity}.  Then 
 $\sigma_\infty(B)
\sim \mu_\infty(B)$, as $B\to \infty$.
\end{hyp}

For the polynomial   \eqref{3cubes}, we shall verify  this hypothesis 
in Section \ref{s:clip}.

\subsection{Oscillatory integrals}

Let us begin by discussing some properties and assumptions around the oscillatory integral $I(t)$ in 
\eqref{eq:def-I} and the real density \eqref{eq:infinity'}.
To begin with, it is clear that $I(t)$ is infinitely differentiable and satisfies 
\begin{equation}\label{eq:trivial}
I(t)\ll 1,
\end{equation}
 for any $t\in \RR$, where 
the implied constant depends only on $w$.
If  $w$ is a smooth function and 
$|\nabla f(\x)|>0$ throughout $\supp(w)$, 
then it is possible to establish 
exponential decay for $I(t)$, by 
using repeated applications of integration by parts.
This would lead to 
a bound of the form
$$
\int_{-\infty}^\infty
|I(t)| \d t \ll  1,
$$
which is the most favourable situation and underpins  many applications of the  circle method. 
When  
$\nabla f(\x)=\0$
for some $\x\in \supp(w)$,
on the other hand,  the situation is much more subtle, as indicated 
in works of Greenblatt \cite{greenblatt} and Varchenko \cite{varchenko}.

\begin{example}\label{example-3cubes}
It is instructive to consider the polynomial
$  f(x,y,z)=x^3+y^3+z^3-k$,
for  non-zero $k\in \ZZ$. Taking $w(x,y,z)=\nu(x)\nu(y)\nu(z)$, where $\nu:\RR\to\RR$ is a smooth even bump function such that $\nu(x)=1$ on $[-1,1]$, it follows that 
$$
I(t)=e(-kt/B^3) R_\nu(t)^3,
$$
where
$$
R_\nu(t)=\int_{-\infty}^\infty \nu(x) e(tx^3)\d x.
$$
The second  derivative test \cite[Lem~4.4]{tit} yields  $R_\nu(t)\ll \min\{1,|t|^{-1/3}\}$. Hence 
$
I(t)\ll  \min\{1,|t|^{-1}\}$
and we have 
$$
\int_{-\infty}^\infty |t|^{-\delta} |I(t)|\d t \ll_\delta 1,
$$
for any $\delta\in (0,1)$, where the implied constant depends on $\delta$.
We can get an asymptotic formula for $I(t)$, as $|t|\to \infty$, by noting that 
$$
R_\nu(t)=\int_{-\infty}^\infty e(tx^3)\d x -
R_{1-\nu}(t).
$$
The first term can be evaluated as 
$$
\int_{-\infty}^\infty e(tx^3)\d x=\frac{1}{|t|^{1/3}}\cdot \frac{\Gamma(\frac{1}{3})}{(2\pi)^{1/3}\sqrt{3}},
$$
for any $t\in \RR^*$. Since $1-\nu$ is smooth and supported on the region $\RR^2\setminus [-1,1]$, the second term is easily seen to be $O_N(|t|^{-N})$ on repeated integration by parts. 
Hence
\begin{equation}\label{eq:I(t)}
I(t)=
\frac{\Gamma(\frac{1}{3})^3}{6\pi \sqrt{3}}
\cdot 
\frac{e(-kt/B^3)}{|t|}
 +O_N(|t|^{-N}).
\end{equation}
This formula can be used to check that  the integral
$$
\int_{-\infty}^\infty  \frac{I(t) }{(2+\pi^2t^2)^{\frac{s+1}{2}}}
\d t
$$
is a holomorphic function of $s\in \CC$ in the half-plane $\Re(s)\geq -1$.
\end{example}

Motivated by this example, our circle method heuristic will proceed under the following assumptions about $I(t)$.

\begin{hyp}\label{hyp}
Let $I(t)$ be given by \eqref{eq:def-I}. Then 
the following hold:
\begin{enumerate}
\item[(i)] We have 
$$
\int_{-\infty}^\infty |t|^{-\delta} |I(t)|\d t \ll_\delta 1,
$$
for any $\delta\in (0,1)$, where 
the implied constant  depends only on $w, f$ and $\delta$.
\item[(ii)] The integral
$$
\int_{-\infty}^\infty  \frac{I(t) }{(2+\pi^2t^2)^{\frac{s+1}{2}}}
\d t
$$
is a holomorphic function of $s\in \CC$ in the half-plane $\Re(s)\geq -1$.
\end{enumerate}
\end{hyp}

\subsection{The real density}

We now proceed to an analysis of 
$\mu_\infty(B)$, as defined in~\eqref{eq:infinity}, using work of 
Chambert-Loir and Tschinkel \cite[Thm.~4.7]{ACL}. To begin with, 
we write
\[
  \mu_\infty(B)=\lim_{\eps\to 0} \frac{1}{2\eps} \vol\left\{P\in\bA^3(\RR) : H_\infty(P)\le B, ~\abs{f(x,y,z)}< \eps \right\},
\]
where $H_\infty(P)=\max\{\abs{x},\abs{y},\abs{z},1\}$ for a real point $P=(x,y,z)\in \bA^3(\RR)$. This is a volume on the surface $U$ defined by $f$.   Let $f_0(t_0,x_0,y_0,z_0)$ be the homogenisation of $f(x,y,z)$, with 
  $x=x_0/t_0$, $y=y_0/t_0$, and $z=z_0/t_0$, so that $f = f_0/t_0^3$.
Let $X = V(f_0)$ be the closure of $U$ in $\PP^3$, which we assume to be normal.
Let $\rho\colon \tX\to X$ be a minimal desingularisation. We shall assume that $\tD = \tX\setminus \tU$ has strict normal crossings and that $U$ is log~K3, so that the log canonical bundle $\omega_{\tX}(\tD) \cong \cO_{\tX}$ is trivial. As a consequence of the adjunction formula, this condition is equivalent to $\rho^*\cO_X(D) \cong \cO_{\tX}(\tD)$, where $D=X\setminus U$;
in particular, this is automatic if $X$ is smooth.

The Leray form on $U$ is a regular 2-form $\omega$ on $U$ such that $\diff f \wedge \omega=\diff x\wedge \diff y \wedge \diff z$.  This allows us to write
\begin{equation}\label{eq:mu_infty}
  \mu_\infty(B)=\int_{P\in U(\RR),\ H_\infty(P)\le B} \abs{\omega}.
\end{equation}
We shall endow certain line bundles with adelic metrics.
On $\cO_{\PP^3}(d)$, for $d\in\ZZ$, consider the standard sup-norm
\begin{equation}\label{eq:norm-O-d}
  \norm{g(P)}_v=\frac{\abs{g(P)}_v}{\max\{\abs{t_0}_v,\abs{x_0}_v,\abs{y_0}_v,\abs{z_0}_v\}^d},
\end{equation}
where $P=(t_0:x_0:y_0:z_0)\in\PP^3(\QQ_v)$ is a point over one of the local fields, and $g\in\Gamma(\cO_{\PP^3}(d),U)$ is a local section.
We have $\omega_{\PP^3}\cong \cO_{\PP^3}(-4)$ mapping $\diff x\wedge \diff  y \wedge \diff z$ to $t_0^{-4}$. This induces a metric on $\omega_{\PP^3}$
with
$$
  \norm{\diff x \wedge \diff y \wedge \diff z}_{\omega_{\PP^3}} = \norm{t_0^{-4}}_{\cOPt(-4)} = \max\{1, \abs{x}, \abs{y},\abs{z}\}^4,
$$
after dividing numerator and denominator of~\eqref{eq:norm-O-d} by $t_0^{-4}$. We have an isomorphism
$\cOPt(X) \to \cOPt(3)$,
mapping the canonical section $1_X$ to $f_0$, inducing an adelic metric on the former bundle.
Now the adjunction isomorphism $\omega_{X}\to\omega_{\PP^3}(X)|_{X}$ induces a metric on $\omega_{X}$. We follow~\cite[Sec.~2.1.13]{ACL} to get an explicit description.
Consider the local equation $f\in\Gamma(\bA^3,\cO_{\PP^3}(-X))$  of $X$. On this bundle, we have an adelic metric (induced by the one on $\cOPt(X)$), with
\[
  \norm{f}_{\cOPt(-X)} \norm{1_{X}}_{\cOPt(X)}=\abs{f},
\]
since the product of the two sections on the left is $f$ in $\cO_{\PP^3}\subset \cK_{\PP^3}$. As the adjunction isomorphism sends $\omega\mapsto \omega\wedge f^{-1}\diff f$, we get
\begin{equation}\label{eq:norm-omega}
  \begin{aligned}
  \norm{\omega}_{\omega_{X}}
  & =
  \norm{\omega \wedge f^{-1} \diff f}_{\omega_{\PP^3}(X)}\\
&    =
  \norm{f}_{\cOPt(-X)}^{-1}
  \norm {\omega \wedge \diff f}_{\omega_{\PP^3}}
  \\
  &=
  \frac{ \norm{1_{X}}_{\cOPt(X)} }{ \abs{f} }
  \norm{\diff x \wedge \diff y \wedge \diff z}_{\omega_{\PP^3}}
  \\
  &=
  \frac{ \abs{f} }{ \max\{1, \abs{x}, \abs{y}, \abs{z}\}^3 }
  \frac{1}{\abs{f}}
  \max\{1,\abs{x}, \abs{y}, \abs{z}\}^4
  \\
  &=
  \max\{ 1, \abs{x}, \abs{y}, \abs{z} \}.
  \end{aligned}
\end{equation}

Using all
this, we can reformulate~\eqref{eq:mu_infty}  as
\[
  \mu_\infty(B) =
  \int_{\{P\in X(\RR):  \norm{t_0}^{-1} \le B\}}
  \norm{1_D}^{-1} \frac{\abs{\omega}}{\norm{\omega}}.
\]
As $\rho$ is crepant, the metric on $\omega_X$ can be pulled back to one on $\omega_{\tX}$, and the one on its dual bundle  $\cO_X(D)$ to $\cO_{\tX}(\tD)$ by the log K3 assumption. It follows that the above volume can be expressed as 
\[
  \mu_\infty(B) =
  \int_{\{P\in \tX(\RR):  \norm{\rho^*t_0}^{-1} \le B\}}
  \norm{1_{\tD}}^{-1} \frac{\abs{\rho^*\omega}}{\norm{
    \rho^*\omega}}
\]
on the desingularisation: the exceptional set and its image are null sets and the integrands and height conditions coincide outside them.
In the notation of~\cite[Sec.~4.2]{ACL},  we have 
\[
  \norm{1_{\tD}}^{-1} \frac{\abs{\rho^*\omega}}{\norm{\rho^*\omega}} = \diff \tau_{(\tX,\tD)},
\]
and an asymptotic expansion of this quantity is studied by means of its Mellin transform and a Tauberian theorem~\cite[Thm.~4.7]{ACL}. In this analysis, Tamagawa measures on certain subsets of $D$ arise naturally.

Let $b$ be the maximal number of components of $\tD(\RR)$ that share a common real point,
as in Conjecture \ref{con1}.
 Note that if the set of integral points is Zariski dense, the set $U(\RR)$ of real points cannot be compact, so that $\tD(\RR)\ne \emptyset$ and $b\ge 1$.
Denote by $\cA$ all sets of such components of cardinality $b$. For each $A\in \cA$, let $Z_A = \bigcap_{D'\in A} D'$ be the intersection, which is a nonempty subset of $\tD(\RR)$ by assumption, and set $D_A = \sum_{D\in A}D$ and $\Delta_A = \sum_{D'\in \cA\setminus A} D'$.
In~\cite[Sec.~2.1.12]{ACL}, Chambert-Loir and Tschinkel define a \emph{residue measure} $\tau_A$ on $Z_A$, which we normalise with a factor of $2^b$ as in \cite[Sec.~4.1]{ACL}.
This measure depends on a metric on $\omega_{\tX}(D_A)$, but $ \norm{1_{\Delta_A}}^{-1} \tau_A$
does not, provided that  the metrics on $\omega_{\tX}(D_A)$ and $\cO_{\tX}(\Delta_A)$ are chosen so that their product coincides with the one on $\omega_{\tX}(D)$.
Moreover, since the residue measure is finite and $\norm{1_{\Delta_A}}$ is bounded from below on $Z_A$ as a consequence of the maximality assumption and compactness, we get a finite  volume
\begin{equation}\label{eq:min-stratum-integral}
  \mu_A = 2^b \int_{Z_A} \norm{1_{\Delta_A}}^{-1} \diff \tau_A.
\end{equation}
We may now record the following  asymptotic formula.

\begin{proposition}\label{prop:arch-volume-expansion}
  Under the above assumptions and notation, including that 
the set of  integral points is Zariski dense,
  we have 
  \[
    \mu_\infty(B)=c_\infty (\log B)^{b} + O\left((\log B)^{b-1}\right),
  \]
  where
  \begin{equation}\label{eq:arch-constant-abstract}
    c_\infty = \frac{1}{b!}\sum_{A\in\cA} \mu_A.
  \end{equation}
\end{proposition}
\begin{proof}
  This follows from \cite[Thm.~4.7]{ACL} with $d_\alpha = \lambda_\alpha = 0$ for all $\alpha\in \cA$, whence $\sigma=0$.  Then \cite[Eq.~(4.3)]{ACL} coincides with $\mu_A$
once evaluated at $s=0$.
\end{proof}

In the next result we derive an  explicit expression for $c_\infty$ 
for certain polynomials featuring in our work. 

\begin{lemma}\label{lem:arch-bu}
Suppose that $X$ is a smooth compactification of the affine surface defined by   $f(x,y,z) = q(x,y,z) - dxyz$, where $q$ is a quadratic polynomial. Then
  \[
    \sum_{A\in \cA} \mu_A=\frac{12}{\abs{d}},
  \]
  so that $c_\infty=6/|d|$ in \eqref{eq:arch-constant-abstract}.
\end{lemma}

\begin{proof}
  Write $q(x,y,z) = Q(x,y,z) + L(x,y,z)+e$ with $Q$ and $L$ homogeneous of degrees $2$ and $1$, respectively. Then $X$ is defined by the cubic form
  \[
    f_0 = dx_0 y_0 z_0 - Q(x_0, y_0, z_0)t_0 - L(x_0, y_0, z_0) t_0^2 - et_0^3.
  \]
  In particular, the complement $D$ of $U$ in $X$ is $V(dx_0y_0z_0)$, a union of three lines $D = L_1 + L_2 + L_3$.
  The Clemens complex associated with $D$ is a triangle with three edges $\{L_i,L_j\}$, for $1\le i<j\le 3$. Associated with each of these edges is a \emph{residue measure} $\tau_{i,j}$ on $L_i\cap L_j$, this intersection consisting of only one rational point $P_{i,j}$.
  For the case $(i,j)=(1,2)$ we have 
  $$
  L_1 = V(t_0,x_0), \quad L_2 = V(t_0,y_0), \quad P_{1,2} = (0:0:0:1).
  $$ 
  We  are interested in the norm
  $\norm{1_{L_k}}_{\omega_P}$ induced by the adjunction formula.
  Consider the affine chart around $P_{1,2}$ given by the coordinate functions $x'=x_0/z_0$, $y'=y_0/z_0$, and $t'=t_0/z_0$. In these coordinates, $P_{1,2}=(0,0,0)$ and $X$ is cut out by
  \[
    f'=f_0/z_0^3 = d x' y' - Q(x',y',1)t' - L (x', y', 1) t'^2 - et'^3.
  \]
  Note that the two partial derivatives $\diff f'/\diff x'$, $\diff f'/\diff y'$ vanish in $(0,0,0)$, so that 
  \[
\frac{    \diff f'}{ \diff t'} (0,0,0)= Q(0,0,1)\neq 0,
  \]
by the smoothness assumption. Since $f'$ is analytic, so is $t'$ as a function of $x'$ and $y'$ by the implicit function theorem. 
  Note that
  \[
    f' = dx'y' - Q (0,0,1) t'(1+O(x')+O(y')) + O(t'^2)
  \]
  as a formal power series. Hence
  \[
    t' = \frac{d x'y'}{Q(0,0,1)} (1 + O(x')+O(y')).
  \]

  Now
  \begin{equation}\label{eq:residue-measure}
    \norm{1_{L_3}}\norm{x'^{-1}\diff x' \wedge y'^{-1}\diff y} = \norm{f'^{-1}}\norm{\diff x' \wedge \diff y' \wedge \diff f'} \frac{\norm{1_{L_1}1_{L_2}1_{L_3}}}{\abs{x'y'}},
  \end{equation}
  by  arguments similar to those appearing in the proof of Proposition \ref{prop:arch-volume-expansion}.
  Analogously to there, $\norm{f'^{-1}}=\max\{\abs{t'}, \abs{x'}, \abs{y'},1\}^{-3} = 1+O(x')+O(y')$.
  Note that
  \[
    \diff f' = (Q(0,0,1)+O(x')+O(y')) \diff t' + f_1 \diff x + f_2 \diff y
  \]
  for some $f_1$ and $f_2\in \QQ[x', y']$, so that
  \begin{align*}
  \norm{\diff x' \wedge \diff y' \wedge \diff f'}
  &= \abs{Q(0,0,1)}\norm{\diff x' \wedge \diff y' \wedge \diff t'(1+O(x')+O(y'))}  \\
  &= \abs{Q(0,0,1)}\max\{\abs{t'},\abs{x'},\abs{y'}, 1\}^4 (1+O(x')+O(y')) \\
  & = \abs{Q(0,0,1)}+O(x')+O(y').
  \end{align*}
  Finally,
  \begin{align*}
  \norm{1_{L_1}1_{L_2}1_{L_3}}
  = \norm{t_0} 
  &= \frac{\abs{t_0}}{\max\{\abs{t_0},\abs{x_0},\abs{y_0},\abs{z_0}\}} \\
  &= \frac{\abs{t'}}{\max\{\abs{t'},\abs{x'},\abs{y'}, 1\}} \\
  & = \abs{\frac{d x'y'}{Q(0,0,1)}} (1 + O(x')+O(y')).
  \end{align*}
Hence~\eqref{eq:residue-measure} becomes
$    \abs{d} + O(x') + O(y')$.
  The integral~\eqref{eq:min-stratum-integral} is over a single point and its value is simply the inverse of~\eqref{eq:residue-measure} evaluated at $P_{1,2} = (0,0)$ in the chosen chart. It still has to be renormalised by multiplying with $c_\RR^2=4$, as in \cite[Sec.~4.1]{ACL}.
 The sum in~\eqref{eq:arch-constant-abstract} now runs over the three edges of the Clemens complex and each of the summands $\mu_A$ is equal to $4/\abs{d}$, finishing the proof.
\end{proof}

\section{A circle method heuristic}\label{s:heuristic}

In this section we explore a heuristic  based on the 
the smooth $\delta$-function version of the circle method due to  Duke, Friedlander and Iwaniec~\cite{DFI}. This was developed and applied to quadratic forms by 
Heath-Brown~\cite{HB'} and put  on an adelic footing 
by Getz~\cite{getz} and Tran~\cite{tran}, in an effort to detect  lower order terms. It is the latter  approach that we shall adopt here. We begin, however, by analysing 
a certain Dirichlet series whose coefficients are  complete exponential sums.

To fix notation, let $U=V(f)\subset \bA^3_\QQ$ and $\mU=V(f)\subset \bA^3_\ZZ$ be the $\QQ$-variety and $\ZZ$-scheme defined by our irreducible, cubic polynomial $f$. Throughout, we shall assume that $U$ is smooth, and only briefly sketch a key difference of the singular case 
in Section \ref{eq:lied}. Denote  by $\mX$ and $X$ the closures of $\mU$ and $U$ in $\PP^3_\ZZ$ and $\PP^3_\QQ$, respectively, and assume that $X$ is normal.
If $X$ is singular, some parts of our arguments will require us to pass to a minimal desingularisation $\tX \to X$,  described by a sequence of blowups of $\PP^3$. Let $\rho\colon \mtX \to \mX$ be the model described by the sequence of blow-ups in the closures of the centres. As $U$ is smooth, these blowups keep it and its model $\mU$ invariant. Finally, denote by $\mtD = \mtX\setminus \mU$ and $\tD=\tX\setminus U$ the \emph{boundary divisor}. If $\tD$  does not have strict normal crossings, we replace $\tX$ and $\mtX$ by varieties arising as blow-ups with centres outside $U$ that achieve this condition. Finally, let $S$ be the set of primes of bad reduction of $\mtX$.

\subsection{Exponential sums and global $L$-functions}

Let $e_q(\cdot)=\exp(\frac{2\pi i\cdot}{q})$, for any $q\in \NN$.
A key role in our work will be played by the 
Dirichlet series
\begin{equation}\label{eq:F}
F(s)=\sum_{q
=1}^\infty q^{-s-3}
\sum_{\substack{a\bmod{q}\\ \gcd(a,q)=1}} 
\sum_{\b \in (\ZZ/q\ZZ)^3}
e_q(af(\b)),
\end{equation}
for $s\in \CC$ and  a given cubic polynomial $f\in \ZZ[x,y,z]$. 
It is easy to see that $F(s)$ is absolutely convergent for $\Re(s)>2$.
In this section we shall relate $F(s)$ to an infinite Euler product involving  the quantities
\begin{equation}\label{eq:nu}
\nu(p^k)=
\#\left\{ \x\in (\ZZ/p^k\ZZ)^3: f(\x)\equiv 0\bmod{p^k}\right\},
\end{equation}
for prime powers $p^k$. 
 The following result is standard but we include its proof for the sake of completeness.

\begin{lemma}\label{lem:F}
Assume that $\Re(s)>2$.
Then  
\[
F(s)=\prod_p \sigma_p(s),
\]
where
\[
\sigma_p(s)=
1+
\sum_{k=1}^\infty \frac{1}{p^{ks}}\left( \frac{\nu(p^k)}{p^{2k}}
-\frac{\nu(p^{k-1})}{p^{2(k-1)}}\right).
\]
If $p\not\in S$ then
\begin{equation}\label{eq:sigmap}
\sigma_p(s)
=1-\frac{1}{p^s}+\frac{\nu(p)}{p^{s+2}}.
\end{equation}
\end{lemma}

\begin{proof}
Since we are working with $s\in \CC$ such that  $\Re(s)>2$, the  infinite sum in $F(s)$
is absolutely  convergent. 
Define 
the exponential sum
$$S_q=
\sum_{\substack{a\bmod{q}\\ \gcd(a,q)=1}} 
\sum_{\b \in (\ZZ/q\ZZ)^3}e_q(af(\b))=
\sum_{\b \in (\ZZ/q\ZZ)^3} c_q(f(\b)),
$$
for $q\in \NN$,
where $c_q(\cdot)$ is the Ramanujan sum. Then $S_q$ is a multiplicative function of $q$ and so we obtain an Euler product
\[
F(s)=\prod_p \sigma_p(s),
\]
where
\[
\sigma_p(s)=\sum_{k=0}^\infty \frac{1}{p^{k(s+3)}} \sum_{\b \bmod{p^k}} c_{p^k}(f(\b)).
\]
Let $a\in \ZZ$. 
At prime powers the Ramanujan sum takes the values 
\begin{equation}\label{eq:ram}
c_{p^k}(a)=
\begin{cases}
0 & \text{ if $p^{k-1}\nmid a$,}\\
-p^{k-1} & \text{ if $p^{k-1}\mid a $ but $p^k\nmid a$,}\\
p^k-p^{k-1} & \text{ if $p^{k}\mid a$.}
\end{cases}
\end{equation}
It follows that 
$$
\sigma_p(s)
=1+
\sum_{k=1}^\infty \frac{ \nu(p^k)-p^{2}\nu(p^{k-1})}{p^{ks+2k}},
$$
as claimed in the first part of the lemma. Moreover, if $U$ is smooth and $p\not\in S$, then $p$ is a prime of good reduction and Hensel's lemma yields $\nu(p^k)=p^{2(k-1)}\nu(p)$ for $k\geq 1$.
The second part  easily follows.
\end{proof}

Lemma~\ref{lem:F} can be used to give a meromorphic continuation of $F(s)$, provided one has  enough information about $\nu(p)$ for large primes $p$. 
 Let $\tilde X_{\bar{\QQ}}=\tX\otimes_{\QQ} \bar \QQ$ and let $\Pic(\tilde X_{\bar{\QQ}})$ be the geometric Picard group of $\tX$. 
The global $L$-function that plays a role here is defined as an Euler product
 \begin{equation}\begin{split}\label{eq:globalL}
 L(s,\Pic(\tilde X_{\bar{\QQ}}))&=\prod_{p<\infty} L_p(s,\Pic(\tilde X_{\bar{\QQ}})) ,\\
 L_p(s,\Pic(\tilde X_{\bar{\QQ}}))&=\det\left(1-p^{-s}\Fr_p\mid (\Pic(\tilde X_{\bar{\QQ}})\otimes \QQ)^{I_p}\right)^{-1},
 \end{split}
 \end{equation}
 where $\Re(s)>1$, $\Fr_p$ is a geometric Frobenius element, and $I_p$ is an inertia subgroup at $p$. Let   $\rho_{\tX}$ be the rank of the Picard group $\Pic(\tX)$.
  Then, as described by Peyre~\cite[Sec.~2.1]{peyre-duke},
$L(s,\Pic(\tilde X_{\bar{\QQ}}))$  is an Artin 
 $L$-function which has a meromorphic continuation to the whole complex plane, with a pole of order $\rho_{\tX}$ at $s=1$.

Bearing this notation in mind, we will need to examine $\nu(p)$ carefully.
Note that
\[
\nu(p)
= \# \mU(\FF_p)
=\#\mtX(\FF_p)-\#\mtD(\FF_p).
\]
For all sufficiently large primes, the Hasse--Weil bound implies that 
\begin{equation}\label{eq:horse-remark}
\frac{\#\mtD(\FF_p)}{p^2}=\frac{b_p(\tD)}{p}+O(p^{-3/2}),
\end{equation}
where $b_p(\tD)$
 is the number of irreducible components 
of $\tD$ of maximal dimension which are fixed under the Frobenius automorphism 
$\Fr_p\in \Gal(\bar\FF_p/\FF_p)$. 
In particular, $0\le b_p(\tD)\le r_{\tD}$,
where
$r_{\tD}$ is the number of irreducible components of $\tD$ as a divisor over $\QQ$. 

Next, as described by Manin~\cite[Thm.~23.1]{M}, 
a result of 
Weil yields
\begin{equation}\label{eq:weil}
\frac{\# \mtX(\FF_p)}{p^2}=1+\frac{a_p(\tX)}{p}+\frac{1}{p^2},
\end{equation}
where $a_p(\tX)$ is the trace of the Frobenius element $\Fr_p$
acting on the Picard group 
$\Pic(\mtX_{\bar \FF_p})$, which is isomorphic to $\Pic( \tilde X_{\bar{\QQ}})$ for almost all primes by \cite[Lem.~2.2.1]{peyre-duke}. We note that
$a_p(\tX)$ is bounded independently of $p$,
by Deligne's resolution of the Weil conjectures.
Hence
\[
\frac{\nu(p)}{p^2} =1+\frac{a_p(\tX)-b_p(\tD)}{p}+O(p^{-3/2}).
\]
Returning to the Dirichlet series $F(s)$, it therefore follows from applying 
this in  Lemma~\ref{lem:F} that 
\begin{equation}\label{eq:goat}
\sigma_p(s)=
1-\frac{1}{p^s}+\frac{\nu(p)}{p^{s+2}}=
1+
\frac{a_p(\tX)-b_p(\tD)}{p^{s+1}}+O(p^{-\Re(s)-3/2}),
\end{equation}
for any $p\not\in S$. 
In particular $\sigma_p(s)=1+O(p^{-\Re(s)-1})$,
for any $p\not\in S$, and so  $F(s)$ is an absolutely convergent Euler product for 
$\Re(s)>0$.

We can relate the analytic properties of $F(s)$ to those of the global 
$L$-function introduced in~\eqref{eq:globalL}.
For  $s\in \CC$ with $\Re(s)>-1/2$ and sufficiently large primes, we find that 
\[
  L_p(s,\Pic(\tilde X_{\bar{\QQ}}))^{-1}=1-\frac{a_p(\tX)}{p^{s}}+\frac{1}{p^{s+1}}.
\]
Since $\Re(s)>-1/2$,  we deduce from \eqref{eq:goat} that 
\[
\sigma_p(s)=
 L_p(s+1,\Pic(\tilde X_{\bar{\QQ}}))
 \left(1-\frac{b_p(\tD)}{p^{s+1}}\right)
  \left(1+O(p^{-\Re(s)-3/2})\right).
\]
Let us define
another Euler product
\begin{align*}
\zeta(s,\tD)&=\prod_{p} \zeta_p(s,\tD), 
\quad 
\zeta_p(s,\tD) =
\begin{cases}
\left(1-\frac{b_p(\tD)}{p^{s}}\right)^{-1} 
&
\text{ if $p>r_{\tD}$,}\\
1 & \text{ otherwise,}
\end{cases}
\end{align*}
for $\Re(s)>1$.
This has a meromorphic continuation to the whole complex plane with a pole  at $s=1$.

In conclusion, our work  shows that there is a function $\tilde F(s)$ which is holomorphic in the half-plane
$\Re(s)>-1/2$, such that
\begin{equation}\label{eq:peter}
F(s)= L(s+1,\Pic( \tX_{\bar{\QQ}}))\zeta(s+1,\tD)^{-1}\tilde F(s).
\end{equation}
An expression like this is essentially implied by work of Chambert-Loir 
and Tschinkel \cite[Thm.~2.5]{ACL} on 
convergence factors on adelic spaces, but we have chosen to include our own 
deduction for the sake of completeness and to  deal explicitly with $F$ as a function in $s$.
In particular, we have  used factors associated with the easier zeta function $\zeta(s+1,\tD)$, rather than  $L(s+1,\CH^0(\tD_{\overline{\QQ}}))$.

\begin{proposition}\label{prop:local-densities-convergence-factors}
Assume that $\mU(\ZZ)$ is Zariski dense.
Then the function $F(s)$ has a meromorphic continuation to the half-plane $\Re(s)>-1/2$ with a singularity at $s=0$ of order $\rho_U$.    Moreover,
letting 
$\sigma_p = \lim_{k\to\infty}p^{-2k}\nu(p^k)$, we have 
\[
  \lim_{s\to 0} \left(s^{\rho_U} F(s)\right) =
  \lambda_0 \prod_{p} \lambda_p \sigma_p,
\]
where
\[
  \lambda_0 = \lim_{s\to 0} s^{\rho_U} \frac{L(s+1,\Pic(\tilde X_{\bar{\QQ}}))}{\zeta(s+1,\tD)}
\quad
\text{  and }
\quad
  \lambda_p = 
  \zeta_p(1,\tD)L_p(1,\Pic(\tilde X_{\bar{\QQ}}))^{-1}.
\]
\end{proposition}

\begin{proof}
Our starting point is
the observation that 
$
\sigma_p(0)=\sigma_p
$
in Lemma \ref{lem:F}. 
Recall that $L(s,\Pic(\tilde X_{\bar{\QQ}}))$ has a pole of order $\rho_{\tX}$ at $s=1$. Moreover, we claim that  $\zeta(s,\tD)$ has a 
meromorphic continuation to the region $\Re(s)>-1/2$
with a
pole of order $r_{\tD}$ at $s=1$,
  where  $r_{\tD}$ is the number of irreducible components of $\tD$ as a divisor over $\QQ$. 
  It follows from 
\eqref{eq:horse-remark} that 
$$
\zeta(s,\tD)=\prod_{p>r_{\tD}} \left(1-
\frac{\#\mtD(\FF_p)}{p^{s+1}}\right)^{-1}\left(1+O(p^{-\sigma-1/2})\right),
$$
where $\sigma=\Re(s)>1$. 
Taking logarithms of both sides, it follows that 
$$
\log \zeta(s,\tD)= \sum_{p>r_{\tD}} \frac{\#\mtD(\FF_p)}{p^{s+1}} +O_\sigma(1),
$$
where the implied constant depends on $\sigma$.
According to work of 
Serre
    \cite[Cor.~7.13]{serre-1}, we have 
$$
\sum_{p\le x} \#\mtD(\FF_p) = \frac{r_{\tD} x}{\log x} \left(1+ O\left(\frac{1}{\log x}\right)\right).
$$
But then, for any $1\leq y<x$, we may combine  this with Abel summation  to deduce that 
$$
\sum_{y<p\leq x} \frac{\#\mtD(\FF_p)}{p^{s+1}} =\frac{(s+1)r_{\tD}}{2}\int_y^x \frac{\d u}{u^s \log u}  +O_\sigma\left(\frac{1}{\log y}\right).
$$
Similarly, it follows from the prime number theorem that 
$$
\int_y^x \frac{\d u}{u^s \log u}   =
 \frac{1}{s}\sum_{y<p\leq x} \frac{1}{p^{s}} +O_\sigma\left(\frac{1}{\log y}\right)
=
-\frac{1}{s}\sum_{y<p\leq x} \log\left(1-\frac{1}{p^{s}}\right) +O_\sigma\left(\frac{1}{\log y}\right),
$$
for $\sigma> 1$. Hence, we obtain
$\zeta(s,\tD)=\zeta(s)^\alpha G(s)$ in the region $\sigma> 1$, where $G(s)$ is holomorphic  
in the region $\sigma>-1/2$ 
and 
$\alpha=(s+1)r_{\tD}/(2s)$.     This therefore establishes the claim.

    It follows from \eqref{eq:peter} that $F(s)$ has a meromorphic 
    continuation to the region $\Re(s)>-1/2$
    with a
   pole of order $\rho_{\tX} - r_{\tD}$ at $s=0$.  
    If the set of integral points on $\mU$ is Zariski dense then,  
    as explained in the proof of \cite[Thm.~2.4.1(ii)]{wilsch}, 
    $U$ cannot have invertible regular functions inducing a relation between the components of $\tD$ in $\Pic \tX$. Thus the left morphism in the localisation sequence $\CH_0(\tD)\to\Pic \tX\to\Pic U\to 0$ is injective and it follows that $ \rho_{\tX} - r_{\tD}=\rho_U$. 
\end{proof}


%
%

\subsection{A smooth $\delta$-function}

We now come to record the version of the  smooth $\delta$-function
that we shall use in our analysis.  Let
$$
\delta(n)=\begin{cases}
1 & \text{ if $n=0$,}\\
0 & \text{ if $n\in \ZZ$ and $n\neq 0$.}
\end{cases}
$$
A  smooth interpretation of this $\delta$-function goes back to work of 
Duke, Friedlander and Iwaniec~\cite{DFI}, but was  developed for Diophantine equations by 
Heath-Brown~\cite[Thm.~1]{HB'}.  The version recorded below is essentially due to 
Tran ~\cite{tran}, but we have elected to reprove it here, since Tran is missing a factor $2$ in his statement.

\begin{proposition}\label{prop:tran}
Let $\Phi:\RR^2\to \RR$ be a Schwartz function satisfying the hypotheses
\begin{enumerate}
\item[(i)]
 $\Phi(-x,-y)=\Phi(x,y)$ for all $x,y\in \RR$,
\item[(ii)] $\Phi(x,0)=0$ for all $x\in \RR$,
\item[(iii)]$\int_{-
\infty}^\infty \Phi(0,y)\d y=1.
$
\end{enumerate}
Then for any $n\in \ZZ$ and sufficiently large $Q$, there exists $c_Q>0$ such that
\[
\delta(n)=
\frac{2c_Q}{Q} \sum_{q=1}^\infty \frac{1}{q} \sum_{a\bmod{q}} 
e\left(\frac{an}{q}\right) 
h\left(\frac{n}{qQ}, \frac{q}{Q}\right),
\]
where $h(x,y)=\Phi(x,y)-\Phi(y,x)$. Moreover, 
$
  c_Q=1+O_N(Q^{-N})
$ 
for any $N\geq 1$.
\end{proposition}

\begin{proof}
Since $n/q$ runs over all divisors of $n$ as $q$ does, we are easily led to the expression
$$
\sum_{\substack{q\in \NN\\ q\mid n}} \left( \Phi\left(\frac{n}{qQ}, \frac{q}{Q}\right) -
\Phi\left(\frac{q}{Q}, \frac{n}{qQ}\right)\right)=\delta(n)\sum_{q\in \NN}\Phi\left(0,\frac{q}{Q}\right),
$$
by (ii). We take $Q$ large enough to ensure  that the point $(0,1/Q)$ is contained in the support of $\Phi$, so that the right hand side doesn't vanish.
It follows from (i) and Poisson summation  that 
$$
\sum_{q\in \NN}\Phi\left(0,\frac{q}{Q}\right)=\frac{1}{2}
\sum_{q\in \ZZ}\Phi\left(0,\frac{q}{Q}\right)=\frac{1}{2}\sum_{c\in \ZZ} \int_{-\infty}^\infty 
\Phi\left(0,\frac{t}{Q}\right) e(-ct) \d t.
$$
Part (iii) 
shows that  the inner integral is $Q$ when $c=0$. On the other hand, repeated integration by parts shows that the  integral is $O_N(Q(Q|c|)^{-N})$ when $c\neq 0$. Defining $c_Q$ via 
$$
\sum_{q\in \NN}\Phi\left(0,\frac{q}{Q}\right)
=c_Q^{-1}\frac{Q}{2},
$$
we deduce that $c_Q=1+O_N(Q^{-N})$ and 
$$
\delta(n)
=\frac{2c_Q}{Q} 
\sum_{\substack{q\in \NN\\ q\mid n}} h\left(\frac{n}{qQ}, \frac{q}{Q}\right).
$$
The proposition follows on using additive characters to detect the condition $q\mid n$.
\end{proof}

The main difference between  Proposition \ref{prop:tran} and the version  in Heath-Brown~\cite[Thm.~1]{HB'} is that one has a sum over all additive characters, rather than just over primitive characters. We note that the function 
\begin{equation}\label{eq:choice}
\Phi(x,y)=\frac{e^{-x^2}(e^{-y^2}-e^{-2y^2})}{\sqrt{\pi}(1-2^{-1/2})}
\end{equation}
is a Schwartz function  that  clearly satisfies the conditions (i)--(iii) in 
Proposition \ref{prop:tran}.

\subsection{Application of the circle method}

Let  $w:\RR^3\to \RR_{\geq 0}$ be a 
compactly supported smooth weight
function. 
Rather than studying $N_U(B)$, we 
shall begin by considering the weighted counting function
\[
N_{U}(B,w)=\sum_{\substack{\x\in \ZZ^3\\ f(\x)=0}} w(B^{-1}\x),
\]
as $B\to \infty$. As is well-known, on assuming a reasonable dependence on $w$
in all the error terms, it is   possible to approximate the characteristic function 
of $[-B,B]^3$ by suitable weight functions to deduce the asymptotic behaviour of $N_U(B)$, as $B\to \infty$.

Let $\Phi:\RR^2\to \RR$ be the function \eqref{eq:choice}, which satisfies the hypotheses (i)--(iii) in Proposition \ref{prop:tran}. Let $h(x,y)=\Phi(x,y)-\Phi(y,x)$.
Then  it follows from this result that 
\[
N_U(B,w)=\frac{2c_Q}{Q} \sum_{q=1}^\infty \frac{1}{q} \sum_{a\bmod{q}} 
\sum_{\x\in \ZZ^3}w(B^{-1}\x)
e_q(af(\x)) 
h\left(\frac{f(\x)}{qQ}, \frac{q}{Q}\right),
\]
where 
$c_Q=1+O_N(Q^{-N})$.
Breaking the sum over $\x$ into residue classes modulo $q$ and applying the $3$-dimensional Poisson summation formula, one readily obtains
\[
N_U(B,w)=\frac{2c_Q}{Q} 
\sum_{\c\in \ZZ^3}
\sum_{q=1}^\infty q^{-4}S_q(\c)
\int_{\RR^3}
w(B^{-1}\x)
h\left(\frac{f(\x)}{qQ}, \frac{q}{Q}\right)e_q(-\c.\x) \d\x,
\]
where
\[
S_q(\c)=
\sum_{a\bmod{q}} \sum_{\b \bmod{q}}
e_q(af(\b)+\c.\b) .
\]
Since $h(x,y)$ is a  Schwartz function, we expect that only 
$q,\x$ with 
$q\ll Q$  and $f(\x)\ll Q^2$ make a dominant contribution.
Moreover, the integrand is zero unless 
$w(B^{-1}\x)\neq 0$ and it can be shown that $|f(\x)|$ has exact order of magnitude $B^3$ for  typical such  $\x$. In this way we are led to make the choice $Q=B^{3/2}$ in our analysis. 
(In fact, one can  take $Q=cB^{3/2}$ for any constant $c>0$ without affecting  
the heuristic main term, while taking $Q=B^\theta$ for  $\theta>\frac{3}{2}$ would 
cause   problems in the analysis of the oscillatory integral.)

Our circle method heuristic arises from 
asymptotically evaluating the contribution from 
the trivial character, corresponding to $\c=\0$. (In fact, there is evidence to suggest 
that the contribution from possible accumulating subvarieties is  accounted for by  the non-trivial characters, as discussed by Heath-Brown~\cite{royal} for  diagonal cubic surfaces in $\PP^3$.)
For our heuristic, we shall take $\c=\0$ and $c_Q=1$, leaving us to  estimate 
\begin{equation}\label{eq:M-step1}
M(B,w)=\frac{2 }{Q} \sum_{q
=1}^\infty q^{-4}S_q(\0)
\int_{\RR^3}
w(B^{-1}\x)
h\left(\frac{f(\x)}{qQ}, \frac{q}{Q}\right)\d\x,
\end{equation}
with  $Q=B^{3/2}$.
This will eventually be achieved in Theorem~\ref{thm:trivial-character}.

Let 
$$D(s)=\sum_{q
=1}^\infty q^{-s-4}S_q(\0),
$$
for $s\in \CC$. In view of the trivial bound $|S_q(\0)|\leq q^4$, this is absolutely convergent for $\Re(s)>1$.  We proceed by proving the following result. 

\begin{lemma}\label{lem:Ds}
Let $\Re(s)>1$. Then we have 
\[
D(s)=F(s+1)\zeta(s+1),
\]
where  $F(s)$ is given by~\eqref{eq:F}.
\end{lemma}

\begin{proof}
Let $s\in\CC$ such that $\Re(s)>1$, so that  $D(s)$ is absolutely convergent. 
Breaking the sum according to the greatest common divisor of $a$ and $q$, we find 
that 
\[
D(s)=\sum_{q=1}^\infty q^{-s-4}\sum_{r\mid q} 
\sum_{\substack{a\bmod{q}\\ \gcd(a,q)=r}} 
\sum_{\b \bmod{q}}
e_q(af(\b)).
\]
Making the change of variables $q=rq'$ and $a=ra'$, one  concludes that 
\[
D(s)=\sum_{q'=1}^\infty\sum_{r=1}^\infty {q'}^{-s-4}r^{-s-4}
\sum_{\substack{a'\bmod{q'}\\ \gcd(a',q')=1}} 
r^3
\sum_{\b \bmod{q'}}
e_{q'}(a'f(\b)).
\]
The statement of the lemma is now obvious. 
\end{proof}

An application of the  Mellin inversion theorem yields
\begin{align*}
&\frac{1}{2\pi i}\int_{2-i \infty}^{2+i\infty} D(s) \int_0^\infty \int_{\RR^3}
w(B^{-1}\x)
h\left(\frac{f(\x)}{yQ}, \frac{y}{Q}\right)\d\x y^{s-1} \d y \d s \\
&=\sum_{q=1}^\infty q^{-4}S_q(\0) \cdot 
\frac{1}{2\pi i} \int_{2-i \infty}^{2+i\infty}  \left(\int_0^\infty \int_{\RR^3}
w(B^{-1}\x)
h\left(\frac{f(\x)}{yQ}, \frac{y}{Q}\right)y^{s-1}\d\x \d y \right)\frac{\d s}{q^s} \\
&=\sum_{q=1}^\infty q^{-4}S_q(\0) 
 \int_{\RR^3}
w(B^{-1}\x)
h\left(\frac{f(\x)}{qQ}, \frac{q}{Q}\right)\d\x,
\end{align*}
which we recognise as appearing in our expression for $M(B,w)$. 
Replacing $y/Q$ by $y$ 
and $\x$ by $B\x$,
we obtain
\begin{align*}
\int_0^\infty \int_{\RR^3}
&w(B^{-1}\x)
h\left(\frac{f(\x)}{yQ}, \frac{y}{Q}\right)y^{s-1}\d\x \d y \\
&=B^{3}Q^s
\int_0^\infty \int_{\RR^3}
w(\x)
h\left(\frac{g(\x)}{y}, y\right)y^{s-1}\d\x \d y ,
\end{align*}
where
$g(\x)=Q^{-2}f(B\x)$. (Note that this coincides with the definition \eqref{eq:def-g}, since  $Q=B^{3/2}$.)
Returning to~\eqref{eq:M-step1}, 
it follows from Lemma \ref{lem:Ds} that 
\begin{equation}\label{eq:main}
M(B,w)=\frac{2Q^{-1}B^3 }{2\pi i} 
\int_{2-i \infty}^{2+i\infty} F(s+1)\zeta(s+1) G(s) Q^s\d s ,
\end{equation}
where
$$
G(s)=
\int_0^\infty \int_{\RR^3}
w(\x)
h\left(\frac{g(\x)}{y}, y\right)y^{s-1}\d\x \d y.
$$
We seek to obtain a meromorphic continuation of $G(s)$ sufficiently far to  the left of the line $\Re(s)=2$.

Let 
\begin{equation}\label{eq:def-katy}
k_t(y)=\int_{-\infty}^\infty h(x,y) e(-txy)\d x,
\end{equation}
for any $t,y\in \RR$.  The following result evaluates this integral.
 
\begin{lemma}\label{lem:hatk}
We have 
$$
k_t(y)=\frac{1}{1-2^{-1/2}}
\left(2^{-1/2}e^{-\frac{1}{2}y^2(2+\pi^2t^2)}-
e^{-y^2(2+\pi^2t^2)}\right).
$$
In particular $k_{\ve_1 t}(\ve_2 y)=k_{t}(y)$ for $\ve_1,\ve_2\in \{\pm 1\}$.
\end{lemma}

\begin{proof}
On recalling that $h(x,y)=\Phi(x,y)-\Phi(y,x)$, in the notation of  \eqref{eq:choice}, it follows that 
$k_t(y)=M(t,y)-M^{\text{sw}}(t,y)$
where
$$
M(t,y)=
\int_{-\infty}^\infty \Phi(x,y) e(-txy)\d x \quad \text{ and }
\quad 
M^{\text{sw}}(t,y)=\int_{-\infty}^\infty \Phi(y,x) e(-txy)\d x.
$$
It will be convenient to put $C=\sqrt{\pi}(1-2^{-1/2})$ in the proof, so that we may write $\Phi(x,y)=C^{-1}e^{-x^2}(e^{-y^2}-e^{-2y^2})$. 
Then 
\begin{align*}
k_t(y)
&=
\frac{1}{C}\left(-e^{-2y^2}
\int_{-\infty}^\infty e^{-x^2}e(-txy)\d x +
e^{-y^2}
\int_{-\infty}^\infty e^{-2x^2}e(-txy)\d x \right)\\
&=
\frac{\sqrt{\pi}}{C}\left(-e^{-y^2(2+\pi^2t^2)}+2^{-1/2}
e^{-\frac{1}{2}y^2(2+\pi^2t^2)}
 \right),
\end{align*}
on completing the square and executing the integral over $x$. Substituting in the value of $C$ completes the proof of the lemma.
\end{proof}

We may now assess the analytic properties of the function $G(s)$, in which it will be convenient
to recall the 
 definition \eqref{eq:def-I} of $I(t)$.

\begin{lemma}\label{lem:GGGG}
Assume that
Hypothesis \ref{hyp} holds.  Then 
$G(s)$ has a meromorphic continuation to the region $\Re(s)\geq -1$ with a simple pole at $s=-1$. 
Moreover, in this region we have 
\begin{equation}\label{eq:R}
G(s)=\frac{1}{2}
\cdot
\frac{2^{s/2}-1}{1-2^{-1/2}}\cdot \Gamma\left(\frac{s+1}{2}\right)R(s),
\end{equation}
where
\begin{equation}\label{eq:R(s)}
R(s)=
\int_{-\infty}^\infty  \frac{I(t) }{(2+\pi^2t^2)^{\frac{s+1}{2}}}
\d t
\end{equation}
is holomorphic.
\end{lemma}

\begin{proof}
Let 
 $\sigma=\Re(s)>2$.
In this region we have 
$$
G(s)=
\int_0^\infty \int_{\RR^3}
w(\x)
h\left(\frac{g(\x)}{y}, y\right)y^{s-1}\d\x \d y.
$$
We denote by
\[
H(t,y)=\int_{-\infty}^\infty h(x,y) e(-tx)\d x
\]
the Fourier transform of $h(x,y)$ with respect to the first variable.
Then  
the
 Fourier inversion theorem yields
\begin{align*}
\int_{\RR^3}
w(\x)
h\left(\frac{g(\x)}{y}, y\right)\d\x
 &=
\int_{-\infty}^\infty
H(t,y)\left(
\int_{\RR^3}
w(\x) e\left(\frac{tg(\x)}{y}\right)  \d\x \right)\d t  \\
&=
y
\int_{-\infty}^\infty H(yt,y) I(t) 
 \d t ,
\end{align*}
on  replacing $t/y$ by $t$ and recalling the definition \eqref{eq:def-I} of $I(t)$.
But $H(yt,y)=k_t(y)$, in the notation of \eqref{eq:def-katy}.
Hence it follows from Lemma \ref{lem:hatk} that 
$$
\int_{\RR^3}
w(\x)
h\left(\frac{g(\x)}{y}, y\right)\d\x
=
\frac{y}{1-2^{-1/2}}
\sum_{i\in \{0,1\}}
\frac{(-1)^{i+1}}{2^{i/2}}
\int_{-\infty}^\infty
I(t)e^{-\frac{1}{2^i}y^2(2+\pi^2t^2)}
\d t.
$$
Thus 
$$
G(s)=
\frac{1}{1-2^{-1/2}}
\sum_{i\in \{0,1\}}
\frac{(-1)^{i+1}}{2^{i/2}}
\int_0^\infty y^{s}
\int_{-\infty}^\infty
I(t)e^{-\frac{1}{2^i}y^2(2+\pi^2t^2)}
\d t \d y.
$$

Taking  absolute values we see that  
\begin{align*}
G(s)
\ll
\int_0^{\infty} y^{\sigma} e^{-y^2}J(y)
\d y,
\end{align*}
where
$$
J(y)
=\int_{-\infty}^{\infty} | I(t) |e^{-\frac{1}{2}\pi^2t^2y^2} \d t.
$$
If $|ty|\leq 1$ then we take $e^{-\frac{1}{2}\pi^2t^2y^2}\ll 1$ and it follows from \eqref{eq:trivial} that 
$$
\int_{|t|\leq 1/y} | I(t) |e^{-\frac{1}{2}\pi^2t^2y^2} \d t\ll y^{-1}.
$$
When $|ty|\geq 1$ we may take  $e^{-\frac{1}{2}\pi^2t^2y^2}\ll |ty|^{-1/2}$, and it follows from 
part (i) of Hypothesis \ref{hyp} that 
$$
\int_{|t|> 1/y} | I(t) |e^{-\frac{1}{2}\pi^2t^2y^2} \d t\ll  y^{-1/2}
\int_{-\infty}^{\infty} |t|^{-1/2}| I(t) |\d t\ll y^{-1/2}.
$$
Hence
$$
G(s)\ll  \int_0^\infty y^{\sigma-1}e^{-y^2} \d y+   \int_0^\infty y^{\sigma-1/2}e^{-y^2} \d y,
$$
which is bounded  for $\sigma>0$. Thus  $G(s)$ is 
absolutely convergent  in the region $\sigma>0$.

Working in the region $\sigma>0$,
an application of  Fubini's theorem allows us to  interchange the order of integration. This leads to the expression 
\[
G(s)=
\frac{1}{1-2^{-1/2}}
\sum_{i\in \{0,1\}}
\frac{(-1)^{i+1}}{2^{i/2}}
\int_{-\infty}^\infty  I(t) 
J_i(t,s)\d t,
\]
where
\begin{align*}
J_i(t,s)
&=\int_0^\infty y^{s}  e^{-\frac{1}{2^i}y^2(2+\pi^2t^2)}
 \d y\\ 
 &= \frac{2^{i(s+1)/2}}{(2+\pi^2t^2)^{\frac{s+1}{2}}} 
 \int_0^\infty y^{s}  e^{-y^2} \d y\\
 &= \frac{2^{i(s+1)/2}}{(2+\pi^2t^2)^{\frac{s+1}{2}}} \cdot \frac{1}{2}\cdot \Gamma\left(\frac{s+1}{2}\right).
 \end{align*}
We may now execute the sum over $i\in \{0,1\}$ and finally arrive at the expression for $G(s)$ recorded in the statement of the lemma.   Since part (ii) of Hypothesis~\ref{hyp} ensures that 
$R(s)$  is holomorphic in the region $\sigma\geq -1$, this gives the desired meromorphic continuation of $G(s)$ to the region $\sigma\geq -1$.
\end{proof}

We will need to understand the  derivatives of 
the function \eqref{eq:R(s)}
in the region $\Re(s)\geq -1$. 
Let $R^{(\ell)}(s)$ be the $\ell$th derivative with respect to $s$, for any integer $\ell\geq 0$. Then it follows that 
$$
R^{(\ell)}(s)
= \frac{(-1)^{\ell}  }{2^\ell} 
\int_{-\infty}^\infty  \frac{I(t) (\log (2+\pi^2t^2))^\ell}{(2+\pi^2t^2)^{\frac{s+1}{2}}}
\d t,
$$
whence
\begin{align*}
R^{(\ell)}(-1)
= 
(-1)^{\ell}  \int_{-\infty}^\infty I(t)(\log \sqrt{2+\pi^2t^2})^\ell \d t.
\end{align*}
In the light of \eqref{eq:trivial}, 
the interval $[-2,2]$ contributes $O(1)$ to the integral. On the other hand, when $|t|\geq 2$ we have 
$$
\log  \sqrt{2+\pi^2t^2} = \log |t| +\log \pi +O(|t|^{-2}),
$$
since $\log(1+\frac{2}{\pi^2t^2})=O(|t|^{-2})$. 
According to part (i) of Hypothesis \ref{hyp}, we conclude that 
\begin{align*}
R^{(\ell)}(-1)
= 
(-1)^{\ell}  \int_{-\infty}^\infty I(t)(\log |t|+\log\pi)^{\ell}  \d t
+O_\ell(1).
\end{align*}
The following result summarises our analysis of this function.

\begin{lemma}\label{cor:0.3}
Assume Hypothesis \ref{hyp} and 
let $\ell\geq 0$ be an integer. Then there 
is a monic degree $\ell$ polynomial $P\in \RR[x]$ 
such that 
\begin{align*}
R^{(\ell)}(-1)
=~& (-1)^{\ell} 
\int_{-\infty}^{\infty} I(t) P(\log |t|) \d t +
O_\ell(1).
\end{align*}
\end{lemma}

\subsection{Conclusions and heuristics}
 
It is now time to return to our expression~\eqref{eq:main} for $M(B,w)$, in order to 
record our main circle method heuristic. 
Assume  Hypothesis~\ref{hyp} holds.
If  $Q=B^{3/2}$ then 
$Q^{-2}B^3=1$ and we find that 
\[
M(B,w)=2 \cdot \frac{1 }{2\pi i} 
\int_{2-i \infty}^{2+i\infty} F(s+1) \zeta(s+1)G(s) Q^{s+1}\d s .
\]
Proposition \ref{prop:local-densities-convergence-factors} implies 
that  $F(s+1)$ has a meromorphic continuation to the region $\Re(s)>-\frac{3}{2}$ with a pole of order $\rho_U$ at $s=-1$.    
Moreover, Lemma \ref{lem:GGGG}
implies that $G(s)$ has a meromorphic continuation to the region $\Re(s)\geq -1$ with a simple pole at $s=-1$.  In addition to this, it follows from \eqref{eq:R} that $G(0)=0$, so that the integrand is holomorphic at $s=0$. 
Overall, we conclude that in the region $\Re(s)\geq -1$ the function 
$F(s+1)  \zeta(s+1)G(s) Q^s$ has a pole of order $\rho_U+1$ at $s=-1$ and is holomorphic everywhere else. 

In the usual way the asymptotic behaviour of $M(B,w)$ is obtained
by moving the line of integration to the left in order to capture the pole at $s=-1$.  
We shall not delve into details here, but content ourselves with recording the expected  asymptotic formula
\begin{align*}
M(B,w)
&\sim
2\cdot 
\Res_{s=-1}  \left(F(s+1) \zeta(s+1)G(s)Q^{s+1}\right)\\
&=
\Res_{s=0}\left( 
\frac{2^{(s-1)/2}-1}{1-2^{-1/2}}
F(s)\zeta(s)\Gamma\left(\frac{s}{2}\right)R(s-1)Q^s\right),
\end{align*}
on making the substitution  \eqref{eq:R}.

We recall that 
we have 
$
F(s)=s^{-\rho_U}\tilde F(s),
$
for some function $\tilde F(s)$ which is holomorphic for $\Re(s)> -\frac{1}{2}$. 
Moreover, we have $\Gamma(\frac{s}{2})=s^{-1}(2+O(s))$. Let 
$$
U(s)=
\frac{2^{(s-1)/2}-1}{1-2^{-1/2}}
\zeta(s)R(s-1)Q^s.
$$
This is holomorphic for $\Re(s)\geq 0$.
Taking the Taylor expansion about the point $s=0$, we obtain
$$
U(s)=U(0)+
\frac{U'(0)}{1!} s +\frac{U''(0)}{2!} s^2 +\cdots +
\frac{U^{(\rho_U)}(0)}{\rho_U!} s^{\rho_U}+ O(s^{\rho_U+1}).
$$
It therefore follows that 
$$
\Res_{s=0}\left( U(s)
F(s)\zeta(s)\Gamma\left(\frac{s}{2}\right)\right)=\frac{2\tilde F(0)U^{(\rho_U)}(0)}{\rho_U!}
+
O\left(\max_{0\leq \ell\leq \rho_U-1} |U^{(\ell)}(0)|\right).
$$
At this point it is convenient to make another assumption about the asymptotic behaviour of the integral $I(t)$.

\begin{hyp}\label{hyp'}
Let $I(t)$ be given by \eqref{eq:def-I} and define
\begin{equation}\label{qq}
J_\ell(B)=\int_{-\infty}^\infty I(t) (\log |t|)^\ell \d t,
\end{equation}
for $\ell\geq 0$.
Then 
$
J_\ell(B) \ll_{\ell} (\log B)^{b+\ell},
$
where 
$b$
is  defined in Conjecture \ref{con1} and 
the implied constant is  allowed to depend  on $w, f$ and $\ell$.
\end{hyp}

  Under Hypotheses~\ref{hyp} and \ref{hyp'},
it is clear from Lemma \ref{cor:0.3} and the Leibniz rule  that 
\begin{align*}
U^{(\ell)}(0)
&=-\zeta(0) \sum_{j=0}^{\ell} \binom{\ell}{j}
R^{(j)}(-1)(\log Q)^{\ell-j}
 +
O\left((\log B)^{b+\ell-1}\right)\\
&=-\zeta(0) \cdot \ell!\sum_{j=0}^{\ell} \frac{(-1)^jJ_j(B)(\log Q)^{\ell-j} }{j!(\ell-j)!}+
O\left((\log B)^{b+\ell-1}\right),
\end{align*}
for any integer $\ell\geq 0$.  Since $\zeta(0)=-\frac{1}{2}$ and $Q=B^{3/2}$, we therefore  deduce the following result.  

\begin{theorem}\label{thm:trivial-character}
  Under Hypotheses~\ref{hyp} and \ref{hyp'}, the contribution from the trivial character is
  \[
  M(B,w) =   \lim_{s\to 0} \left(s^{\rho_U} F(s)\right)  \cdot 
r(B)  
  +O((\log B)^{\rho_U+b-1}),
    \]
    where if $J_j(B)$ is given by  \eqref{qq} then 
$$
r(B)=\sum_{j=0}^{\rho_U}
\frac{ (-1)^j\left(\frac{3}{2}\right)^{\rho_U-j} J_j(B)(\log B)^{\rho_U-j}}{j!(\rho_U-j)!}.
$$
\end{theorem}

According to  Hypothesis  \ref{hyp'}, we have $r(B)=O((\log B)^{\rho_U+b})$, which  therefore 
accords with  Conjecture \ref{con1}. 
When $\rho_U>0$ the sum $r(B)$ features multiple terms, some of which have negative coefficients, but all with seemingly equal  order of magnitude. This is very different to classical applications of the circle method. 
When $\rho_U=0$, however,  the contribution from the trivial character is more straightforward. Thus, 
under Hypotheses~\ref{hyp} and \ref{hyp'}, we obtain 
  \[
  M(B,w) =  F(0)  \cdot J_0(B)
  +O((\log B)^{b-1}),
    \]
    where
 $J_0(B)$ is given by  \eqref{qq}. 
 In particular, we  have $J_0(B)=\sigma_\infty(B)$, in the notation of \eqref{eq:infinity'}.
  It follows from 
 Proposition~\ref{prop:local-densities-convergence-factors} that 
 $F(0)=\prod_p \sigma_p$, where $\sigma_p$ are the local densities. 
For our heuristic we shall suppose that the characteristic function of  the region $[-B,B]^3$ is approximated by an appropriate compactly supported smooth weight function $w$. This leads to  the following expectation.

\begin{heuristic}\label{heur:pre}
  Let $U\subset \AAA^3$ be a smooth cubic surface that is log K3 over $\QQ$ and that is
     defined by a cubic polynomial $f\in \ZZ[x,y,z]$.  Assume that $\rho_U=0$.
Then 
$$
N_U(B)\sim 
 \prod_{p}  \sigma_p
 \cdot \sigma_\infty(B),
 $$
as $B\to \infty$, 
where $\sigma_\infty(B)$ is given by  \eqref{eq:infinity'}.
\end{heuristic}

We may further refine this by supposing that Hypothesis \ref{hyp:garlic} holds.
Combining Heuristic \ref{heur:pre} with 
Proposition~\ref{prop:arch-volume-expansion}, we are therefore led to the following expectation.

\begin{heuristic}\label{heur:circle-method}
  Let $U\subset \AAA^3$ be a smooth cubic surface that is log K3 over $\QQ$  and that is
     defined by a cubic polynomial $f\in \ZZ[x,y,z]$.  Assume that 
     $U(\ZZ)$ is Zariski dense and
     $\rho_U=0$.
Then 
  \[
    N_{U}^\circ(B)\sim
    c_\infty \prod_{p}  \sigma_p \cdot
        (\log B)^{b},
  \]
  as $B\to \infty$,
  where  $c_\infty$ is given by \eqref{eq:arch-constant-abstract}.
  \end{heuristic}

Later, in Section \ref{s:gen_cubes}, we shall return to the heuristic suggested by 
Theorem \ref{thm:trivial-character} for some explicit cubic surfaces $U\subset \AAA^3$ 
with  $\rho_U>0$. 
  It remains to offer some justification for Heuristic \ref{main-heur}, which is concerned with arbitrary
  $\rho_U\geq 0$.

\subsubsection*{Analogy with Manin's conjecture}

It is natural to draw comparisons with Manin's conjecture, which concerns the distribution of $\QQ$-rational points on smooth, projective, Fano varieties $V$ defined over $\QQ$.  A value for the leading constant in this conjecture has been 
suggested by Peyre \cite{peyre-duke,peyre-bordeaux}.   Let $H:V(\QQ)\to \RR_{\geq 0}$ be an anticanonical height function and define the counting function
$
N(Z;B)=\#\left\{P\in Z: H(P)\leq B\right\},
$ 
for any subset $Z\subset V(\QQ)$. Then, as put forward in  
\cite[Sec.~8]{peyre-bordeaux}, we expect there  to
exist a {\em thin set} $\Omega\subset V(\QQ)$ for which
\[
N(V(\QQ)\setminus \Omega;B)\sim c_V B (\log B)^{\rho_V-1}, 
\]
as $B\to \infty$. (Note that rational points are much more prolific for Fano varieties than integral points are expected to be in the setting of log K3 surfaces.)
The conjectured leading constant has the structure
\begin{equation}\label{eq:sky}
c_V= \alpha_V \cdot  \beta_V \cdot \tau_{V,H}\left(V(\A_\QQ)^{\Br V}\right), 
\end{equation}
where $V(\A_\QQ)^{\Br V}$ is the set  of adeles on $V$ that are orthogonal
to the Brauer--Manin pairing, 
 and $\tau_{V,H}$ is the {\em Tamagawa measure} defined in \cite[Sec.~2]{peyre-duke}. The constants 
$\alpha_V$ and $\beta_V$ are rational numbers; 
the latter is  the order of  the Brauer group $\Br(V)/\Br(\QQ)$   and the former 
is the volume of a certain polytope in the dual of the effective cone of divisors,
as defined in
 \cite[Déf.~2.4]{peyre-duke}.
In particular, if the Picard group is trivial and the height function is associated with $q$ times a generator, then $\alpha_V=1/q$.

Suppose that $V\subset \PP^{n}$ is a smooth complete intersection of $r$ degree $d$ 
hypersurfaces, 
with $n\geq r+2^{d-1}(d-1)r(r+1)$. 
Then it follows from 
 work of  Birch \cite{birch}, which is proved using  the circle method, that an asymptotic formula is available for 
$N(V(\QQ);B)$, with $\alpha_V=1/(n+1-rd)$ and $\beta_V=1$,  and  where
$\tau_{V,H}(V(\A_\QQ)^{\Br V})$ is the usual product of local densities. 
However, there exist many Fano varieties $V$ for which the full Manin--Peyre conjecture holds 
with $\beta_V\neq 1$ or with a more complicated expression for $\alpha_V$.
An example of the latter is provided by  the smooth quartic  del Pezzo surface (corresponding to $(d,n,r)=(2,4,2)$ in the above notation)
studied by de la Bret\`eche and Browning \cite{dp4}, in which an asymptotic formula is obtained with $\alpha_V=1/36$ and $\beta_V=1$.

\subsubsection*{A refined heuristic}

Returning to the setting of log K3 surfaces $U\subset \AAA^3$, as in Conjecture \ref{con1}, 
we have seen that 
Theorem \ref{thm:trivial-character} suggests an asymptotic behaviour $N_U^\circ(B)\sim c_{\mathrm{h}} (\log B)^{\rho_U + b}$, for a suitable constant $c_{\mathrm{h}}$. Inspired by our discussion of Peyre's constant, we have been led to modify the product of local densities to 
account for failures of strong approximation, in addition to allowing for  
an unspecified positive rational factor. This has led us to the value for 
$c_{\mathrm{h}}$ proposed in 
\eqref{eq:leading-const-with-gamma}, and thereby  concludes our discussion of Heuristic   \ref{main-heur}.

\subsection{Singularities on $U$}\label{eq:lied}
We  conclude by briefly remarking on the case of singular $U$.  In this case, we still have
 $ \sigma_p= \lim_{k\to\infty} p^{-2k}\nu(p^k)$
and we  proceed by comparing this quantity to the analogous one defined on a minimal desingularisation. To this end, let $\rho\colon \mtX\to \mX$ and $\tX\to X$ be  minimal desingularisations as before, but now without $\rho$ being an isomorphism above $U$ and without the requirement that $D$ have strict normal crossings. Define
$\tilde\nu(p) = \#\mtU(\FF_p)$.

\begin{lemma}
  There is a finite set $S$ of places (containing the archimedean one and those of bad reduction of $\mtX$) such that
  \[
    \lim_{k\to\infty}\frac{\nu(p^k)}{p^{2k}} = \frac{\tilde\nu(p)}{p^2}
  \]
  for all $p\not\in S$.
\end{lemma}
\begin{proof}
  The minimal desingularisation is crepant, so that  $\rho^*\omega_X \cong \omega_{\tX}$. Moreover, this isomorphism spreads out to an isomorphism between $\rho^*\omega_{\mX}$ and $\omega_{\mtX}$, except possibly above a finite set of places. 
  Let $S$ be the union of these places, the places of bad reduction of $\mtX$, and the archimedean place. Equip $\omega_X$ with a metric that is the model metric outside $S$,  and $\omega_{\tX}$ with the pullback metric. After possibly enlarging $S$, this pullback is the model metric outside $S$.
  Let $p\not\in S$ be a prime, and denote by $\tau_p$ and $\tilde\tau_p$ the resulting Tamagawa measures on $X(\QQ_p)$ and $\tX(\QQ_p)$, respectively,  satisfying $\tau_p = \rho_*\tilde\tau_p$.  (Their definitions coincide on $X_{\reg}\cong \rho^{-1}X_{\reg}$.)

  We construct a countable disjoint covering $\cB$ of $\mU(\ZZ_p) \cap X_\reg(\QQ_p)$ as follows.
  Let $\x\in \mU(\ZZ_p)\cap X_\reg(\QQ_p)$ be a $\ZZ_p$-point whose generic point is regular. Denote by $e_\x$ the minimal power of $p$ annihilating the torsion of $H^0(\Spec \ZZ_p, \x^*\Omega_{\mU/\ZZ_p})$.
  Define
  \[ 
    U_\x = \{ \x'\in\mU(\ZZ_p) : \x' \equiv \x \bmod p^{e_\x}\}.
  \]
  For fixed $E$, those $\x$ with $e_\x\le E$ form a finite set of balls $\cB_E = \{U_{\x_1},\dots U_{\x_{s_E}}\}$; let $\cB=\bigcup_{E=1}^\infty \cB_E$ be their union.
  
  For each of the $U_\x$, the arguments used by Salberger~\cite[Thm.~2.13]{salberger} are applicable and show that
  \[
    \tau_p(U_\x) = p^{-2e_\x} = \frac{\#\lim_{l\to\infty} \#\{\x'\in\mU(\ZZ/p^l\ZZ) : \x'\equiv \x \bmod p^{e_\x}\}}{p^{2l}}.
  \]
  As $l$ grows, 
apart from the at most  finitely many singularities, 
    eventually every point modulo $p^l$ is counted this way; it follows that
  \[
    \tau_p(\mU(\ZZ_p) \cap X_\reg(\QQ_p)) = \sum_{U_\x\in \cB} \tau_p(U_\x)
    = \lim_{l\to \infty} \frac{\nu(p^{l})}{p^{2l}}.
  \]
According to~\cite[Cor.~2.15]{salberger}, we also get $\tilde\tau_p(\mtU(\ZZ_p)) = \tilde\nu(p)/p^2$. 
  Moreover, $\rho$ restricts to a measure preserving homeomorphism $\rho^{-1}X_\reg(\QQ_p) \to X_\reg(\QQ_p)$ by construction of the Tamagawa measures. As the complement $\rho^{-1}X_\sing(\QQ_p)$ of the former set is a null set, we obtain
  \begin{align*}
    \frac{\tilde\nu(p)}{p^2}
    = \tilde\tau_p(\mtU(\ZZ_p) \cap \rho^{-1}X_\reg(\QQ_p)) 
    = \tau_p(\mU(\ZZ_p)\cap X_\reg(\QQ_p)) 
     = \lim_{l\to\infty}\frac{\nu(p^l)}{p^{2l}},
  \end{align*}
  as claimed.
\end{proof}
As a consequence of this result, we deduce that 
\[
\prod_{p\not\in S}
  \frac{L_p(1,\Pic(\tilde X_{\bar{\QQ}}))}{\zeta(1,\tilde D)}\sigma_p =\prod_{p\not\in S} \frac{L_p(1,\Pic( \tilde{X}_{\bar{\QQ}}))}{\zeta(1,\tilde D)} \frac{\tilde\nu(p)}{p^2}
\]
is absolutely convergent, suggesting an asymptotic behaviour 
\[
N_U^\circ(B)\sim
 c (\log B)^{\rho_{\tU} + b},  
\]
for a suitable constant  $c$.

\section{Norm form equations}\label{s:norm}

Let $K/\QQ$ be a cubic number field and let $k\in \ZZ$ be non-zero. Let $\mU\subset \AAA_\ZZ^3$ be the smooth cubic surface defined by the polynomial $f(x,y,z)=\mathrm{N}_{K/\QQ}(x,y,z)-k$, where $\mathrm{N}_{K/\QQ}(x,y,z)$ is the  {\em norm form} associated to the number field, and let $U$ be its generic fibre.  
Since $U_{\QQbar}$ is a torus over $\bar \QQ$, it is an open
subset of affine space over $\bar \QQ$.  
The Picard group of affine space vanishes and so it follows that the  geometric Picard group of $U$ vanishes. Thus $\rho_U=0$, since the Picard group $\Pic (U)$ is a subgroup of
the geometric Picard group.

We proceed by showing that the exponent $r-1$ in \eqref{eq:lgw} agrees with the exponent of 
$\log B$ in 
Conjecture \ref{con1} and 
Heuristic \ref{heur:circle-method}. Observe that the divisor $D$ is given by 
$V(\mathrm{N}_{K/\QQ}(x_0,y_0,z_0))$.  If $K$ is totally real, then $b=2$. On the other hand, if $K$ has a complex embedding, then $b=1$. 
Thus $b=r-1$, and so the exponents of $\log B$ do indeed match.
We do not expect the leading constant in  
\eqref{eq:lgw} to agree with the leading constant in 
Heuristic \ref{heur:circle-method}, since   $U(\ZZ)$ is a thin set.

\section{Sums of  cubes: rank zero}\label{s:cubes}

In this section, we specialise to the smooth cubic surface $U\subset \AAA^3$ defined by the polynomial 
$f(x,y,z)=x^3+y^3+z^3-k$,  for an integer $k\not\equiv \pm 4\bmod{9}$.  
Our main aim is to check that 
Heuristic \ref{heur:circle-method} aligns with the prediction worked out by Heath-Brown~\cite[p.~622]{33} when $k$ is cube-free. 

The compactification $X\subset \PP^3$ is the smooth cubic surface defined by the polynomial 
$
f_0=x_0^3+y_0^3+z_0^3-kt_0^3.
$
The divisor $D$ is the  smooth genus $1$ curve 
$V(x_0^3+y_0^3+z_0^3)$.
In particular, $D$  is geometrically irreducible and we have  $b=1$ in 
the notation of Conjecture \ref{con1}.
We claim that 
\begin{equation}\label{eq:onion}
\rho_U=\begin{cases}
0 & \text{ if $k$ is not a cube,}\\
3 & \text{ if $k$ is a cube.}
\end{cases}
\end{equation}
Since $\rho_U=\rho_X-1$, it will suffice to calculate $\rho_X$. 
When $k$ is a cube the surface $X$ is $\QQ$-isomorphic to the surface 
$x_0^3+y_0^3+z_0^3+t_0^3=0$. But then it follows from \cite[Prop.~6.1]{pt2}
that $\rho_X=4$. 
When $k$ is not a cube it follows from work of Segre \cite[Thm.~IX]{segre} that $\rho_X=1$.
This establishes the claim.

\subsection{Non-archimedean densities}

For any prime $p$ the relevant $p$-adic density is 
$
   \sigma_p = \lim_{\ell\to\infty}p^{-2\ell}\nu(p^\ell), 
 $
 where
 $$
\nu(p^\ell)=  \#\{(x,y,z)\in (\ZZ/p^\ell\ZZ)^3: x^3+y^3+z^3\equiv k \bmod{p^\ell}\}.
$$
Heath-Brown \cite[p.~622]{33} has calculated these explicitly when $k$ is cube-free,  beginning with 
\begin{equation}\label{eq:K1}
\sigma_3=\frac{\nu(27)}{27^2}.
\end{equation}
Moreover, 
\begin{equation}\label{eq:K2}
\sigma_p=\begin{cases}
1 &\text{ if $p\equiv 2\bmod{3}$ and $p\nmid k$,}\\
1-\frac{1}{p^2} &\text{ if $p\equiv 2\bmod{3}$ and $p\mid k$.}
\end{cases}
\end{equation}
On the other hand, if $p\equiv 1\bmod{3}$,
let  $a_p,b_p$ the unique choice of integers such that 
$4p=a_p^2+27b_p^2$, with 
$a_p\equiv 1 \bmod{3}$ and $b_p>0$.
Define 
\begin{equation}\label{eq:cpk}
c_p(k)=\begin{cases}
2 & \text{ if  $k^{(p-1)/3} \equiv 1 \bmod{p}$},\\
-1 & \text{ if  $k^{(p-1)/3} \not\equiv 1 \bmod{p}$}.
\end{cases}
\end{equation}
 Then 
\begin{equation}\label{eq:K3}
\sigma_p=\begin{cases}
1+\frac{3c_p(k)}{p} -\frac{a_p}{p^2} &\text{ if $p\equiv 1\bmod{3}$ and $p\nmid k$,}\\
1+ \frac{(p-1)a_p-1}{p^2} &\text{ if $p\equiv 1\bmod{3}$ and $p\mid k$.}
\end{cases}
\end{equation}
When $k$ is not cube-free it is still possible to calculate explicit expressions for $\sigma_p$, but 
we have chosen not to do so. However, if $p\nmid k$ then 
\eqref{eq:K1}, \eqref{eq:K2} and \eqref{eq:K3} remain true.

\subsection{Archimedean  density}\label{s:clip}

We begin by discussing the integral $I(t)$ in \eqref{eq:def-I}
in the special case \eqref{3cubes}.
We already  saw in Example \ref{example-3cubes} that Hypothesis \ref{hyp} holds in this case.
In the development of 
Heuristic~\ref{heur:circle-method} we introduced 
Hypothesis \ref{hyp'}, which concerns the asymptotic behaviour of the integral $J_\ell(B)$, as  defined in 
\eqref{qq}. The following result confirms this hypothesis, since $b=1$ in this setting. 

\begin{lemma}\label{lem:jiB}
Let $\ell\geq 0$ be an integer. Then
$$
J_\ell(B)=\kappa_\ell (\log B)^{\ell+1} +O((\log B)^{\ell}),
$$
where
$$
\kappa_\ell=\frac{3^{\ell}}{\ell+1} \cdot \frac{\Gamma(\frac{1}{3})^3}{\pi\sqrt{3}}.
$$
\end{lemma}

\begin{proof}
Combining  \eqref{eq:trivial} with  \eqref{eq:I(t)} in 
\eqref{qq}, we readily obtain
$$
J_\ell(B)=
\frac{\Gamma(\frac{1}{3})^3}{3\pi \sqrt{3}}
\int_{2}^\infty \frac{\cos(2\pi t/B^3) (\log t)^\ell}{t} \d t +O_\ell(1).
$$
We have $\cos(\theta)=1+O(\theta^2)$ for $|\theta|\leq 1$. Hence 
\begin{align*}
\int_{2}^{B^3/(2\pi)} \frac{\cos(2\pi t/B^3) (\log t)^\ell}{t} \d t
&=
\int_{2}^{B^3/(2\pi)} \frac{(\log t)^\ell}{t} \d t  +O_\ell((\log B)^\ell)\\
&=
\frac{3^{\ell+1}}{\ell+1} \cdot (\log B)^{\ell+1} +O_\ell((\log B)^\ell).
\end{align*}
On the other hand, we have 
\begin{align*}
\int_{B^3/(2\pi)}^\infty \frac{\cos(2\pi t/B^3) (\log t)^\ell}{t} \d t
&=
\int_{1}^\infty \frac{\cos(t)(\log (B^3t/(2\pi)))^\ell}{t} \d t  +O_\ell((\log B)^\ell)\\
&=
O_\ell((\log B)^\ell),
\end{align*}
since
$$
\int_{1}^\infty \frac{\cos(t)(\log t)^{j} }{t} \d t \ll_j 1,
$$
for any $j\geq 0$.
Putting these estimates together yields the lemma.
\end{proof}

Taking $\ell=0$, it  follows from this result that 
$$
\int_{-\infty}^\infty I(t) \d t\sim 
\frac{\Gamma(\frac{1}{3})^3}{\pi\sqrt{3}} \cdot
 \log B.
$$
On the other hand, it follows from Proposition 
\ref{prop:arch-volume-expansion} that 
$
\mu_\infty(B)\sim \mu_D \log B,
$
where
$\mu_D$ is defined in
\eqref{eq:min-stratum-integral}.
Although we omit details, one can adopt the argument in~\eqref{eq:norm-omega}  to prove that 
\begin{equation}\label{eq:table}
 \mu_D = 
 \frac{\Gamma(\frac{1}{3})^3}{\pi\sqrt{3}} .
\end{equation}
Thus  Hypothesis \ref{hyp:garlic} is also true in this case.

\subsection{Application of the heuristic}

We have already seen in 
\eqref{eq:k=1} that the surface $U$ can contain $\mathbb{A}^1$-curves over $\ZZ$, depending on the choice of $k$. (In fact, the $\mathbb{A}^1$-curves of degree at most $4$ have all been identified by Segre \cite[Thm.~XII]{segre}.)
Thus, we let
 $N_U^\circ(B)$ be the counting function 
defined in \eqref {eq:restrict-count}, where 
$U(\ZZ)^\circ$ is obtained by removing those points in $U(\ZZ)$ that are contained in any such curve. 
We are now ready to reveal what our heuristic says about 
$N_U^\circ(B)$ when  $k$ is not a cube, so that 
 $\rho_U=0$ and $b=1$ in
Heuristic~\ref{heur:pre}.  
On  applying Lemma \ref{lem:jiB}, we are therefore led to the  following expectation,
which  fully recovers Heath-Brown's heuristic~\cite{33}.

\begin{heuristic}\label{con:HB-noncube}
Let $k\in \ZZ$ be a non-cube that is not congruent to 
$ \pm 4\bmod{9}$. 
 Let $U\subset \AAA^3$ be the cubic surface defined by 
\eqref{3cubes}.
Then  
 \[
N_U^\circ(B)
\sim 
 \frac{\Gamma(\frac{1}{3})^3}{\pi\sqrt{3}} 
 \cdot \prod_p\sigma_p\cdot \log B,
  \]
  as $B\to \infty$.  Explicit expressions for $\sigma_p$ are given by 
  \eqref{eq:K1}--\eqref{eq:K3} when $k$ is cube-free. 
\end{heuristic}

When $k$ is not a cube,
the arithmetic of  $U$ has been studied by Colliot-Th\'el\`ene and Wittenberg \cite{CTW}, with the aim of understanding the effect of the Brauer group  on the integral Hasse principle. 
 In this setting, it follows from \cite[Props.~2.1 and 3.1]{CTW} that 
 $\Br(U)/\Br(\QQ)\cong \ZZ/3\ZZ$. 
 Although it is found in \cite[Thm.~4.1]{CTW} that there is no obstruction to the Hasse principle, it can certainly happen that the  Brauer group obstructs strong approximation. When $k=3$, for example, it was discovered
 by Cassels \cite{cassels} that any  point $(x,y,z)\in U(\ZZ)$ must satisfy $x\equiv y\equiv  z \bmod{9}$.
 (This is explained by the Brauer--Manin obstruction in \cite[Remark~5.7]{CTW}.)
 In general, for any cube-free $k\in \ZZ$, 
 let $\mathcal{A}$ be a non-trivial class of the Brauer group from \cite[Prop.~2.1]{CTW}. 
For any  prime $p\neq 3$  such that $v_p(k)=0$, the evaluation of $\mathcal A$ at
any local point at $p$ is equal to $0$ since both the surface and the class $\mathcal A$
have good reduction at such primes. Thus there is no obstruction to strong approximation at these primes. On the other hand, it  follows from 
 \cite[Prop.~4.6]{CTW} that 
there is an   obstruction to strong approximation at any prime $p\neq 3$ for which 
 $v_p(k)\in \{1,2\}$ and that these are the only obstructions.
It is natural to 
expect that the 
the local factors $\sigma_p$ in 
 Heuristic \ref{con:HB-noncube} 
 should be 
 modified to take into account the possible failures of strong approximation that occur when $p\mid 3k$ and we conjecture  that  Heuristic \ref{main-heur} holds with 
$\gamma_U=3$ and 
$V=(D(\RR) \times U(\A^{\mathrm{fin}}_\ZZ))^{\Br U}$, where the pairing with $\Br(U)\cong \Br(X)$ is the restriction of the usual Brauer--Manin pairing on $X(\A_\QQ)$.
 In their work \cite[Sec.~2A]{booker'},   Booker and Sutherland have provided numerical evidence 
 that the constant  in Heuristic \ref{con:HB-noncube} 
  is correct on average, and so 
  we expect that Brauer--Manin  obstruction cuts out $\frac{1}{3}$ of the adelic points, as for the case   $k=3$.

 \section{Sums of cubes: higher rank}\label{s:gen_cubes}

We proceed by  investigating $N_U^\circ(B)$ for the polynomial \eqref{3cubes} when $k$ is a cube, and  secondly, for 
the polynomial 
$f(x,y,z)=x^3+ky^3+kz^3-1$  when $k>1$ is a  square-free integer. In both cases we have $b=1$. 
We shall find that  $\rho_U=3$ in the former case and $\rho_U=2$ in the latter. 

\subsection{Representations of a cube as a sum of three cubes}\label{s:one_cube}
We begin by studying  the cubic surface $U\subset \AAA^3$ defined  by 
\eqref{3cubes} when $k$ is a cube, having already seen in   \eqref{eq:onion} that   $\rho_U=3$. 
Thus it follows from Conjecture \ref{con1} that $N_U^\circ(B)=O((\log B)^4)$.
It is natural to appeal to Theorem \ref{thm:trivial-character} in order 
to get an analogue of Heuristic \ref{con:HB-noncube}
for the case that $k$ is a cube. On 
returning to the setting of Proposition \ref{prop:local-densities-convergence-factors}, 
we have $\zeta(s,D)=\zeta(s)$. Moreover,
if
 $X\subset \PP^3$ 
 is the  compactification of $U$, then 
 it follows from 
Lemma 3.3 and Proposition 3.6 in \cite{pt2} that
$
L(s,\Pic( X_{\bar{\QQ}}))= \zeta(s)\zeta_K(s)^3,
$
where $\zeta_K(s)$ is the Dedekind zeta function associated to 
$K=\QQ(\sqrt{-3})$. But $\zeta_K(s)=\zeta(s)L(s,\chi)$, where $L(s,\chi)$ is the Dirichlet $L$-function associated to the 
real Dirichlet character 
\begin{equation}\label{eq:dirichlet}
\chi(n)=\begin{cases}
(\frac{-3}{n}) & \text{ if $3\nmid n$,}\\
0 & \text{ if $3\mid n$.}
\end{cases}
\end{equation}
It therefore follows that 
$
L(s,\Pic(X_{\bar{\QQ}}))= \zeta(s)^4L(s,\chi)^3,
$
whence
\[
  \lambda_0 = \lim_{s\to 0} s^{3}  \zeta(s+1)^3L(s+1,\chi)^3 =L(1,\chi)^3=
  \frac{\pi^3}{3^{4}\sqrt{3}},
\]
by Dirichlet's class number formula. 
Moreover, 
$$
  \lambda_p = 
  \zeta_p(1,D)L_p(1,\Pic( X_{\bar{\QQ}}))^{-1}=
  \left(1-\frac{1}{p}\right)^{3}\left(1-\frac{\chi(p)}{p}\right)^{3}.
$$
Thus 
Theorem \ref{thm:trivial-character} suggests the heuristic
$$
N_U^\circ(B)\sim 
\lambda_0 
\prod_p\lambda_p\sigma_p\cdot r(B),
$$
where
$$
r(B)=\frac{\left(\tfrac{3}{2}\right)^3 J_0(B) (\log B)^3}{3!}
-\frac{\left(\tfrac{3}{2}\right)^2 J_1(B)(\log B)^2}{2!}
+
\frac{\left(\tfrac{3}{2}\right) J_2(B)\log B}{2!}-
\frac{J_3(B)}{3!}.
$$
But  it follows from   Lemma \ref{lem:jiB} that 
$
r(B)\sim C(\log B)^4,
$
with 
$$
C=
\frac{\left(
\tfrac{3}{2}\right)^3 \kappa_0}{6}
-\frac{\left(\tfrac{3}{2}\right)^2 \kappa_1}{2}+
\frac{\left(\tfrac{3}{2}\right) \kappa_2}{2}-
\frac{\kappa_3}{6}=0.
$$
Thus we seem to run into trouble 
when  applying our circle method heuristic to this particular case.

Instead, we appeal to  Heuristic 
 \ref{main-heur}. 
When $k$ is a cube it follows from Segre \cite{segre} that the compactification $X\subset\PP^3$ is  $\QQ$-rational. In particular, the Brauer group $\Br(X)$ is trivial. Since we also have 
$\Br(U)\cong \Br(X)$, by \cite[Prop.~3.1]{CTW},
it follows that  Brauer group considerations don't  require us to make any adjustment to the leading constant. In this way, on recalling Lemma \ref{lem:jiB},  we are led to expect that 
\begin{equation}\label{eq:leek}
N_U^\circ(B)\sim 
\gamma_U
\cdot 
\frac{\pi^2\Gamma(\frac{1}{3})^3}{3^5}
\cdot 
\prod_p \left(1-\frac{1}{p}\right)^{3}\left(1-\frac{\chi(p)}{p}\right)^{3}\sigma_p \cdot (\log B)^4
 \end{equation}
for some $\gamma_U\in \QQ_{>0}$, where
     $ \sigma_p = \lim_{\ell\to\infty}p^{-2\ell}\nu(p^\ell)$.

We proceed to study this  numerically when $k=1$.
 We shall need to remove the set of integral points lying on the infinite family of $\AAA^1$-curves found by Lehmer~\cite{lehmer}. Coccia has shown that this set is thin~\cite[p.~371]{coccia}, while its complement is not~\cite[Thm.~8]{coccia}.  
We can easily get explicit expressions for the local densities $\sigma_p$ in this case.
Thus it follows from  \eqref{eq:K1} that $
\sigma_3=\nu(27)/27^2=2.
$
If $p\neq 3$ we can assess $\sigma_p$
via \eqref{eq:K2} and \eqref{eq:K3}, which  leads to the expression
\[
\sigma_p=
1 +\frac{3(1+\chi(p))}{p}-\frac{a_p(1+\chi(p))}{2p^2}.
\]
The expectation $N_U^\circ(B)\sim \gamma_U \cdot c\cdot (\log B)^4$ now   follows from \eqref{eq:leek},
 with   $$
  c=\frac{56\pi^2\Gamma(\frac{1}{3})^3}{3^{10}}
  \cdot
  \prod_{p\equiv 2\bmod{3}} 
\left(1-\frac{1}{p^2}\right)^{3}
\prod_{p\equiv 1\bmod{3}} 
\left(1-\frac{1}{p}\right)^{6}
\left(
1 +\frac{6}{p}-\frac{a_p}{p^2}\right).
$$
Evaluating the Euler product for $p\leq 10^8$ results in
$c\approx 0.0958$.

Based on his work with Booker, Sutherland has determined all integer solutions of
$x^3+y^3+z^3=1$ with 
$
  \max\{\abs{x},\abs{y},\abs{z}\}\le B_{\max}=\sqrt[3]{2}\cdot 10^{15},
$
excluding those on lines. We filtered out solutions on the first three embedded $\bA^1$-curves that were discovered by  Lehmer~\cite[Thm.~A]{lehmer}. The remaining curves have degree $\ge 22$ and contribute negligibly many points.
Let $N(B)$ denote the contribution to $N_U(B)$ from the points not on one of the three curves of lowest degree. 
We determined a least squares linear regression of $\log N(B)$ with respect to $\log \log B$. In this, as  in all regressions in this paper, the input is the unweighted set of vectors 
$(\log \log H(P), \log N(H(P))$ such that $P$ is an integral point with $\sqrt{B_{\max}}\leq H(P)\leq  B_{\max}$ . In this way, we obtain the estimate
\[
  \log N(B) \approx \sigma \log \log B + \delta,
\]
with 
$\sigma = 3.75$ and $\delta=-3.48$,  as illustrated
 in Figure~\ref{fig:sums-of-cubes}.
  This seems to be compatible with 
 \eqref{eq:leek}, which predicts  $\sigma = 4$. 
 We will take $\gamma_U = \frac{7}{72}$ in  \eqref{eq:leek}, which yields the modified constant $c’ = \frac{7}{72} \cdot c \approx 0.00931$. The estimate 
 $$
 c_{\exp} = \frac{N(B_{\max})}{(\log B_{\max})^4} \approx 0.013 
 $$
 for the leading constant is roughly four thirds the size of  this prediction, 
 though as reflected in 
 Figure~\ref{fig:sums-of-cubes-quotient}, 
  it most likely overestimates the true leading constant. In summary, the modified leading constant seems to bring the prediction closer to the actual data.

  \begin{figure}
  \begin{center}
    \includegraphics[width = .8\textwidth]{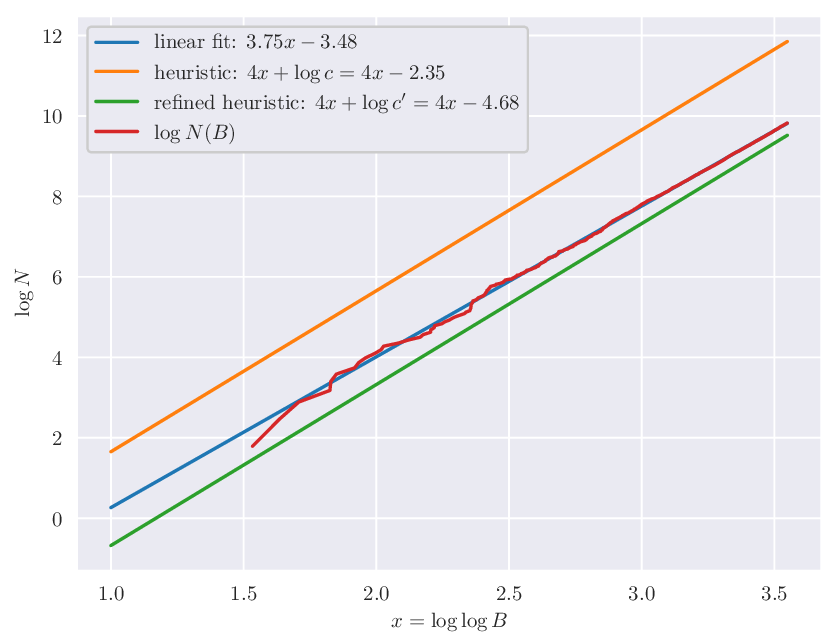}
    \caption{A comparison of $N(B)$ and a linear fit with the heuristic.}\label{fig:sums-of-cubes}
  \end{center}
\end{figure}

\begin{figure}
  \begin{center}
    \includegraphics[width = .8\textwidth]{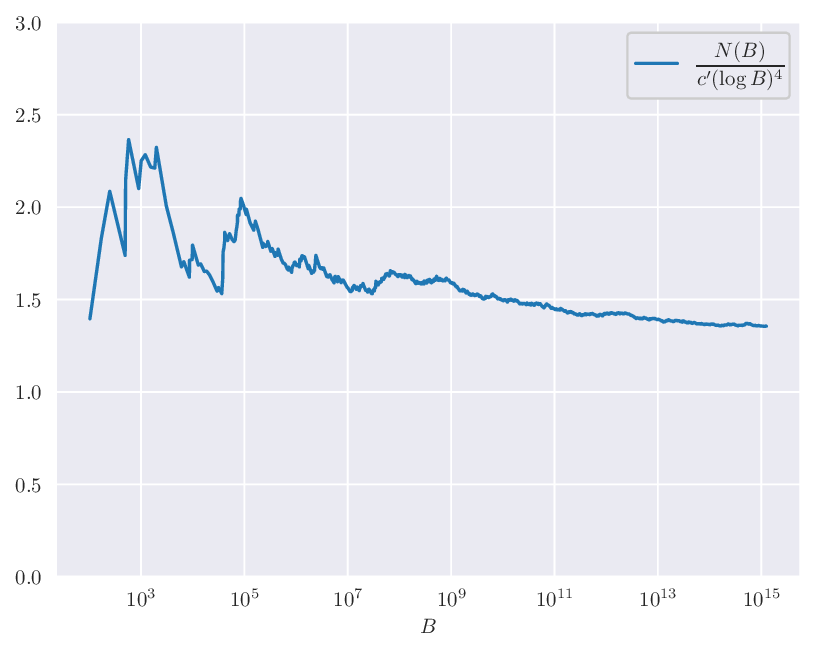}
    \caption{A comparison of $N(B)$ with $c'(\log B)^4$.}\label{fig:sums-of-cubes-quotient}
  \end{center}
\end{figure}

It remains to justify the numerical value  $\gamma_U=\frac{7}{72}$
in \eqref{eq:leek}. 
In the setting of rational points on Fano varieties, 
as  in \eqref{eq:sky}  (and further described in \cite[p.~335]{peyre-bordeaux}), 
Peyre's prediction for the leading constant involves a factor $\alpha_X$ that depends on the geometry of the effective cone $\Eff(X)\subset \Pic(X)_\RR\cong \RR^4$. Denoting by $\Eff(X)^\vee\subset\Pic(X)_\RR^*$ the dual of this cone, it can be described as an integral
\begin{equation}\label{eq:alpha-integral}
  \alpha_X =
  \frac{1}{3!}
\int_{\Eff(X)^\vee} e^{-\langle t, \omega_X^\vee\rangle},
\end{equation}
or as a volume
\begin{equation}\label{eq:alpha-volume}
  \alpha_X = \vol_H\{t\in \Eff(X)^\vee : \langle t, \omega_X^\vee\rangle = 1\},
\end{equation}
for the hyperplane volume normalised by $\omega_X$ and the Picard lattice.
More generally, 
as explained by Batyrev and Tschinkel 
\cite[Def.~2.3.13]{BatT},
for arbitrary height functions $H_\cL$ associated with a metrised line bundle $\cL$ such that $\omega_X^\vee = \cL^{\otimes a}$ is a multiple of it, the anticanonical bundle $\omega_X^\vee$ has to be replaced by $\cL$ in both formulas.

In the context of integral points, formulae such as those described by Santens~\cite[Conj.~6.6 and Thm.~6.11]{tim} and Wilsch~\cite[Sec.~2.5]{wilsch} have the feature that the effective cone appearing in \eqref{eq:alpha-integral} and \eqref{eq:alpha-volume} needs to  be replaced by that of a certain subvariety $V$ that depends on intersection properties of the boundary divisor $D$. If all components of $D$ share a real point, however, then this subvariety is simply $V=X$, by \cite[Rem.~2.2.9\,(i)]{wilsch}. Moreover, the log anticanonical bundle $\omega_X(D)^\vee$ assumes the role of the canonical bundle in this setting.
For   the Fermat cubic,  the bundle associated with the height function $H(x, y, z) = \max\{|x|, |y|, |z|,1 \}$ is  $\cO(1)\cong \omega_X^\vee$. Since  the log anticanonical bundle is its multiple $\omega_X(D)^\vee \cong \cO_X \cong \omega_X^{\otimes 0}$, it would seem natural to include the factor $\alpha_X = \frac{7}{18}$, as determined by Peyre and Tschinkel~\cite[Prop.~6.1]{pt2}.
However, one further modification seems prudent.
When $\alpha_X$ appears in its form~\eqref{eq:alpha-integral}, the factor $\frac{1}{3!}$ can be interpreted as issuing from a Tauberian theorem for a meromorphic function whose right-most pole is at $s=1$ and is of order $4$. This results  in an expected main term of order $B(\log B)^{3}$ in the Manin conjecture. If such a pole is at $s=0$, this factor  becomes $\frac{1}{4!}$ instead, and the resulting main term is $(\log B)^4$.  It therefore seems natural to believe that $\frac{\alpha_X}{4}$ appears in the leading constant for the counting function $N_U^\circ (B)$, which  leads us to the value $\gamma_U=\frac{7}{72}$.
An alternative point of view on this constant stems from 
Peyre's \emph{all the heights} philosophy \cite{beyond}.
As articulated in \cite[Qn.~4.8]{beyond}, Peyre 
asks whether   an equidistribution property holds for the logarithmic multiheight 
defined in \cite[Def.~4.4]{beyond}. 
Since 100\% of points on Fano varieties whose height is at most $B$ should have height at least $B^{1-\ve}$, 
any equidistribution phenomenon for the normalised multiheights $h(P)/\log B$ needs to be  concentrated on the hyperplane appearing in~\eqref{eq:alpha-volume}; in the setting of  logarithmic growth, this is no longer true, and it is more natural to expect equidistribution on the cone 
$\{t\in \Eff(X)^\vee : \langle t, \omega_X^\vee\rangle \le 1\}$, whose volume is $\frac{1}{4}$ times the volume~\eqref{eq:alpha-volume}.

\subsection{An example with Picard rank two}\label{s:oblong}

We now consider the  smooth cubic surface $U_k\subset \AAA^3$ defined by the polynomial 
$f(x,y,z)=x^3+ky^3+kz^3-1$,  for a square-free integer $k>1$. 
This time, we shall see that Theorem \ref{thm:trivial-character} suggests a meaningful heuristic
for $N_{U_k}^\circ(B)$.
The compactification $X_k\subset \PP^3$ is the smooth cubic surface 
$
x_0^3+ky_0^3+kz_0^3=t_0^3.
$
The geometry of $X_k$ has been studied by Peyre and Tschinkel \cite{pt2} and it  follows from 
\cite[Prop.~6.1]{pt2} that $\rho_{X_k}=3$. 
The divisor $D$ is the  smooth genus $1$ curve 
$V(x_0^3+ky_0^3+kz_0^3)$. 
In particular, we have  $b=1$  
and
$
\rho_{U_k}=\rho_{X_k}-1=2.
$
It follows from \cite[Lemme~1]{ct-k-s} that  the Brauer group $\Br(X_k)$ is trivial and from 
\cite[Thm.~1.1]{julian} that $\Br(U_k)\cong\Br(X_k)$.

\subsubsection{Local densities}

Adapting Lemma \ref{lem:jiB}, it is straightforward to prove that 
\begin{equation}\label{eq:spring}
J_\ell(B)=\kappa_\ell (\log B)^{\ell+1} +O((\log B)^{\ell}),
\end{equation}
in the notation of \eqref{qq}, 
where
$$
\kappa_\ell=\frac{3^{\ell}}{\ell+1} \cdot \frac{\Gamma(\frac{1}{3})^3}{\pi\sqrt{3}k^{2/3}}.
$$

Turning to the non-archimedean densities, we have 
$
   \sigma_p = \lim_{\ell\to\infty}p^{-2\ell}\nu(p^\ell), 
 $
with 
 $
\nu(p^\ell)=  \#\left\{(x,y,z)\in (\ZZ/p^\ell\ZZ)^3: x^3+ky^3+kz^3\equiv 1 \bmod{p^\ell}\right\}.
$
When $p\mid 3k$, the densities can be calculated using a computer, with the outcome that 
\begin{equation}\label{eq:fox1}
\sigma_p=
\begin{cases}
1 & \text{ if $p\mid k$ and $p\equiv 2 \bmod{3}$,}\\
3 & \text{ if $p\mid k$ and $p\equiv 1 \bmod{3}$,}
\end{cases}
\end{equation}
and 
\begin{equation}\label{eq:fox2}
 \sigma_3=
 \begin{cases}
 3 & \text{ if $k\equiv 0 \bmod{9}$,}\\
 2 & \text{ if $k\equiv \pm 1 \bmod{9}$,}\\
\frac{5}{3} & \text{ if  $k\equiv \pm 2 \bmod{9}$,}\\
1 & \text{ if  $k\equiv \pm 3 \bmod{9}$,}\\
\frac{4}{3} & \text{ if $k\equiv \pm 4 \bmod{9}$.}
 \end{cases}
\end{equation}
Recall that any prime $p\equiv 1\bmod 3$ 
admits a unique representation as
$4p=a_p^2+27b_p^2$, for  $a_p,b_p\in \ZZ$ such that 
$a_p\equiv 1 \bmod{3}$ and $b_p>0$.
We can then write 
 $p=\pi\bar\pi$ in $\QQ(\sqrt{-3})$, with 
 $\pi=\frac{1}{2}(a_p+3b_p \sqrt{-3})$.
 Denote by 
 $\omega=\frac{1}{2}(-1+\sqrt{-3})$,  a primitive cube root of unity.

\begin{lemma}\label{lem:sig-rank2}
Let $p\nmid 3k$. Then
$$
\sigma_p=
\begin{cases}
1 &\text{ if $p\equiv 2\bmod{3}$,}\\
1+\frac{6}{p}-\frac{a_p}{p^2}
&\text{ if $p\equiv 1\bmod{3}$ and $(\frac{k}{\pi})_3=1$,}\\
1+\frac{3}{p}+\frac{\frac{1}{2}(a_p+9b_p)}{p^2}
&\text{ if $p\equiv 1\bmod{3}$ and $(\frac{k}{\pi})_3=\omega$,}\\
1+\frac{3}{p}+\frac{\frac{1}{2}(a_p-9b_p)}{p^2}
&\text{ if $p\equiv 1\bmod{3}$ and  $(\frac{k}{\pi})_3=\omega^2$,}
\end{cases}
$$
where $(\frac{\cdot }{\pi})_3$ is the cubic residue symbol associated to $\pi$.
\end{lemma}

\begin{proof}
It follows from Hensel's lemma that
$
\sigma_p=\nu(p)/p^2.
$
We can use cubic characters to evaluate $\nu(p)$, following the approach in \cite[Rem.~4.2]{pt2} and the various identities recorded in 
\cite[Chapter~8]{ir}.
We begin by writing 
$$
\nu(p)=\sum_{\substack{\chi_i^3=1\\ i=1,2,3}} \chi_1(k^{-1})
\chi_2(k^{-1})J(\chi_1,\chi_2,\chi_3),
$$
where the sum is over all characters $\chi_i: \FF_p^*\to \CC^*$ of order dividing $3$, and 
$$
J(\chi_1,\chi_2,\chi_3)=\sum_{\substack{u,v,w\in \FF_p\\ u+v+w=1}} \chi_1(u)\chi_2(v)\chi_3(w)
$$
is the Jacobi sum. 
If $p\equiv 2\bmod{3}$ then there is only the trivial character and it follows that $\nu(p)=p^2$, which gives the result. 

Suppose next that $p\equiv 1\bmod{3}$. Then $\chi_i^3=1$ if and only if $\chi_i\in \{1, \psi,\bar\psi\}$, where $\psi(\cdot)=(\frac{\cdot }{\pi})_3$ is the cubic residue symbol associated to $\pi$.
$J(\chi_1,\chi_2,\chi_3)=0$ whenever precisely one or two of the characters 
$\chi_1,\chi_2,\chi_3$ is trivial. Hence
\begin{align*}
\nu(p)=~&p^2+
J(\psi,\psi,\bar \psi)(\psi(k^{-2})+2) 
+
J(\bar\psi,\bar\psi, \psi)(\bar\psi(k^{-2})+2) \\
&\quad +
J(\psi,\psi,\psi) \psi(k^{-2}) +
J(\bar\psi,\bar\psi,\bar\psi) \bar\psi(k^{-2}).
\end{align*}
We note that $\psi(-1)=1$ since $-1$ is a cube, and moreover 
$\psi(k^{-2})=\psi(k)$ and 
$\bar\psi(k^{-2})=\bar\psi(k)$.
On appealing to the standard formulae for Jacobi sums, we therefore find that 
$
J(\psi,\psi,\bar \psi)=\tau(\psi)\tau(\bar\psi)=p\psi(-1)=p,
$
where
$$
\tau(\psi)=\sum_{t\in \FF_p} \psi(t)e_p(t)
$$
is the Gauss sum.  Similarly, $J(\bar\psi,\bar\psi, \psi)=p$. Moreover, we have 
\begin{align*}
J(\psi,\psi,\psi)&=-\tau(\psi)^3/p=-J(\psi,\psi),\quad 
J(\bar\psi,\bar\psi,\bar\psi)&=-\tau(\bar\psi)^3/p=-J(\bar\psi,\bar\psi).
\end{align*}
Hence it follows that 
\begin{align*}
\frac{\nu(p)}{p^2}
&=1+\frac{4+c_p(k)}{p}
-\frac{2\Re \left(J(\chi,\chi)\psi(k)\right)}{p^2},
\end{align*}
in the notation of \eqref{eq:cpk}. 
Now  it follows from \cite[Prop.~8.3.4]{ir} and its corollary that 
$J(\psi,\psi)=\frac{1}{2}(a_p+3b_p)+3b_p\omega$. 
Let us write $A_k=2\Re \left(J(\psi,\psi)\psi(k)\right)$ for simplicity.
If $\psi(k)=1$ then  $A_k=2\Re J(\psi,\psi)=a_p$, as claimed in the lemma. If $\psi(k)=\omega$ then 
$A_k=2\Re (\omega J(\psi,\psi))=-\frac{1}{2}(a_p+9b_p)$. 
Finally, if 
 $\psi(k)=\omega^2$ we  get
$A_k=2\Re (\omega^2 J(\psi,\psi))=
-\frac{1}{2}(a_p-9b_p)$. 
\end{proof}

\subsubsection{Application of the heuristic}

We can adapt the parameterisation of 
Lehmer \cite{lehmer} to the present setting. On substituting $k^{\lfloor n/3\rfloor}t^n$ 
for $t^n$ in the Lehmer parametrisation, we are led to 
infinitely many $\mathbb{A}^1$-curves of increasing degree. The curves of lowest degree are
given parametrically by 
$$
x(t)=9kt^3+1, \quad \{y(t),z(t)\}=\{-9kt^4-3t, 9kt^4\},
$$
and 
\begin{align*}
x(t)&=2^43^5k^3t^9-3^2kt^3+2^33^4k^2t^6+1
,\\
\{y(t),z(t)\}&=\{-2^43^5k^3t^{10}-2^43^4k^2t^7
-3^4kt^4+3t,2^43^5k^3t^{10}-
135kt^4\}.
\end{align*}
Let
 $N_k(B)=N_{U_k}^\circ(B)$ be the counting function 
defined in \eqref {eq:restrict-count}, where 
$U_k(\ZZ)^\circ$ is obtained by removing those points in $U_k(\ZZ)$ that are contained in any such curve. 
We are now ready to reveal what our heuristic says about 
$N_k(B)$. 

We have already seen that we may take 
 $\rho_{U_k}=2$ and $b=1$ in Theorem \ref{thm:trivial-character}.
Returning to  Proposition \ref{prop:local-densities-convergence-factors}, 
we have  $\zeta(s,D)=\zeta(s)$ and it follows from 
Lemma~3.3 and Proposition 3.6 in \cite{pt2} that
$
L(s,\Pic(X_{k,\bar{\QQ}}))= \zeta(s)^2\zeta_{K}(s)L(s,\chi)^2,
$
where $\zeta_{K}(s)$ is the Dedekind zeta function associated to 
$K=\QQ(k^{1/3})$ and $L(s,\chi)$ is the Dirichlet $L$-function associated to the 
real Dirichlet character  \eqref{eq:dirichlet}.
Hence we have 
$F(s)= 
\zeta(s+1)\zeta_{K}(s+1)L(s+1,\chi)^2\tilde F(s)$ in \eqref{eq:peter}.
But then, 
on recalling Lemma~\ref{lem:F}, 
we see that 
$\tilde F(0)
=\prod_{p}\lambda_p \sigma_p$,
where
\[     
  \lambda_p = 
  \left(1-\frac{1}{p}\right)    \left(1-\frac{\chi(p)}{p}\right)^2 \zeta_{K,p}(1)^{-1}.
\]
Moroeover, 
\[
  \lambda_0 = \lim_{s\to 0} s^{2}  \zeta(s+1)\zeta_{K}(s+1)L(s+1,\chi)^2 =\frac{\pi^2}{27}\cdot \lim_{s\to 1} (s-1)\zeta_{K}(s),
\]
since $L(1,\chi)^2=\pi^2/27$. 

Next, we clearly have 
$$
r(B)=\frac{\left(\tfrac{3}{2}\right)^2 J_0(B) (\log B)^2}{2!}
-\frac{\left(\tfrac{3}{2}\right) J_1(B)\log B}{1!}
+
\frac{J_2(B)}{2!},
$$
in  Theorem \ref{thm:trivial-character}. But then \eqref{eq:spring}  yields
$
r(B)=C_\infty(\log B)^3,
$
with 
$$
C_\infty=
\frac{\left(\tfrac{3}{2}\right)^2 \kappa_0}{2!}
-\frac{\left(\tfrac{3}{2}\right) \kappa_1}{1!}+
\frac{\kappa_2}{2!}=
\frac{\sqrt{3}\Gamma(\frac{1}{3})^3}{8\pi k^{2/3}}.
$$
Similarly to  \eqref{eq:table}, we 
note that $C_\infty=\frac{3}{8}\cdot \mu_D$, where $\mu_D$ is the constant appearing in Proposition 
\ref{prop:arch-volume-expansion}, related to the real density. 
In summary, we are  led to the following expectation. 

\begin{heuristic}\label{con:higher-rank}
Let $k>1$ be a square-free integer and let 
 $U_k\subset \AAA^3$ be the cubic surface defined by 
 $x^3+ky^3+kz^3=1$.
Then  
 \[
N_k(B)
\sim c_\cir^{(k)}  (\log B)^3,
  \]
  as $B\to \infty$,
  where 
  \begin{align*}
  c_\cir^{(k)}=~&
  \frac{\pi\Gamma(\frac{1}{3})^3}{72 \sqrt{3}k^{2/3}}
\cdot  
  \lim_{s\to 1} (s-1)\zeta_{K}(s) 
\cdot 
\prod_p 
  \left(1-\frac{1}{p}\right)    \left(1-\frac{\chi(p)}{p}\right)^2 \zeta_{K,p}(1)^{-1}
\sigma_p.
\end{align*}
Explicit expressions for $\sigma_p$ are given by Lemma \ref{lem:sig-rank2} for $p\nmid 3k$, and by 
\eqref{eq:fox1}--\eqref{eq:fox2} for $p\mid 3k$.
\end{heuristic}

\subsubsection{Numerical data}

\begin{figure}
  \begin{center}
    \includegraphics[width = .8\textwidth]{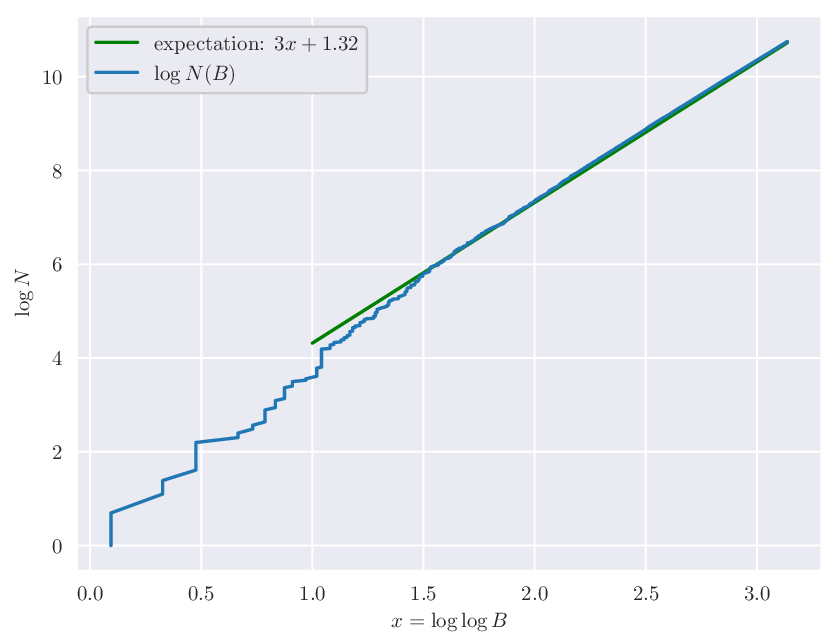}
    \caption{A comparison of $N(B)$ with the circle method prediction.
    }\label{fig:var-sums-of-cubes}
  \end{center}
\end{figure}

\begin{figure}
  \begin{center}
    \includegraphics[width = .8\textwidth]{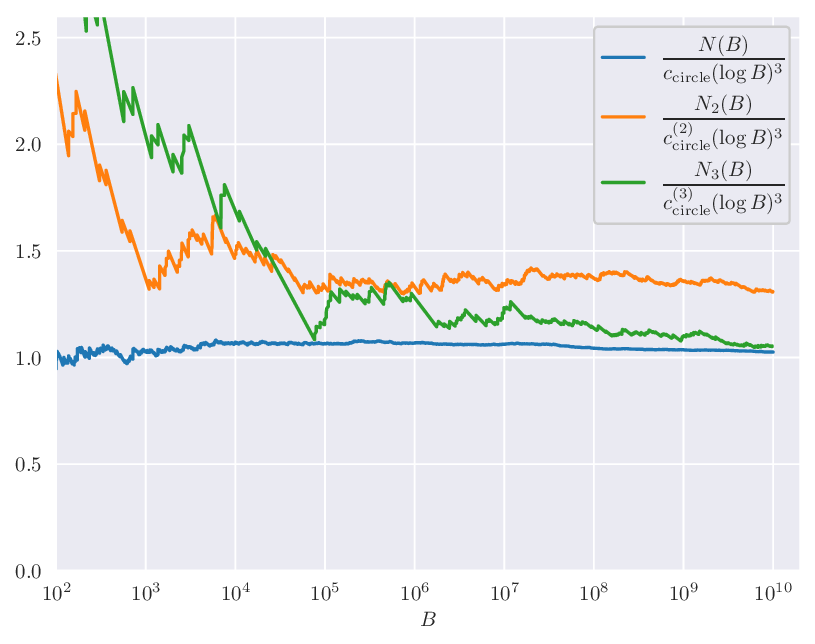}
    \caption{A comparison of $N(B)$ with $(\log B)^3$.}\label{fig:var-sums-of-cubes-quotient}
  \end{center}
\end{figure}

\begin{figure}
  \begin{center}
    \includegraphics[width = .8\textwidth]{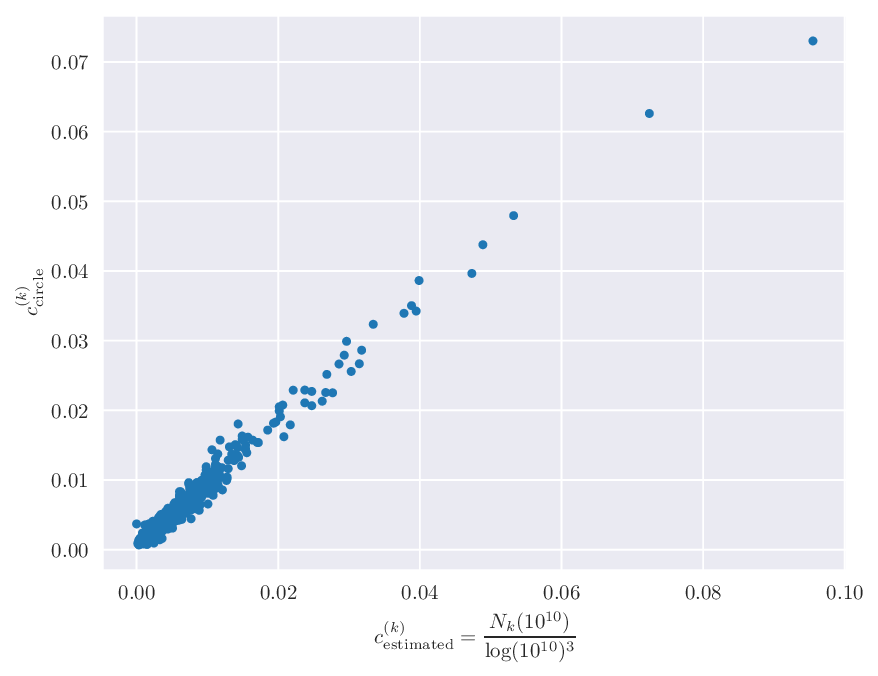}
    \caption{A scatter plot comparing 
the predicted leading constants to the heuristic leading constants determined from the data.}
      \label{fig:var-scatter}
  \end{center}
\end{figure}

We have determined all integer points $(x,y,z)\in U_k(\ZZ)$ with 
$\max\{\abs{x},\abs{y},\abs{z}\}\le 10^{10}$, for all square-free integers $2\leq k\leq 1000$.
We removed all points contained in the $\bA^1$-curves of degrees $4$ and $10$, that we identified above. The higher degree curves contribute  negligibly and the numerics don't suggest the presence of any further $\bA^1$-curves of low degree. 
 Let 
$$
N(B) = \sum_{
\substack{2\le k\le 1000\\ \text{$k$ square-free}}} N_{k}(B).
$$
The sum of the predicted constant over all relevant $k$ is
$$
c_{\text{circle}}
=\sum_{
\substack{2\le k\le 1000\\ \text{$k$ square-free}}}  c_{\text{circle}}^{(k)}
 \approx 3.73.
$$
Figure \ref{fig:var-sums-of-cubes} confirms that our prediction is very close to the numerical data. 
Moreover, 
a least squares linear regression of $\log N(B)$ against $\log \log B$ results in a fit
$$
\log N(B) \approx 3.02 \log \log B + 1.31,
$$ 
which suggests  the experimental leading constant 
  $c_{\expm} = \exp(1.31) \approx 3.71$. 
 This agrees  with Heuristic \ref{con:higher-rank}, which predicts slope $3$
 and leading 
 constant 
$c_{\text{circle}}\approx 3.73$.

In fact, as seen in 
     Figure  \ref{fig:var-sums-of-cubes-quotient}, 
both the cumulative counting function $N(B)$ as well as individual counting functions (depicted for $k=2$ and $k=3$) align rather well with  the circle method prediction for $B \le 10^{10}$. 
Moreover,  
 in Figure  \ref{fig:var-scatter} we have included a scatter plot, in which each blue dot represents a surface in the family;
    on the $x$-axis is the prediction for the constant coming from the circle method and on the $y$-axis is the ratio $N_{U_k}^\circ(B)/(\log B)^3$, for $B=10^{10}$.
    The correlation is very good. 
Note that both $C_\infty$ and the product of non-archimedean densities 
vary significantly with the parameter $k$, and the presence of both factors in $c_\text{circle}^{(k)}$ is necessary to achieve the correlation seen in Figure \ref{fig:var-scatter}. 
   Indeed, estimating 
   $$\log(c_\expm^{(k)}) \approx \log(c_{\text{circle}}^{(k)}) + \mu
   $$
   results in an $R^2$-value of $0.90$, while analogous estimates using only $C_\infty$ or $\lambda_0 \prod_p \lambda_p \sigma_p$ 
   instead of the full circle method  constant
   result in $R^2$-values of $0.40$ and $0.38$, respectively. 

Finally, we compare our findings with Heuristic \ref{main-heur}, recalling that  $\Br(U_k)$ is trivial in this case. Thus 
Heuristic \ref{main-heur} predicts that
$N_U^\circ(B)\sim 
\gamma_U\cdot \tau_{U,H}(V)
 (\log B)^3
$
for  $\gamma_U\in \QQ_{>0}$
and $V=D(\RR)\times 
U(\A^{\mathrm{fin}}_\ZZ)$.
If we take 
 $\gamma_U=\frac{3}{8}$, 
we will therefore have  $c_{\mathrm{h}}=c_{\text{circle}}^{(k)}$,
in the notation of  Heuristic \ref{con:higher-rank}, whence   Conjecture \ref{k>1}.

\section{The Baragar--Umeda examples}\label{s:eg2}

In this section we examine the surfaces appearing in Table \ref{tb:surface-params} that were
studied by Baragar and Umeda \cite{baragar}. 
In Figure \ref{fig:BU} we have plotted the  integer points of low height on the first surface in the table.
Let $U\subset \AAA^3$ be any  surface in 
Table \ref{tb:surface-params} 
 and let $X\subset \PP^3$ be its compactification. 
In particular $X$ is a clearly a smooth cubic surface. 
An analysis of the lines contained in $X$, similar to the calculations in \cite{colliot}, reveals that 
$\rho_U=0$. 
The divisor at infinity  is $D=X\setminus U$, which is equal to  $V(dx_0y_0z_0)$, a union of three lines
\begin{equation}\label{eq:triangle-names}
  L_1 = V(t_0, x_0), \quad L_2 = V(t_0,y_0), \quad L_3 = V(t_0, z_0)
\end{equation}
It follows that $\rho_U=0$ and $b=2$ in Heuristic~\ref{heur:circle-method}, and so the  exponent of $\log B$ in the heuristic agrees with the asymptotic formula  \eqref{eq:BU-asy}.

\begin{figure}
  \begin{center}
    \includegraphics[width = .8\textwidth]{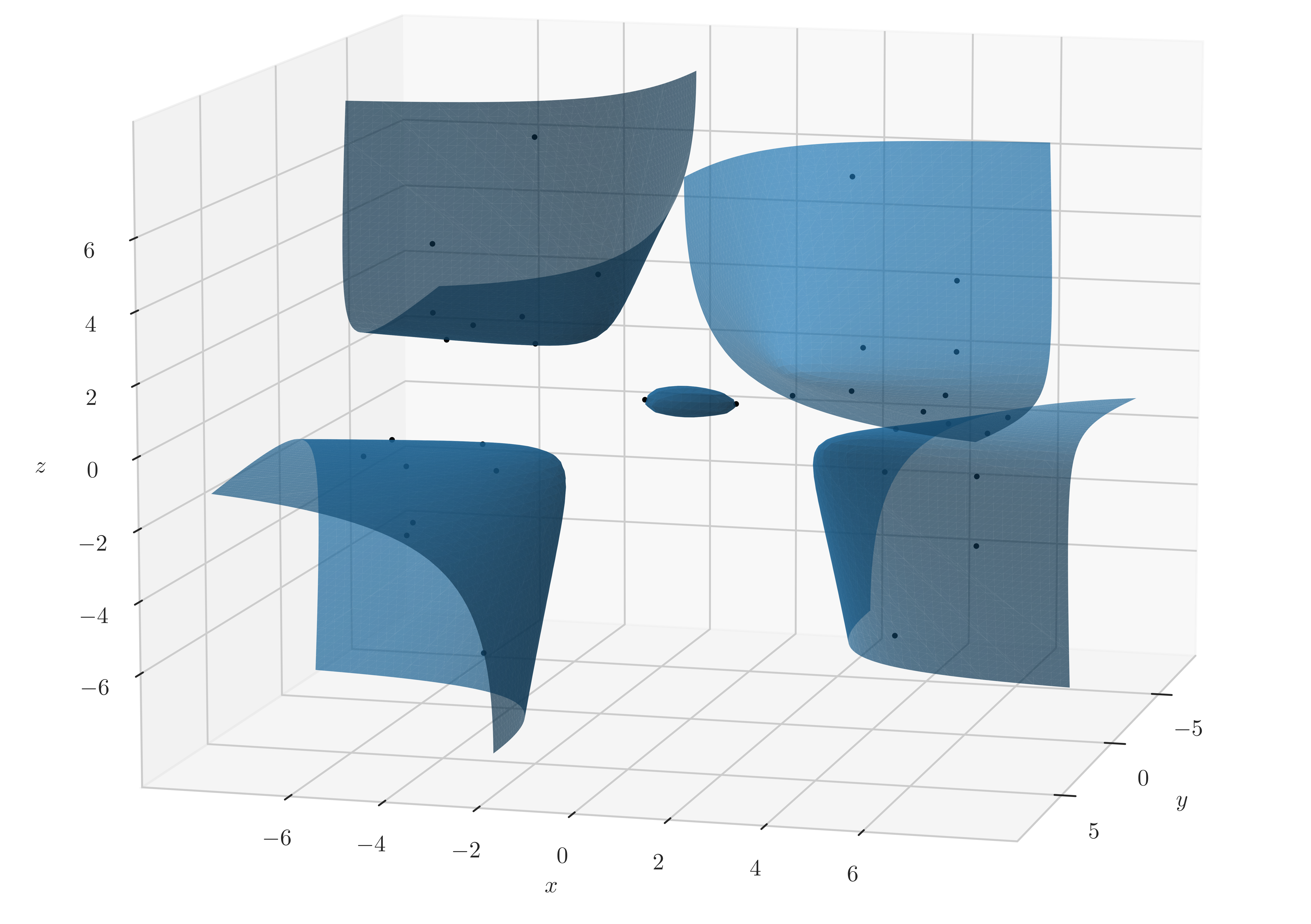}
    \caption{Integer points  on the surface $x^2+5y^2+5z^2=5xyz+1$ of height $\leq 10$. }\label{fig:BU}
  \end{center}
\end{figure}

\subsection{The leading constant}
We proceed by studying  the 
constant  in Heuristic~\ref{heur:circle-method} and comparing it to the  constant $c_\BU$ in Table~\ref{tb:surface-params}, for the different choices of coefficient vectors.
We shall find that they do not agree, even after making natural  modifications along the lines suggested in Heuristic \ref{main-heur}.

\subsubsection{The number of solutions modulo $p$ }

Henceforth we focus on the surfaces~\eqref{eq:bu-equation} for  square-free  $a,b,c,d\in \NN$ such that 
$4abc-d^2\neq 0$ and $d$ is divisible by $a,b$ and $c$. Moreover, we assume
that none of 
$d^2-4abc, a(d^2-4abc), \dots, c(d^2-4abc)$ is the square of an integer. 
These conditions are clearly satisfied by the six surfaces in Table~\ref{tb:surface-params}.
We let $S$ be the set of  prime divisors of $2abcd(d^2-4abc)$.

Let $p\not \in S$ and recall the definition~\eqref{eq:nu} of $\nu(p^k)$.
We need to calculate this quantity when $k=1$ 
While it is  possible to evaluate $\nu(p)$ using~\eqref{eq:weil}, 
we shall give an elementary treatment using character sums, based on the expression
\begin{equation}\label{eq:stepping}
\nu(p)=p^2+\frac{1}{p}\sum_{h\in \FF_p^*} \sum_{x,y,z\in \FF_p} e_p\left(hf(x,y,z)\right).
\end{equation}
We will need to  recollect some relevant facts about character sums. 
Let $A,B,C\in \ZZ$.
The quadratic Gauss sum is
\begin{equation}\label{eq:gauss}
\sum_{x\in \FF_p} e_p(Ax^2+Bx)= \ve_p \sqrt{p} \left(\frac{A}{p}\right) e_p(-\bar{4A}B^2),
\quad
\text{ if $p\nmid 2A$},
\end{equation}
where $\bar{4A}$ is the multiplicative inverse of $4A$ modulo $p$ and 
\[
\ve_p=\begin{cases}
1 & \text{ if $p\equiv 1\bmod{4}$,}\\
i & \text{ if $p\equiv 3\bmod{4}$.}
\end{cases}
\]
When $B=0$ and $p\nmid 2A$, we note that the sum on the left hand side of~\eqref{eq:gauss} can 
be written in the equivalent form
\[
\sum_{x\in \FF_p} e_p(Ax^2)=
\sum_{x\in \FF_p} \left(1+\left(\frac{x}{p}\right)\right) e_p(Ax)=
\sum_{x\in \FF_p}\left(\frac{x}{p}\right) e_p(Ax).
\]
Next, the Legendre sum is
\begin{equation}\label{eq:legendre}
\sum_{x\in \FF_p} \left(\frac{Ax^2+Bx+C}{p}\right)= 
\begin{cases}
-(\frac{A}{p}) & \text{ if $p\nmid A$ and $p\nmid B^2-4AC$,}\\
(p-1)(\frac{A}{p}) & \text{ if $p\nmid A$ and $p\mid B^2-4AC$.}
\end{cases}
\end{equation}
We are now ready to reveal our calculation of $\nu(p)$.

\begin{lemma}\label{lem:local1}
For any  $p\not \in S$, we have 
\[
\frac{\nu(p)}{p^2}= 1+\frac{1}{p}\left(\frac{d^2-4abc}{p}\right)\left(1+\left(\frac{a}{p}\right)+
\left(\frac{b}{p}\right)+\left(\frac{c}{p}\right)\right) +\frac{1}{p^2}.
\]
\end{lemma}

\begin{proof}
Recall from~\eqref{eq:bu-equation} that 
$f(x,y,z)=ax^2+by^2+cz^2-dxyz-1$.
Applying the formula~\eqref{eq:gauss} for Gauss sums in~\eqref{eq:stepping}, we deduce that 
\begin{align*}
\nu(p)
&=p^2+\frac{\ve_p}{\sqrt{p}}\sum_{h\in \FF_p^*}
 \left(\frac{ha}{p}\right)
 \sum_{y,z\in \FF_p} e_p\left(h(by^2+cz^2-1-\bar{4a} d^2y^2z^2)\right)\\
&=p^2+\frac{\ve_p}{\sqrt{p}}\sum_{h\in \FF_p^*}
 \left(\frac{ha}{p}\right)\sum_{z\in \FF_p}
 e_p\left(h(cz^2-1)\right)
 \sum_{y\in \FF_p} e_p\left(h y^2(b-\bar{4a} d^2z^2)\right).
\end{align*}
Next we  evaluate the sum over $y$. If $b-\bar{4a}d^2z^2\not\equiv  0 \bmod{p}$ then 
the inner sum is $\ve_p\sqrt{p} (\frac{h(b-\bar{4a} d^2z^2)}{p})$ by~\eqref{eq:gauss}. Alternatively,  it takes the value $p$.  Thus 
\begin{equation}\label{eq:way}
\nu(p)
=p^2+
\Sigma_1+\Sigma_2,
 \end{equation}
where 
\[
\Sigma_1=
\ve_p^2  \left(\frac{a}{p}\right)\sum_{z\in \FF_p}
\left(\frac{b-\bar{4a} d^2z^2}{p}\right)
 c_p\left(cz^2-1\right)
 \]
 and 
\[
  \Sigma_2=
  \ve_p\sqrt{p}\sum_{h\in \FF_p^*}
 \left(\frac{ha}{p}\right)\sum_{\substack{z\in \FF_p\\ 
 b-\bar{4a}d^2z^2\equiv  0 \bmod{p}}}
 e_p\left(h(cz^2-1)\right).
\]
It  follows from~\eqref{eq:ram} that 
\[
\Sigma_1=
 \ve_p^2  \left(\frac{a}{p}\right)
\left(-\sum_{z\in \FF_p}
\left(\frac{b-\bar{4a} d^2z^2}{p}\right)
+p
\left(\frac{b-\bar{4ac} d^2}{p}\right)\left(1+\left(\frac{c}{p}\right)\right)\right).
\]
We evaluate the sum over $z$ by appealing to~\eqref{eq:legendre}. This yields
\[
-\sum_{z\in \FF_p}
\left(\frac{b-\bar{4a} d^2z^2}{p}\right)
= 
\left(\frac{-\bar{4a} d^2}{p}\right) =\left(\frac{-a}{p}\right).
\]
Thus 
\begin{align*}
\Sigma_1
&=
 \ve_p^2  \left(\frac{a}{p}\right) \left(\frac{-a}{p}\right)
+p \ve_p^2 
\left(\frac{ab-\bar{4c} d^2}{p}\right)\left(1+\left(\frac{c}{p}\right)\right)
\\&=
1
+p 
\left(\frac{d^2-4abc}{p}\right)\left(1+\left(\frac{c}{p}\right)\right),
\end{align*}
since $\ve_p^2=(\frac{-1}{p})$.

Next,
we see that 
\begin{align*}
\Sigma_2
&=
\ve_p\sqrt{p}\sum_{h\in \FF_p}
 \left(\frac{ha}{p}\right)
 e_p\left(h(4abc\bar{d}^2-1)\right) \left(1+\left(\frac{ab}{p}\right)\right)\\
&=
\ve_p\sqrt{p}
 \left( \left(\frac{a}{p}\right)+\left(\frac{b}{p}\right)\right)
\sum_{h\in \FF_p}
 \left(\frac{h}{p}\right)
 e_p\left(h(4abc\bar{d}^2-1)\right).
\end{align*}
The inner sum is another Gauss sum and can be evaluated using~\eqref{eq:gauss}.
 Thus 
\[
\Sigma_2
=
 p \left(
\frac{d^2-4abc}{p}\right)
 \left( \left(\frac{a}{p}\right)+\left(\frac{b}{p}\right)\right).
\]
 Combining our expressions for $\Sigma_1$ and $\Sigma_2$ in~\eqref{eq:way}
and dividing by $p^2$, we arrive at the statement of the lemma.
\end{proof}

\subsubsection{Non-archimedean densities}

Throughout this subsection, let $k=d^2-4abc$ and let $S$ be the set of prime divisors of $2abcdk$.
It is convenient to define Dirichlet characters $\chi_1,\dots,\chi_4$ via the Kronecker symbols
$\chi_i(n) = \left(\frac{D_i}{n}\right)$,
where $D_i=\disc(K_i)$, for
$$
K_1=\QQ(\sqrt{k}), \quad K_2=\QQ(\sqrt{ka}), \quad  K_3=\QQ(\sqrt{kb}), \quad K_3=\QQ(\sqrt{kc}).
$$
In particular, we have  
$
\chi_1(p)=(\frac{d^2-4abc}{p})
$
and 
$$
\chi_2(p)=\chi_1(p)\left(\frac{a}{p}\right), \quad
\chi_3(p)=\chi_1(p)\left(\frac{b}{p}\right), \quad
\chi_4(p)=\chi_1(p)\left(\frac{c}{p}\right),
$$
for $p\not\in S$. Thus 
Lemma~\ref{lem:local1} yields
\begin{equation}\label{eq:nu-explicit-chi}
  \frac{\nu(p)}{p^2}= 1+\frac{\chi_1(p)+\chi_2(p)+\chi_3(p)+\chi_4(p)}{p}
  +\frac{1}{p^2},
\end{equation}
for any such prime.  It follows from \eqref{eq:sigmap}
that 
$$
\sigma_p(s)=
1+\frac{\chi_1(p)+\chi_2(p)+\chi_3(p)+\chi_4(p)}{p^{s+1}} +\frac{1}{p^{s+2}}.
$$

Define 
$  \lambda(s) = L(s,\chi_1)L(s,\chi_2)L(s,\chi_3)L(s,\chi_4)$
and
\[
  \lambda_p(s) = (L_p(s,\chi_1)L_p(s,\chi_2)L_p(s,\chi_3)L_p(s,\chi_4))^{-1}.
\]
With this notation we have 
\[
  F(s) = \prod_p \sigma_p(s) = \lambda(s+1) \prod_p \lambda_p(s+1)\sigma_p(s)
\]
in Lemma \ref{lem:F},
for $\Re(s)>2$.
Now 
\begin{equation}\label{eq:explicit-euler-factors}
  \begin{aligned}
  \lambda_p(s+1)\sigma_p(s)
  &=
  \left(1-\frac{\chi_1(p)}{p^{s+1}}\right)
  \dots 
  \left(1-\frac{\chi_4(p)}{p^{s+1}}\right)\\
  &\qquad \times \left(
  1+\frac{\chi_1(p)+\chi_2(p)+\chi_3(p)+\chi_4(p)}{p^{s+1}}
  +\frac{1}{p^{s+2}}\right),
  \end{aligned}
\end{equation}
for $p\not\in S$, and so
the Euler product $\prod_p \lambda_p(s+1)\sigma_p(s)$ converges absolutely for
$\Re(s)> -\frac{1}{2}$.  In particular, we have 
\begin{equation}\label{eq:F-BU}
\lim_{s\to 0} \left(s^{\rho_U} F(s)\right)=
  F(0)=\lambda(1) \prod_p \lambda_p(1)\sigma_p=
 \prod_p \sigma_p.
\end{equation}
This expression could also have been deduced from  
Proposition \ref{prop:local-densities-convergence-factors}, but we have chosen 
to present an explicit derivation using  Dirichlet $L$-functions.

\subsubsection{The expected leading constant}

We are now ready to record an explicit expression for the expected leading constant 
$c_{\cir}$, say,
in Heuristic \ref{heur:circle-method}, with $\rho_U=0$ and $b=2$. 
Combining Lemma \ref{lem:arch-bu}
with \eqref{eq:F-BU}, it follows that  
\begin{equation}\label{eq:tea}
  c_{\cir} = \frac{6}{d} \cdot \lambda(1) \cdot c_S \cdot c_{S^c},
\end{equation}
with 
\[
  c_S = \prod_{p\in S} \lambda_p(1) \sigma_p \qquad\text{and}\qquad
  c_{S^c} = \prod_{p\not\in S} \lambda_p(1) \sigma_p.
\]
We determine $\lambda(1)$ using Dirichlet's class number formula, $c_S$ by a computer search for points modulo small powers of $p\in S$, and $c_{S^c}$ by multiplying the factors
$\lambda_p(1) \sigma_p$ for $p<10^7$. (Note that the latter are obtain by taking 
$s=0$ in \eqref{eq:explicit-euler-factors}.) 
The results of these computations are summarised in Table~\ref{tb:expected-const}.

\begin{table}[t]
  \tabcolsep1em
  \def\fracrowsep{.5em}
  \begin{tabular}{c@{\hspace{2\tabcolsep}}ll}\toprule
          & $c_{\cir}$ & $c_{\cir}/c_{\mathrm{BU}}$\\\midrule
    (i)   & 2.997816 & 0.5734700 \\ 
    (ii)  & 2.997094 & 1.0107957 \\ 
    (iii) & 1.484675 & 0.6015930  \\ 
    (iv)  & 2.397675 & 0.5910831  \\ 
    (v)   & 1.16853  & 0.4686900\\ 
    (vi)  & 3.331807 & 0.6770839  \\
        \bottomrule
  \end{tabular}
  \caption{The circle method prediction and a comparison to the actual leading constant.}\label{tb:expected-const}
\end{table}

\subsection{Modified expectations}\label{ssec:modified-expectations}

For each of the surfaces in Table~\ref{tb:surface-params}, we note that $U(\RR)$ has five connected components: one bounded component and four unbounded ones.
This is illustrated in Figure \ref{fig:BU} for the first surface in the table. 
On the unbounded components, we have $xyz>0$, and the four components can be distinguished by imposing conditions on the signs of the variables that are compatible with this observation. Denote by $U_0$ the unbounded component with $x,y,z>0$. Due to the symmetry of the equation, it suffices to study this component.

\subsubsection{Hensel's lemma and the place $2$}
While not a failure of strong approximation, we make  the following observation.
\begin{lemma}\label{lem:hensel-failure}
Let $U\subset \AAA^3$ be one of the surfaces in Table \ref{tb:surface-params}. Then
the map
\[
  \mU(\ZZ_2)\to \mU(\ZZ/2^k\ZZ)
\]
is not surjective for any $k$. Indeed, its image consists of half the points 
in $\mU(\ZZ/2^k\ZZ)$  if $k\ge 3$. 
\end{lemma}
\begin{proof}
Let $k\ge 2$, and
let $(x,y,z)$ be a solution modulo $2^{k}$. Then all eight
points of the form $(x+\delta_1 2^{k-1}, y+\delta_2 2^{k-1}, z+\delta_3 2^{k-1})$ with $\delta_i\in\{0,1\}$ are solutions modulo $2^k$. Indeed, changing $x$ by $2^{k-1}$ results in
\begin{equation}\label{eq:modification-x-mod-2^k-1}
  \begin{aligned}
  f(x+2^{k-1}, y, z) &= a(x+2^{k-1})^2 + by^2 +cz^2 - d(x+2^{k-1})yz -1 \\
  & = f(x,y,z) + 2^k ax + 2^{2k-2}a + d2^{k-1}yz.
\end{aligned}
\end{equation}
Clearly, $2^k$ divides $2^k ax + 2^{2k-2}a$. In case (i), precisely two of $x$, $y$, and $z$ are even so that $2\mid yz$, while in the remaining cases, $d$ is even, so that $2^k\mid d2^{k-1}yz$ in any case. Hence, $f(x+2^{k-1}, y, z)\equiv f(x,y,z) \equiv 0 \bmod{2^k}$, and modifications of $y$ or $z$ can be treated analogously. 

Let the parameters $a,b,c,d$ be as in case (i) for now. If $(x,y,z)$ is a solution modulo $2^{k}$, then precisely one of $x,y,z$ is odd, say $x$ (the other two cases are analogous). We note that $4\mid yz$ by the assumption on the parities of the coordinates, while $2^{k+1}\mid 2^{2k-2}$ by the assumption on $k$, so that~\eqref{eq:modification-x-mod-2^k-1} implies that
\[
  f(x+2^{k-1},y,z) - f(x,y,z) \equiv 2^k ax \equiv 2^k \bmod{2^{k+1}},
\]
noting that both $a$ and $x$ are odd by assumption for the second equivalence. Using that $f(x,y,z)\equiv f(x +2^{k-1}, y, z) \equiv 0 \bmod {2^k}$, this implies that precisely one of $f(x,y,z)$ and $f(x +2^{k-1}, y, z)$ vanishes modulo $2^{k+1}$. (And then $f$ also vanishes modulo $2^{k+1}$ on the other seven points coinciding with this one modulo $2^k$, but on none of the points coinciding with the other one modulo $2^k$.)

The remaining cases can be dealt with similarly, using that precisely one of $x$ and $y$ is odd in case (ii), that $y$ is always odd in case (iii), that $z$ is always odd in cases (iv) and (v), and that $x$ is always odd in case (vi).
\end{proof}
\begin{remark}\label{rmk:double-volume}
As a consequence  of this and by \cite[Ch.~II, Lem.~6.6]{J}, the Tamagawa volume of each residue disc in $\mU(\ZZ_2)$ modulo $2^k$ is $2^{1-2k}$.
(Note that this makes Lemma~\ref{lem:hensel-failure} compatible with~\cite[Lem.~1.8.1]{borovoi-rudnick}.)
Thus, whenever we  count points in the image $\mU(\ZZ)\to \mU(\ZZ/m\ZZ)$ with $2\mid m$, we shall multiply the result by $2$ when using it as part of our modified leading constant.
\end{remark}

On the other hand, for odd primes in $S$, with the help of a computer, we find that each point $P$ modulo $p$ lifts to a point modulo $p^2$ and 
the $p$-adic norm of a least one of the partial derivatives is at least
$p^{-1}$ at $P$. Hence, Hensel's lemma implies that all points modulo odd primes lift to $p$-adic points and~\eqref{eq:sigmap} holds for all odd places.

\subsubsection{Failures of strong approximation}
As usual let $U\subset \AAA^3$ be one of the Baragar--Umeda surfaces \eqref{eq:bu-equation}
and let $\mathfrak{U}$ be its integral model over $U$.
If $a$ is a square modulo $p$, then there are obvious solutions $(\pm 1/\sqrt{a},0,0)\in \mU(\ZZ_p)$ modulo $p$. However, the group $\Gamma$ acts trivially on these, meaning that they only lift to the trivial solutions $(\pm 1,0,0)\in\mU(\ZZ)$ if $a\in \{\pm 1\}$ is an integral square, or not at all if it is not. (For instance, in case (i), there is a solution $(1,0,0)\in\mU(\FF_p)$ for all primes $p$ which lifts only to $(1,0,0)$, while $(0,3,0)\in \mU(\FF_{11})$ does not lift at all.) In the light of these observations, we are led to set
\[
  \mU(\ZZ/p^k\ZZ)' = \left\{P \in \mU(\ZZ/p^k\ZZ) : 
  \begin{array}{l}
  P\not\equiv \{(\alpha, 0, 0),\ (0,\beta,0), (0, 0, \gamma)\} \bmod p \\ 
  \text{for any $\alpha, \beta, \gamma\in\ZZ/p\ZZ$ }
  \end{array}
  \right\}
\]
for odd primes $p$ and $k\ge 1$, and
\[
  \mU(\ZZ/2^k\ZZ)' = \left\{P\in \mU(\ZZ/2^k\ZZ) :
  \begin{array}{l}
P\in \im(\mU(\ZZ_2)\to \mU(\ZZ/2^k\ZZ)),\\
P\not\equiv \{(\alpha, 0, 0), (0,\beta, 0), (0,0,\gamma)\} \bmod{8}\\
  \text{for any $\alpha,\ \beta,\gamma\in\ZZ/8\ZZ$}
  \end{array}
  \right\}
\]
for $k\ge 3$.

For any integer $m>0$, the description of $\mU(\ZZ) \cap U_0$ as the orbit of one or more primitive solutions under the group generated by the Vieta involutions allows us to efficiently compute the image of 
\begin{equation}\label{eq:def-phi}
\phi_m\colon \mU(\ZZ) \cap U_0 \to \mU(\ZZ/m\ZZ).
\end{equation}
Although we omit the details it is possible to extend  the work of Colliot-Th\'el\`ene--Wei--Xu~\cite{colliot} and  Loughran--Mitankin \cite{LM}, in order to study the Brauer--Manin obstruction for the Baragar--Umeda surfaces. In the case of the  surface (i), for example, one can check that  the Brauer--Manin obstruction 
precisely cuts out this image for $m= 2^3\cdot 3 \cdot 5$;  in other words, 
\[
  \im \phi_{2^3 \cdot 3\cdot 5} = \left(\mU(\ZZ/8\ZZ)' \times \mU(\ZZ/3\ZZ)' \times \mU(\ZZ/5\ZZ)'\right)^{\Br U}.
\]
(This  makes sense, since  the pairing is constant over all places different from $\infty, 2, 3$ and $5$, so that the set cut out does not depend on choices of points over the remaining primes.)
Motivated by this, 
for any of the surfaces in Table \ref{tb:surface-params}, 
we define 
\begin{equation}\label{eq:MS}
m_S=\prod_{p\in S} p^{k_p},
\quad \text{ with }
k_p = \begin{cases}
3 &\text{ if $p=2$,}\\
1 &\text{ if $p\in S\setminus \{2\}$}.
\end{cases}
\end{equation}
We can then prove the following facts about  $\im \phi_m$, for various choices of $m\in \NN$.

\begin{proposition}\label{prop:strong-approx-computations}
  Let $I_m = \im \phi_m$ for $m\in\NN$. Then
  \begin{enumerate}
      \item $I_p = \mU(\FF_p)'$ if $p\not\in S$ and  $p\leq 1000$, provided  $U$ is not as in case (ii) or (iv);
      \item $I_{pq}=I_p\times I_q$ if $p,q\le 73$ are distinct primes and $p\not\in S$; 
      \item $I_{pql} = I_{pq}\times I_l$, up to reordering of $p,q,l$, if $p,q,l\le 23$ are distinct primes;
      \item $I_{m} = I_{m/m_S} \times I_{m_S}$ for $m=2^3\cdot3\cdots 11$; and
      \item $\#I_{m\cdot m_S} = m^2 \# I_{m_S}$, where $m=\prod_{p\in S} p$. 
  \end{enumerate}
\end{proposition}
\begin{proof}
  These equalities are established  by determining the respective orbits using a computer. More precisely, for $U$ as in case (ii), the first equality fails for  $p\equiv \pm 1 \bmod{24}$, and for $U$ as in case (iv), the behaviour seems to depend on $p$ modulo $120$.
  The second computation reveals failures of strong approximation for precisely one pair $(p,q)$ with $p,q\in S$ in all cases except (iv), similar to the one in case (i) that is explained by the Brauer--Manin obstruction.\end{proof}

We expect that the failures of strong approximation encountered in the numerical analysis of part (2) of 
Proposition \ref{prop:strong-approx-computations} are  all  explained by the Brauer--Manin obstruction.

\subsubsection{The modified constant}

Based on our observations in the previous section, we 
propose modifying our constant along the lines of Heuristic \ref{main-heur}.
Let 
$P_1, P_2, P_3$ be the three vertices of the triangle at infinity.
Let $m_S$ be defined by \eqref{eq:MS} and recall the definition \eqref{eq:def-phi} of the map $\phi_m$, for any $m\in \NN$.
We apply Heuristic \ref{main-heur} with the set 
$$
V_0 = \{P_1,P_2,P_3\} \times \pi_{m_S}^{-1} \left(\im \phi_{m_S}\right)
\times \prod_{p\not\in S} \mU(\FF_p)',
$$ 
where  $\pi_k\colon \prod_{p\mid k} \mU(\ZZ_p) \to \mU(\ZZ/k\ZZ)$ is the reduction modulo $k$, for 
any  $k\in \NN$. 
Note that taking a different unbounded component $U_i$ to $U_0$ would give a different set $V_i$ of equal volume. We do not take the union, however, since the set $V_0$ only approaches each vertex of the triangle at infinity from one of the four possible directions, so 
in fact the resulting volume would be  $\frac{1}{4} \tau_{U,H}(V_0\cup V_1\cup V_2\cup V_3) = \tau_{U,H}(V_0)$.

We proceed to calculate  the value of $\tau_{U,H}(V_0)$.
Let 
\[
  c_S'=2 \frac{\#\im \left(\mU(\ZZ) \cap U_0 \to \mU(\ZZ/m_S\ZZ)\right)}{m_S^2} \prod_{p\in S} \lambda_p(1),
\]
noting that 
the leading $2$ is a consequence of Remark~\ref{rmk:double-volume}. 
For $p\not \in S$, we set
\begin{equation}\label{eq:sigma-modified-for-sa}
  \sigma_p' = \frac{\# \mU(\FF_p)' }{p^2}.
\end{equation}
Explicitly, 
on modifying~\eqref{eq:nu-explicit-chi} to remove the solutions $(\pm \sqrt{a},0,0)$, etc., if they exist, we
find that 
\[\sigma_p'=
  1+\frac{\chi_1(p)+\chi_2(p)+\chi_3(p)+\chi_4(p)}{p} 
    - \left(2 + \left(\frac{a}{p}\right) + \left(\frac{b}{p}\right) + \left(\frac{c}{p}\right)  \right)\frac{1}{p^2}.
\]
We set
\[
  c_{S^c}' = \prod_{p\not\in S} \lambda_p(1)\sigma_p'.
\]
Then we are led to modify the circle method constant in \eqref{eq:tea}
to 
$$
c_{\cir}' = \frac{6}{d}\cdot \lambda(1)\cdot c_{S}' \cdot c_{S^c}'.
$$
Numerical approximations of these new constants and a comparison to the constants in Table \ref{tb:surface-params} are recorded in Table~\ref{tb:modified-expected-const}. (It is interesting to note that our modified circle method constant is always smaller than the actual constant.)

Recalling that not all points counted in~\eqref{eq:sigma-modified-for-sa} lift to $\ZZ$-points in cases (ii) and (iv), we further set
\[
  c_{S^c}'' = \prod_{p\not\in S} \lambda_p(1) \frac{\#\im(\mU(\ZZ)\cap U_0\to \mU(\FF_p))}{p^2}
\]
and arrive at the modified constant
$$
  c''_{\cir} = \frac{6}{d}\cdot \lambda(1)\cdot c_{S}' \cdot c_{S^c}'',
$$
by computing the images for primes $p<10^3$. It follows from Proposition~\ref{prop:strong-approx-computations} that this modification does not make a difference in cases (i),(iii),(v) and (vi),  
except for a reduction in the bound for $p$ that we can use to numerically calculate it.

\begin{table}[t]
  \tabcolsep1em
  \def\fracrowsep{.5em}
  \begin{tabular}{c@{\hskip 3\tabcolsep}llll}\toprule
             & $c_{\cir}'$ & $c''_{\cir}$ & $c_{\cir}' / c_{\mathrm{BU}}$ & $c_{\cir}'' / c_{\mathrm{BU}}$ \\\midrule
       (i)   & 0.8127795 &      & 0.1554814 & \\
       (ii)  & 0.6682904 & 0.63 & 0.2253867 & 0.21 \\
       (iii) & 0.5012050 &      & 0.2030892 & \\
       (iv)  & 1.038439  & 0.51 & 0.2559995 & 0.13 \\
       (v)   & 0.4472312 &      & 0.1793816 & \\
       (vi)  & 0.7655632 &      & 0.1555764 & \\\bottomrule
  \end{tabular}
  \caption{The modified expected constant and a comparison to the actual leading constant. }\label{tb:modified-expected-const}
\end{table}

It is interesting to speculate on the constant $\gamma_U\in \QQ_{>0}$ in 
Heuristic \ref{main-heur}, led by the situation 
\eqref{eq:sky} for rational points on Fano varieties. 
An integral variant of the $\alpha$-constant has been described by Wilsch \cite[Def.~2.2.8]{wilsch} for split log Fano varieties, but this rational number is the same 
for the six surfaces  considered here,
since the relevant cones are all isometric.
Turning to the $\beta$-constant, it is possible to expand on the arguments in 
\cite{colliot} to deduce that the algebraic part of the Brauer group up to constants has order $8$ in cases (i) and (vi), and order $4$ in the remaining cases.
We note that the quotients $c'_{\cir}/c_{\mathrm{BU}}$ are not integers, nor are they rational numbers of small height. Thus Table~\ref{tb:modified-expected-const} does not seem to be compatible with a version of Heuristic \ref{main-heur} with $\gamma_U$ of small height.

\subsection{Equidistribution}\label{s:equi}
As a consequence of the failures of strong approximation, the equidistribution property
also fails. However, we can still ask about equidistribution to the uniform probability measure on the image of the map $\phi_m$ in \eqref{eq:def-phi}.
In other words, we can ask whether a variant of the \emph{relative Hardy--Littlewood property} holds, as 
defined by Borovoi and Rudnick \cite[Def.~2.3]{borovoi-rudnick},  with respect to the density function
$\delta_{\overline{\mU(\ZZ)}}$, 
which is 
the indicator function of the closure of  $\mU(\ZZ)$
in the adelic points $\mU(\A_{\ZZ, \text{fin}}) = \prod_p\mU(\ZZ_p)$. 
In fact, the relative Hardy--Littlewood property fails: there are infinitely many places at which strong approximation fails, and so 
$\delta_{\overline{\mU(\ZZ)}}$
is not locally constant. However, it is still measurable, and it is natural to investigate this weaker property.

\begin{figure}
  \begin{center}
    \includegraphics[width = \textwidth]{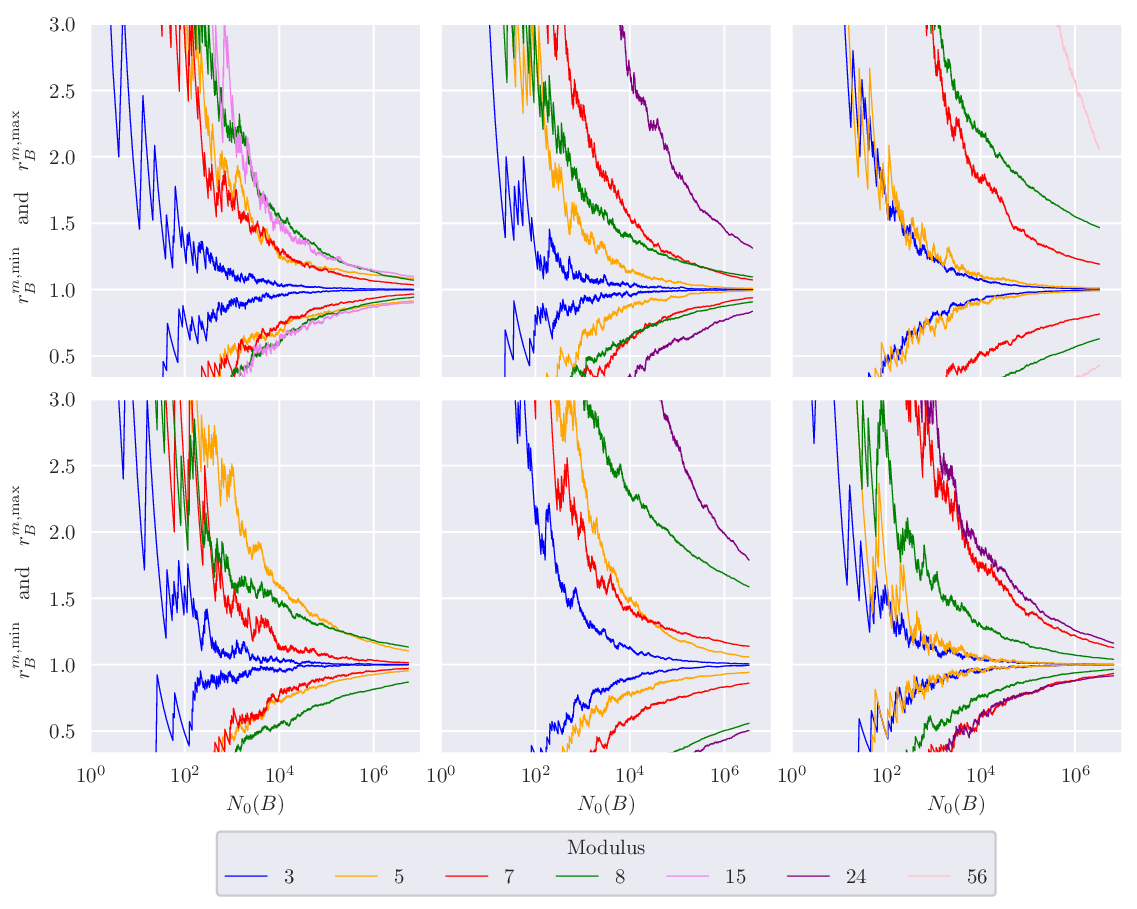}
    \caption{A comparison of the maximal and minimal observed frequencies of reductions modulo $m$ with the expected frequency.}\label{fig:equidistribution-grid}
  \end{center}
\end{figure}

We numerically tested equidistribution of integral points modulo $m$, where $m\in\{8,3,5,7\}$. This set always includes all primes in $S$ and at least one place not in $S$. In cases (ii), (v) and (vi), there is a failure of strong approximation simultaneously involving the primes $2$ and $3$; in case (i), there is a failure involving $3$ and $5$; in case (iii), there is one involving $2$ and $7$. To test for joint equidistribution modulo these primes, we have thus added $m=24$, $15$ and $56$, respectively. In each  case, we computed the set of integral points of height at most $B$ for $B\le 10^{1000}$, which can  be done efficiently using the Vieta involutions,  resulting in between $3\cdot 10^6$ and $7\cdot 10^6$ points. For $m$ as before and $P_m\in \im\phi_m$, we computed the frequencies
\[
  p^{(P_m)}_{B} = \frac{\#\{P\in \mU(\ZZ)\cap U_0 : H(P) \le B,\ P\equiv P_m \bmod{m}\}}{\#\{P\in \mU(\ZZ)\cap U_0 : H(P) \le B\}}.
\]
Equidistribution modulo $m$ means that $p^{(P_m)}_{B}\to 1/k$ as $B\to\infty$, where $k=\#\im\phi_m$ is the 
number of points modulo $m$ that lift to $\ZZ$. We thus determined 
\[
  r_B^{m,\max} = k \max_{P_m \in \im\phi_m } p^{(P_m)}_{B} \qquad \text{and} \qquad r_B^{m,\min} = k \min_{P_m \in \im\phi_m } p^{(P_m)}_{B},
\]
expecting that both quantities converge to $1$. The results are recorded in Figure~\ref{fig:equidistribution-grid}. As $\#\im\phi_m$ grows like $m^2$, we expect our order statistics to converge more slowly for larger values of $m$. With that in mind, our results seem compatible with equidistribution, even though we note that the distributions modulo $8$ in cases (iii) and (v) are outliers.

\section{The Markoff surface}\label{s:markoff}
The Markoff surface  is defined by the cubic equation~\eqref{eq:markoff} and 
has an $\mathbf{A}_1$-singularity at $(0,0,0)$. Over the reals, this singularity is an isolated point, while the remaining four connected components are smooth. Again, let $U_0$ be the unbounded component on which $x,y,z>0$.

Let $\tX\to X$ be a minimal desingularisation
and  $E$ its exceptional divisor.
 Let $\rho\colon \mtX \to \mX$ be a model, 
 let $\mathfrak E $ be the closure of $E$,
  let $\tU = \rho^{-1} U$, and let $\mtU = \rho^{-1}\mU$.
We note that the singular point $(0,0,0)$ is invariant under the Vieta involutions, both as an integral point and as an $\FF_p$-point. It follows that any integral point on $\mU$ or $\mtU$ that reduces to $(0,0,0)\in \mU(\FF_p)$ must be $(0,0,0)$ or lie above $(0,0,0)$, respectively.

\subsection{Non-archimedean local densities}
The 
local densities, adjusted as in Section~\ref{ssec:modified-expectations},  coincide for the Markoff surface and its minimal desingularisation. More precisely,
we note that for all primes, including $2$ and $3$, the point $(0,0,0)$ is the only singular point in $\mU(\FF_p)$.
In the light of this, we set
\[
  \mU(\FF_p)' = 
   \tilde \mU(\FF_p) \setminus \mathfrak{E}(\FF_p)
  \cong \{P\in \mtU(\FF_p) : \rho(P)\ne (0,0,0) \};
\]
this set contains the image of the reduction map $\mU(\ZZ)
\cap U_0 \to \mU(\FF_p)$ and only consists of smooth points.
For $p=2$, we computed the image of $\phi_{2^k}\colon \mU(\ZZ)\cap U_0 \to \mU(\ZZ/2^k\ZZ)$ by the same method as in Section~\ref{ssec:modified-expectations}.
For $2\le k\le 10$, it consists of one fourth of the points in
\begin{equation}\label{eq:points-mod-2k}
  \{P\in \mU(\ZZ/2^k\ZZ) : P \not\equiv (0,0,0) \bmod 2\}.
\end{equation}
In contrast to the observation in Lemma~\ref{lem:hensel-failure}, this is not a consequence of a failure of Hensel's lemma, as all points in the set~\eqref{eq:points-mod-2k} are smooth. Hence, we set
\[
  \sigma_2 = \frac{\#\im\phi_{4}}{4^{2}},
\]
without any of the normalisations  described in Remark~\ref{rmk:double-volume}, and compute this to be 
$
  \sigma_2 =1/4.
$
Computing $\im\phi_m$ for $m$ as in Proposition~\ref{prop:strong-approx-computations}
does not reveal
any further failures of strong approximation. In fact, it follows from
 recent work of Chen \cite[Thm.~5.58]{chen} that the same is true when $m$ is a product of primes, with each prime larger than some absolute constant. 
We are  therefore led to set
\[
  \sigma_p = \frac{\# \mU(\FF_p)'}{p^2},
\]
for odd primes. It follows from \cite[Lem.~6.4]{sarnak} (with $\alpha=3$ and $\beta=0$) that 
$$
\sigma_p=\begin{cases}
\frac{8}{9} &\text{ if $p=3$,}\\
1+\frac{3\chi(p) }{p}
 &\text{ if $p>3$,}
\end{cases}
$$
where $\chi(p) = (\frac{-1}{p})$. 
We have $\rho_U=0$. Moreover, setting $\lambda_p = L_p(1,\chi)^{-3}$ clearly makes $\prod \lambda_p\sigma_p$
absolutely convergent. Letting $\lambda_0 =L(1,\chi)^{3}$, 
the analytic class number formula yields $\lambda_0=
 \pi^3/2^6$.

\begin{remark}
This passage between points on $\mU$ and a desingularisation only works because of the exclusion of the singular point $(0,0,0)$ modulo all primes. Its preimage on $\tX$ is a $(-2)$-curve $E$ and geometrically isomorphic to $\PP^1$. The ranks of $\Pic \tX$ and $\Pic \tU$ increase by one, so that $\rk\Pic \tU=1$.  As $E$ splits over almost all primes $p$, the naïve local densities on $\tU$ would become $1 + \frac{1+3\chi(p)}{p}$ over these primes. In particular, $F_{\tU}(s)$ would have a pole of order $1$ at $s=0$. A  similar heuristic for this desingularisation would thus predict a growth rate of $(\log B)^3$, which is larger than the $(\log B)^2$ obtained by Zagier \cite{zagier}.
Only by modifying  the local densities to account for failures of strong approximation, can we  remove this pole and return the expected order of growth to $(\log B)^2$.
\end{remark}

\subsection{Archimedean local densities}\label{s:sub-markoff}
As $\rho$ is crepant and an isomorphism above the boundary, we have $\cO_{\tX}(\rho^{-1}(D_1+D_2+D_3)) \cong \rho^*\omega_X \cong \omega_{\tX}$. Moreover, it is an isomorphism above the unbounded real components. Hence, arguing similarly to Lemma~\ref{lem:arch-bu} (with $d=3$), we have 
$c_\infty=2$ in 
\eqref{eq:arch-constant-abstract}.

\subsection{Conclusion}
In summary, 
Proposition~\ref{prop:local-densities-convergence-factors} and 
Heuristic \ref{heur:circle-method}  leave us with the prediction
$
  N_U(B) \sim c_\cir (\log B)^2,
$
as $B\to \infty$, where
\[
  c_{\cir} = 2\lambda_0 \prod_p \lambda_p\sigma_p
  = \frac{4 \pi^3}{3^5} \prod_{p>3} \left(1-\frac{\chi(p)}{p}\right)^3\left(1+\frac{3\chi(p)}{p}\right)
    .
\]
We computed the Euler product for $p<10^8$ and compared this constant
with the constant 
$0.180717104712$ obtained by 
Zagier \cite{zagier}. (Note that, as pointed out in 
\cite[p.\,481]{baragar}, 
there is a typo in his paper.) Moreover, 
Zagier 
counts all ordered, positive Markoff triples and so his  constant 
has to be multiplied by $24$ to account for symmetries and signs before comparing it to our expectations.  This is summarised in Table \ref{tb:expected-const-zagier}.
We observe that the results are off by factors in a similar range to those present  in 
Table \ref{tb:modified-expected-const}.

\begin{table}[t]
  \tabcolsep1em
  \def\fracrowsep{.5em}
  \begin{tabular}{ll}\toprule
    $c_{\cir}$ & $c_{\exp}/c_{\mathrm{Zagier}}$\\\midrule
    1.256791 & 0.2897693 \\ 
    \bottomrule
  \end{tabular}
  \caption{The circle method prediction for the Markoff surface and a comparison to the actual leading constant as determined by Zagier.}\label{tb:expected-const-zagier}
\end{table}

\section{Further examples}\label{s:higher-picard}

\subsection{A question posed by Harpaz}\label{s:final}
In~\cite[Qn.~4.4]{harpaz}, Harpaz asks about the number of integral points of bounded height on 
the surfaces $U_k\subset \AAA^3$ defined by the
cubic polynomial $f(x,y,z)=(x^2-ky^2)z-y+1$, 
for a square-free integer  $k>1$. 
It will be useful to recall the construction of Harpaz' compactification, which is based on the map 
  $U_k \to \PP^2$ given by 
  $(x, y, z)\mapsto (x:y:1)$.
This map factors through the blow $X$ of $\PP^2$ in the two points $(\pm\sqrt{k}:1:1)$. 
Let $D_1 = V(z)$, $D_2=V(x-\sqrt{k}y)$, and $D_3=V(x+\sqrt{k}y)$. Then $D= D_1+D_2+D_3$ is defined over $\QQ$, and $U_k$ is isomorphic to $X\setminus D$.

Harpaz proves 
in \cite[Prop.~4.3]{harpaz}  
that $U_k(\ZZ)$ is Zariski dense 
whenever the real quadratic field $K = \QQ(\sqrt{k})$ has class number one and is such that the reduction map  $\fo_K^\times \to (\fo_K/\mathfrak p)^\times$ is surjective for infinitely many prime ideals $\mathfrak p$ of degree $1$ over $\QQ$. 
Moreover, 
the surface $U_k$ is smooth and admits a log K3 structure by~\cite[Ex.~2.13]{harpaz}, 
and furthermore,  its compactification is a del Pezzo surface of degree $7$ having  geometric Picard rank $3$. Since the boundary is a triangle of three lines whose divisor classes are linearly independent, so it follows that  the geometric Picard group of $U_k$ is trivial.  In particular, we have $\rho_{U_k}=0$
and $\Br_1(U_k)/\Br(\QQ)=0$.
Moreover, 
note that the components of $D$ intersect pairwise in a real point, so that $b=2$. 
It now follows from 
Conjecture \ref{con1} that 
$
N^\circ_{U_k}(B) =O\left((\log B)^2\right),
$  
where the implied constant depends on $k$. 

We claim that the only  $\AAA^1$-curve over $\ZZ$ is the line $z=y-1=0$. Suppose for a contradiction that $z\neq 0$ and that $U_k$ contains the $\AAA^1$-curve
$$
x=a_0t^k+\cdots +a_k, \quad y=b_0t^k+\cdots+b_k, \quad 
z=c_0t^l+\cdots+c_l,
$$ 
with integer coefficients such that  $\max\{|a_0|,|b_0|\}\neq 0$ and $c_0\neq 0$. 
Comparing coefficients of $t^{2k+l}$ yields $(a_0^2-kb_0^2)c_0=0$, which 
implies that $a_0=b_0=0$, since $k$ is square-free. This is a contradiction and so
$U_k^\circ$ is obtained by removing the line  $z=y-1=0$. 
 Heuristic \ref{heur:circle-method} then  gives 
   \begin{equation}\label{eq:carrot}
   N_{U_k}^\circ(B)\sim 
   c_\infty \prod_p \sigma_p \cdot 
   (\log B)^{2},
   \end{equation}
 where $c_\infty$ is the leading constant in 
Proposition \ref{prop:arch-volume-expansion} and $\sigma_p = \lim_{k\to\infty}p^{-2k}\nu(p^k)$, in the notation of \eqref{eq:nu}.

  \subsubsection{Real density}

In this section we give a direct estimate for the real density
$\mu_{\infty}(B)$, as defined in 
\eqref{eq:infinity}, as $B\to \infty$.
However, it turns out that  there is an analytic obstruction to the Zariski density of integral points near certain faces of the Clemens complex of a desingularisation of 
 the compactification of $U_k$. The outcome of this is that we should redefine 
 $\mu_\infty(B)$ to involve only  $(x,y,z)\in U_k(\RR)$ for which
\begin{equation}\label{eq:min-max}
\min\{|x- \sqrt{k}y|, |x+ \sqrt{k}y|\}< 1<
\max\{|x- \sqrt{k}y|, |x+ \sqrt{k}y|\},
\end{equation}
and we redefine $c_\infty$ to be the leading constant in the asymptotic formula for this modified real density. 
To check this it is convenient to 
make  the change of variables $u=x+\sqrt{k}y$ and $v=x-\sqrt{k}y$. 
If $\max\{|u|, |v|\}<1$, then $|y| = |u-v|/2\sqrt{k}\leq 1/\sqrt{k}$, leaving only the trivial solutions with $y=0$. If 
$\min\{|u|, |v|\}>1$, on the other hand, then $|z| = |(y-1)/uv| \ll  \max(|u|,|v|)/|uv| <1$, leaving only the non-dense set of solutions with small $|z|$.

\begin{lemma}\label{lem:fox4}
We may take $c_\infty= \frac{4}{\sqrt{k}}$.
\end{lemma}

\begin{proof}
Using the Leray form to calculate the real density, it readily follows that 
$$
\mu_{\infty}(B)=
\int_{\mathcal{R}} \frac{\d x \d y}{|x^2-ky^2|},
$$
where $\mathcal{R}\subset\RR^2$ is cut out by the inequalities
$|x|,|y|\leq B$ and $|y-1|\leq B |x^2-ky^2|$, together with \eqref{eq:min-max}.
Making the change of variables $u=x+\sqrt{k}y$ and $v=x-\sqrt{k}y$, we obtain 
$$
\mu_{\infty}(B)=\frac{1}{2\sqrt{k}}
\int_{\mathcal{S}} \frac{\d x \d y}{|uv|},
$$
where now  $\mathcal{S}\subset\RR^2$ is cut out by the inequalities
$$
|u+v|\leq 2B, \quad |u-v|\leq 2\sqrt{k}B, \quad |u-v-2\sqrt{k}|\leq 2\sqrt{k}B |uv|,
$$
together with
$
\min\{|u|, |v|\}< 1< 
\max\{|u|, |v|\}$.

Summing over the possible signs of $u$ and $v$, we deduce that 
$$
\mu_{\infty}(B)=\frac{1}{2\sqrt{k}}\sum_{\ve_1,\ve_2\in \{\pm 1\}}
\int_0^{2(1+\sqrt{k})B}\int_0^{2(1+\sqrt{k})B} \frac{\1_\mathcal{S}(\ve_1 u ,\ve_2v)}{uv} \d u \d v.
$$
We isolate two subregions $\mathcal{S}_1\sqcup \mathcal{S}_2\subset 
\mathcal{S}$. Let  $A>0$ be a large parameter which doesn't depend on $B$ and define
$
\mathcal{S}_1=\left(\frac{A}{ B}, \frac{1}{A}\right)\times \left(A,\frac{B}{A}\right)$ and $
\mathcal{S}_2=
\left(A,\frac{B}{A}\right) \times 
\left(\frac{A}{ B}, \frac{1}{A}\right).
$
Taking $\1_\mathcal{S}(\ve_1 u ,\ve_2v)\leq 1$, 
the overall contribution to $\mu_{\infty}(B)$ from 
$$
(u,v)\in 
[0, 2(1+\sqrt{k})B]^2
\setminus 
\mathcal{S}_1\sqcup \mathcal{S}_2
$$
is  readily found to be $O(\log B)$, where the implied constant is allowed to depend on $A$ and $k$.
Taking $A$ sufficiently large, we clearly have 
$\1_\mathcal{S}(\ve_1 u ,\ve_2v)= 1$ whenever 
$(u,v)\in 
\mathcal{S}_1\sqcup \mathcal{S}_2$.
Hence
$$
\mu_{\infty}(B)=\frac{2}{\sqrt{k}}\sum_{i\in \{1,2\}}
\iint_{\mathcal{S}_i}
\frac{\d u \d v}{uv} =
\frac{4}{\sqrt{k}}(\log B)^2+O(\log B),
$$
with an implied constant that 
 depends on $A$ and $k$. 
\end{proof}

\subsubsection{Non-archimedean densities }

\begin{lemma}\label{lem:local10}
Let $p$ be a prime. Then 
$$
\sigma_p=
\begin{cases}
1-\frac{1}{p^2} & \text{ if $p>2$ and  $(\frac{k}{p})=-1$,}\\
1+\frac{1}{p^2} & \text{ if $p>2$ and $(\frac{k}{p})=+1$,}\\
1 & \text{ if $p\mid 2k$.}
\end{cases}
$$
\end{lemma}

\begin{proof}
Let $\nu(p)$ be the number  of zeros of $f$ over 
$\FF_p$. It follows from Hensel's lemma that $\sigma_p=\nu(p)/p$.
Applying \eqref{eq:stepping}, we deduce that 
\begin{align*}
\nu(p)
&=p^2+\frac{1}{p}\sum_{h\in \FF_p^*}e_p(h)
\sum_{x,y\in \FF_p} e_p(-hy)
\sum_{z\in \FF_p} e_p\left(h(x^2-ky^2)z\right)\\
&=p^2+\sum_{h\in \FF_p^*}e_p(h) U_p(h),
\end{align*}
by orthogonality of characters,
where 
$$
 U_p(h)
=
\sum_{
\substack{
x,y\in \FF_p\\ x^2=ky^2}} e_p(-hy).
$$
 Suppose first that $p\nmid 2k$.
 If $(\frac{k}{p})=-1$ then $U_p(h)=1$, since only $(x,y)=(0,0)$ can occur. On the other hand, if 
$(\frac{k}{p})=+1$, then 
$$
U_p(h)=1+
\sum_{\substack{\eta\in \FF_p\\ \eta^2=k}}
\sum_{
\substack{
x,y\in \FF_p^*\\ x=\eta y}} e_p(-hy)=1+2\sum_{
y\in \FF_p^*} e_p(-hy)=-1,
$$
since $h\in \FF_p^*$. Suppose next that $p=2$. Then 
$$
 U_2(h)
=
\sum_{
\substack{
x,y\in \FF_2\\ x=ky}} e_2(y)=0.
$$
Finally, we suppose that $p>2$ and $p\mid k$. In this case 
$$
U_p(h)
=
\sum_{
\substack{
y\in \FF_p}} e_p(-hy)=0.
$$
The lemma follows on putting these together and evaluating the sum over 
$h$. 
\end{proof}

\subsubsection{Numerical data}

Combining Lemmas \ref{lem:fox4} and \ref{lem:local10} in \eqref{eq:carrot}, our heuristic leads us to expect that 
$
N_{U_k}^\circ(B)\sim 
c_\cir^{(k)} (\log B)$, with 
$$
c_\cir^{(k)}=
\frac{4}{\sqrt{k}}  \prod_{p\nmid 2k} \left(1+\frac{(\frac{k}{p})}{p^2}\right).
$$ 
\begin{figure}
  \begin{center}
    \includegraphics[width = .8\textwidth]{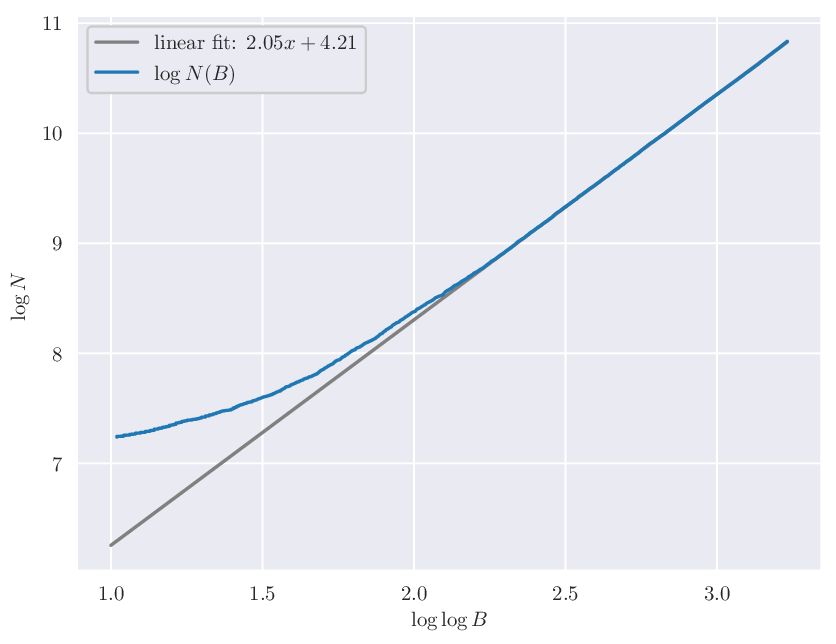}
    \caption{The number of points on $U_k$ for square-free $k\in [2,1000]$ and a linear fit.}\label{fig:ii}
  \end{center}
\end{figure}
\begin{figure}
  \begin{center}
    \includegraphics[width = .8\textwidth]{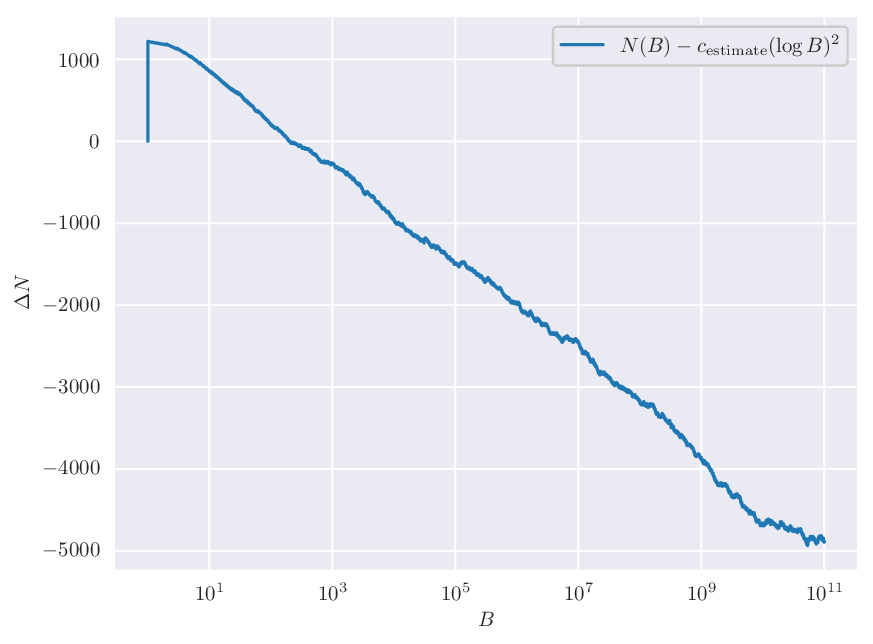}
    \caption{Comparison of $N(B)$ with the  prediction.}\label{fig:i}
  \end{center}
\end{figure}
\begin{figure}
  \begin{center}
    \includegraphics[width = .8\textwidth]{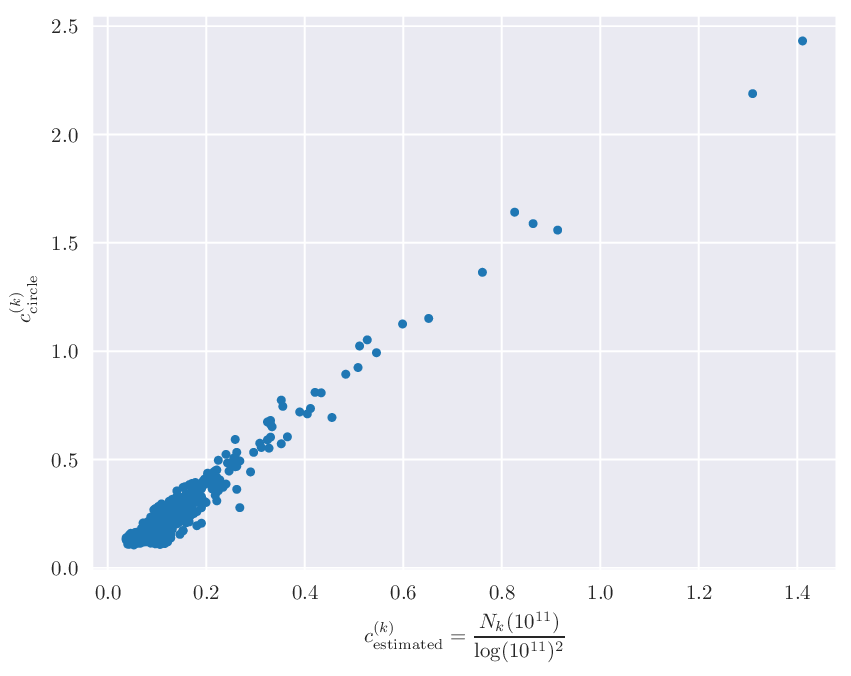}
    \caption{
    A scatter plot comparing 
the predicted leading constants to the heuristic leading constants determined from the data.}\label{fig:iii}
  \end{center}
\end{figure}
We computed integral points of height at most $10^{11}$ on $U_k$ for all square-free integers $k \in [2,1000]$.  
Let
$$
N(B) = \sum_{
\substack{2\le k\le 1000\\ \text{$k$ square-free}}} N_{U_k}^\circ(B).
$$
The sum of the predicted constant over all relevant $k$ is
$$
c_{\text{circle}}=\sum_{
\substack{2\le k\le 1000\\ \text{$k$ square-free}}} 
c_\cir^{(k)} \approx 148.8. 
$$
A linear regression of $\log N(B)$ against $\log \log B$,  as in the previous sections, provides evidence for the exponent $2$ of $\log B$ (Figure \ref{fig:ii}). Based on this, a polynomial regression of degree $2$ suggests a behaviour $N(B) = c_\text{estimate} (\log B)^2 +O(\log B)$, where $c_\text{estimate} = 87$. 
Note that $c_\text{estimate}/c_{\text{circle}} \approx \frac{3}{5}$, but we can offer no 
explanation for this disparity.  This is consistent with taking 
$\gamma_U= \frac{3}{5}$ and
$V=D(\RR)\times 
U(\A^{\mathrm{fin}}_\ZZ)$ in
Heuristic \ref{main-heur}. 
 In Figure~\ref{fig:i} we have plotted the difference $N(B)-c_\text{estimate} (\log B)^2$,  for $B\leq 10^{11}$, which looks convincingly linear in $\log B$.
Finally, in Figure \ref{fig:iii} we have included a scatter plot, in which   
each blue dot represents a surface in the family; on the $x$-axis is an estimated leading constant and on the $y$-axis is the circle method prediction for the leading constant associated to that particular surface.   The correlation is rather good and
and a similar calculation to that recorded at the end of  Section \ref{s:oblong} results in $R^2 = 0.84$.
This further illustrates that 
$\gamma_U\approx  \frac{3}{5}$ is an appropriate value in 
Heuristic \ref{main-heur}.

\subsection{An example with  higher Picard rank}\label{s:outlier}
Finally, we  compare Conjecture \ref{con1} with numerical data for a smooth affine cubic surface of the shape
$$
  (ax+1)(bx+1) + (cy+1)(dy+1) = xyz,
$$
for $a,b,c,d\in \ZZ$. Such a surface $U=U_{a,b,c,d}\subset \AAA^3$ is smooth if $(a-b)(c-d)\ne 0$ and none of $a$, $b$, $c$ or $d$ are $\pm 1$. 
Let $X$ be the completion of $U$ in $\PP^3$, with homogeneous coordinates $t_0$, $x_0$, $y_0$, $z_0$, as in Section~\ref{s:eg2}. The divisor at infinity is again a union of three lines $L_1$, $L_2$, and $L_3$ defined as in \eqref{eq:triangle-names}. In particular, $b=2$ in 
Conjecture \ref{con1}.

Next, we note that 
 the point $Q=(0:0:0:1)$ is an $\singA_2$-singularity. 
Let $\tX$ be a minimal desingularisation. This is  a weak del Pezzo surface of degree $3$ and so it has geometric Picard rank $7$. 
As illustrated in Figure~\ref{fig:pentagon}, 
the triangle at infinity becomes a pentagon on $\tX$, formed by the strict transforms of $L_1$, $L_2$, and $L_3$ (still denoted by the same names) and two $(-2)$-curves $E$. 
The projection away from $Q$ induces a morphism $\tX\to \PP^2$. This morphism is a blow-up of six points, two sets of three on a line, as in Figure~\ref{fig:point-conf}. 
All negative curves are rational, and those making up $\tD$ are linearly independent, whence $\rho_U=2$.  Moreover, this description of $\tX$ as a blow-up  shows that $E_1+E_2+L_1+L_2+L_3$ has anticanonical class in the Picard group, and so $U$ is log K3.
Finally, since the five negative curves making up $\tD$ are linearly independent in $\Pic (\tilde X) = \Pic (\tilde X_{\bar{\QQ}})$, the subvariety $U$ does not have invertible regular functions, whence $\Br (U)/\Br(\QQ) \cong H^1(\QQ,
\Pic(U_\QQbar)) = 0$.

\begin{figure}[t]
  \begin{minipage}{0.39\textwidth}
  \begin{center}
    \begin{tikzpicture}
      \node (P1) at (0,1) [draw, shape=circle, fill=black, scale=.3]{};
      \node (P2) at (-1.5,0) [draw, shape=circle, fill=black, scale=.3]{};
      \node (P5) at (1.5,0) [draw, shape=circle, fill=black, scale=.3]{};
      \node (P3) at (-1,-1.5) [draw, shape=circle, fill=black, scale=.3]{};
      \node (P4) at (1,-1.5) [draw, shape=circle, fill=black, scale=.3]{};

      \draw[thick] (P1) -- (P2) node[pos=0.5, anchor=south east]{$E_1$};
      \draw[thick] (P2) -- (P3) node[pos=0.5, anchor=east]{$L_1$};
      \draw[thick] (P3) -- (P4) node[pos=0.5, anchor=north]{$L_3$};
      \draw[thick] (P4) -- (P5) node[pos=0.5, anchor=west]{$L_2$};
      \draw[thick] (P5) -- (P1) node[pos=0.5, anchor=south west]{$E_2$};
    \end{tikzpicture}
    \caption{Pentagon at infinity.}\label{fig:pentagon}
  \end{center}
\end{minipage}
\hfill
\begin{minipage}{0.6\textwidth}
  \begin{center}
    \begin{tikzpicture}
      \node (P4) at (-2,0) [draw, shape=circle, fill=black, scale=.3]{};
      \node (P5) at (0,0) [draw, shape=circle, fill=black, scale=.3]{};
      \node (P6) at (2,0) [draw, shape=circle, fill=black, scale=.3]{};

      \node (P1) at (-2,1) [draw, shape=circle, fill=black, scale=.3]{};
      \node (P2) at (0,1) [draw, shape=circle, fill=black, scale=.3]{};
      \node (P3) at (2,1) [draw, shape=circle, fill=black, scale=.3]{};

      \node at (P1) [anchor=south west]{$P_1=\pi(L_1)$};
      \node at (P4) [anchor=north west]{$P_4=\pi(L_2)$};

      \draw[thick] (-3,1) -- (3,1) node[pos=0, anchor=east]{$E_1$};
      \draw[thick] (-3,0) -- (3,0) node[pos=0, anchor=east]{$E_2$};
      \draw[dashed] (-2,2) -- (-2,-1) node[pos=1, anchor=east]{$\pi(L_3)$};
    \end{tikzpicture}
    \caption{Configuration of blown up points, with the images of the pentagon at infinity labeled.}\label{fig:point-conf}
  \end{center}
  \end{minipage}
\end{figure}

\begin{figure}
  \begin{center}
    \includegraphics[width = .8\textwidth]{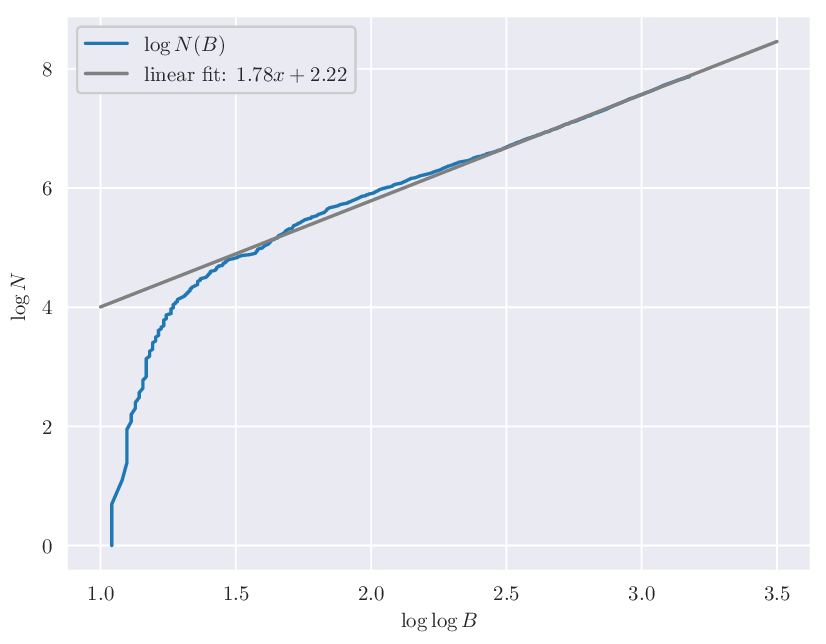}
    \caption{The number of points on $U_{(-2, 3, -3, 5)}$ of height at most $B$ and a linear fit.}\label{fig:harpaz_variant}
  \end{center}
\end{figure}

It follows from Conjecture \ref{con1}
that 
$
N^\circ_U(B)=O\left((\log B)^4\right),
$
and we proceed to investigate numerically the predicted power of $\log B$. While there is an obstruction as in~\cite[Thm.~2.4.1(i)]{wilsch}, it only affects three of the minimal strata of $\tD$, namely those lying above $Q$. Thus this obstruction  does not change the predicted order of growth, but merely  the leading constant.
Let $(a, b, c, d)  = (-2,3,-3,5)$. 
Computing all integral $\bA^1$-curves of degree at most 8, we found curves of degrees $1$, $2$, $3$, and $4$. We computed the set of integral points of height at most $2.5\cdot 10^{10}$ on the surface   and filtered out those on the $\bA^1$-curves that we found. 
We ran a least squares estimate to compare $\log N(B)$ and $\log \log B$, where $N(B)=N^\circ_U(B)$, 
and we plotted the result in  Figure \ref{fig:harpaz_variant}.
This results  in an empirical exponent of $\log B$ of $1.78$, which is much less than 
the prediction of $4$.

\section{Conclusion}\label{s:conclusion}

We end this article by summarising the numerical data that we have gathered. 
All surfaces that we studied are compatible with Conjecture \ref{con1}. 
In fact, apart from the individual surface studied in 
Section \ref{s:outlier}, all of the examples  seem to have an associated counting function 
$N_U^\circ(B)$ that  behaves asymptotically like 
$c (\log B)^{\rho_U+b}$, for suitable $c>0$, 
where $\rho_U$ and $b$ are defined in Conjecture \ref{con1}.  In the case of the singular Markoff surface studied in Section \ref{s:markoff}, this was only true after modifying the heuristic to account for defects of strong approximation: these defects are big enough to not just make the product of $p$-adic densities smaller in size, but they also affect its convergence properties. It would be interesting to know whether a similar phenomenon, or perhaps the presence of large lower order terms,  can explain the 
disparity observed in Section \ref{s:outlier} regarding  the exponent of $\log B$.

Turning to the leading constant $c$, the  results of our investigation are more mixed. While our heuristic specialises to Heath-Brown's  conjecture \cite{33} for sums of three cubes, for which Booker and Sutherland \cite{booker'} have provided evidence on average, 
in Section \ref{s:one_cube} 
we supplied  a  prediction for 
the surface $x^3+y^3+z^3=1$
with less 
compelling numerical data.  On the other hand, 
the  circle method heuristic 
aligned very well with numerical data for the 
the surfaces $x^3+ky^3+kz^3=1$
in Section
\ref{s:oblong}.  Moreover,
in this case,
 we noted that the circle method heuristic is equivalent to 
allowing for an explicit low height rational number $\gamma_U$ in Heuristic \ref{main-heur}.
For the surfaces
$(x^2-ky^2)z=y-1$ in Section \ref{s:final}, we saw that the circle method heuristic only 
agrees with the numerical data when adjoining a suitable $\gamma_U$-factor, as in 
Heuristic~\ref{main-heur}.
While in Section 
\ref{s:final}, this correlation is almost exclusively explained by the dependence of the archimedean volume on the parameter $k$, in Section 
\ref{s:oblong} it is both the Euler product and the archimedean volume that depend highly on the parameter $k$.

Finally, for the Markoff surface and its variants studied in Sections 
\ref{s:eg2} and 
\ref{s:markoff}, the circle method prediction became systematically too small after accounting for failures of strong approximation, and there is no obvious 
choice of $\gamma_U$-factor  explaining the discrepancies. We suspect that the presence of a  group action (generated by the Vieta involutions) makes these surfaces incompatible with the circle method.

In summary, we suspect  that Heuristic \ref{main-heur} is true for most cubic surfaces and that 
there are only finitely possibilities for the $\gamma_U$-factor in the moduli space of all affine cubic 
surfaces over $\QQ$. 
It would be very interesting to find a geometric interpretation for its value.

\end{document}